\documentclass{article}

\usepackage{PRIMEarxiv}

\usepackage[utf8]{inputenc} %
\usepackage[T1]{fontenc}    %
\usepackage{hyperref}       %
\usepackage{url}            %
\usepackage{booktabs}       %
\usepackage{amsfonts}       %
\usepackage{amsmath}
\usepackage{amsthm}
\usepackage{amssymb}
\usepackage{nicefrac}       %
\usepackage{microtype}      %
\usepackage{lipsum}
\usepackage{fancyhdr}       %
\usepackage{graphicx}       %
\usepackage{natbib}
\usepackage{cleveref}
\usepackage{xcolor}
\graphicspath{{media/}}     %

\usepackage{pict2e}
\makeatletter
\DeclareFontFamily{U}{matha}{\hyphenchar\font45}
\DeclareFontShape{U}{matha}{m}{n}{
  <-5.5> matha5
  <5.5-6.5> matha6
  <6.5-7.5> matha7
  <7.5-8.5> matha8
  <8.5-9.5> matha9
  <9.5-11> matha10
  <11->   matha12
}{}
\DeclareSymbolFont{matha}{U}{matha}{m}{n}

\makeatletter
\DeclareRobustCommand{\onontimes}{%
  \mathbin{\mathpalette\on@ntimes\relax}%
}
\newcommand{\on@ntimes}[2]{%
  \vcenter{\hbox{%
    \sbox0{\m@th$#1\otimes$}%
    \setlength\unitlength{\wd0}%
    \begin{picture}(1,1)
    \linethickness{0.45pt}
    \put(.5,.5){\circle{.81}}
    \end{picture}%
  }}%
}
\makeatother

\theoremstyle{plain}

\newtheorem{theorem}{Theorem}[section]
\newtheorem{lemma}[theorem]{Lemma}

\newtheorem{proposition}[theorem]{Proposition}

\theoremstyle{definition}
\newtheorem{definition}[theorem]{Definition}
\newtheorem{assumption}{Assumption}
\newtheorem{example}{Example}
\newtheorem{remark}{Remark}

\crefname{assumption}{assumption}{assumptions}
\Crefname{assumption}{Assumption}{Assumptions} 
\def\R{\mathbb{R}} %
\def\Q{\mathbb{Q}} %
\def\N{\mathbb{N}} %
\def\eps{\varepsilon}

\DeclareMathOperator{\tr}{Tr}
\DeclareMathOperator{\Emp}{Emp}
\newcommand{\E}{\mathbb{E}} %
\newcommand{\Pmhat}{\widehat{P}}
\renewcommand{\Pr}{\mathbb{P}}

\def\sA{\mathcal{A}}
\def\sD{\mathcal{D}}\def\sF{\mathcal{F}}
\def\sG{\mathcal{G}}\def\sH{\mathcal{H}}

\def\sN{\mathcal{N}}
\def\sP{\mathcal{P}}

\def\sX{\mathcal{X}}

\def\bX{\textbf{X}}

	\title{Improved Concentration for Mean Estimators via Shrinkage}
\author{
  \href{https://orcid.org/0009-0004-7132-8481}{\includegraphics[scale=0.06]{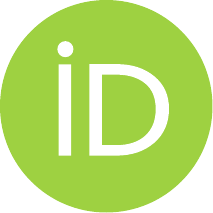}}\hspace{1mm}Antônio Catão \\
  IMPA \\
  \texttt{antonio.catao@impa.br} \\
   \And
  \href{https://orcid.org/0000-0001-6188-299X}{\includegraphics[scale=0.06]{orcid.pdf}}\hspace{1mm}Lucas Resende \\
  ENSAE \\
  \texttt{lucas.resende@ensae.fr} \\
   \And
  \href{https://orcid.org/0000-0003-0907-704X}{\includegraphics[scale=0.06]{orcid.pdf}}\hspace{1mm}Paulo Orenstein \\
  IMPA \\
  \texttt{pauloo@impa.br} \\
}
\begin{document}
    \maketitle

	\begin{abstract}
		We study a class of robust mean estimators $\widehat{\mu}$ obtained by adaptively shrinking the weights of sample points far from a base estimator $\widehat{\kappa}$. Given a data-dependent scaling factor $\widehat{\alpha}$ and a weighting function $w:[0, \infty) \to [0,1]$, we let $\widehat{\mu}=\widehat{\kappa} + \frac{1}{n}\sum_{i=1}^n(X_i - \widehat{\kappa})w(\widehat{\alpha}|X_i-\widehat{\kappa}|)$.
    We prove that, under mild assumptions over $w$, these estimators achieve stronger concentration bounds than the base estimate $\widehat{\kappa}$, including sub-Gaussian guarantees. 
    This framework unifies and extends several existing approaches to robust mean estimation in $\R$, and can also be generalized to the multivariate setting.
    Through numerical experiments, we show that our shrinking approach translates to faster concentration, even for small sample sizes.
	\end{abstract}


\section{Introduction}\label{sec:intro}
Let \( P \) be an unknown distribution over \( \R \) and let \( X_{1:n}=(X_i)_{i=1}^n \sim P^{\otimes n}\) be an independent sample from \( P \). Our goal is to approximate the mean \( \mu := \mathbb{E}X_1 \) through an estimator \( \widehat{\mu}=\widehat{\mu }(X_1,\ldots, X_n) \). In this setting, a canonical choice for \( \widehat{\mu} \) is the empirical mean \( \overline{X} := \frac{1}{n}\sum_{i=1}^n X_i \). Indeed, this estimator possesses many appealing properties, such as consistency, asymptotic efficiency, and admissibility among large classes of distributions.

Nevertheless, as observed by \cite{tukey1948some}, the empirical mean presents important limitations when the underlying distribution deviates from the Gaussian. This motivated the study of robust mean estimation, whose aim is to provide good estimators under challenging settings. For instance, \cite{huber1964robust} developed a systematic treatment for location estimation to minimize the asymptotic variance under adversarial distributions.
Typically, robust mean estimators try to improve on the empirical mean by reducing the sensitivity to unlikely but extreme points in the sample, often exploiting the idea of ``outlyingness'' through order statistics or the distance from a good centrality estimate. Some examples are given by the following estimators: 
\begin{gather}
  \widehat{\mu }_{TM}^{(k)}=\frac{1}{n-2k}\sum_{i=k+1}^{n-k} X_{(i)}, \label{eq:tm}\\
  \widehat{\mu }_{WM}^{(k)}=\frac{1}{n}\sum_{i=1}^{n} X_{(i\vee (k+1)\wedge (n-k))},  \label{eq:wm}\\
  \widehat{\mu}_{LV}^{\eta}=\widehat{\kappa }+ \frac{1}{n}\sum_{i=1}^{n} (X_i-\widehat{\kappa })(1-\min(\alpha^2 \lvert X_i -\widehat{\kappa }  \rvert^2, 1 )),  \label{eq:lv}
\end{gather}
where \( X_{(i)} \) denotes the \( i \)-th order statistic, \( \alpha  \) satisfies \( \sum_{i=1}^n\min(\alpha ^2 \lvert X_i - \widehat{\kappa }  \rvert^2, 1 ) = \eta \), and \( \widehat{\kappa }  \) is a mean-estimator.
The trimmed \eqref{eq:tm} and Winsorized \eqref{eq:wm} means \citep{tukey1963less, huber1964robust} exemplify the former approach, while the estimator introduced by \cite{lee2022optimal} \eqref{eq:lv} follows the latter. Still, all of these examples can be understood as shrinkage estimators.

Inspired by the aforementioned ideas, the present paper studies in a unified way the effect of shrinking the weights of extreme points to minimize their impact on mean estimation. To this end, we propose a family of estimators that uses a base estimate \( \widehat{\kappa }  \) and adds a correction term that takes into account the deviations from \( \widehat{\kappa }  \) to the data, assigning weights for each sample point that decrease with the distance to \(\widehat{\kappa } \). More specifically, let $w:[0, \infty) \to [0, 1]$ be a non-increasing function and let $\widehat{\alpha} > 0$ be a data-dependent scaling factor defined as a $Z$-estimator\footnote{This definition suffices for continuous choices of $w$; the general definition is given in \eqref{eq:def_alpha}.} via $\sum_{i=1}^n w(\widehat{\alpha}|X_i - \widehat{\kappa}|) = n - \eta$, where $\eta \in (0, n)$ is a parameter. We then define estimators of the form
\begin{equation}
	\label{eq:def_estimator-intro}
  \widehat{\mu }(\widehat{\kappa };X_{1:n} ) :=\widehat{\kappa} + \frac{1}{n}\sum_{i=1}^n(X_i - \widehat{\kappa})w(\widehat{\alpha}|X_i-\widehat{\kappa}|).
\end{equation}
Many robust estimators, or variants thereof, can be written in this form for suitable choices of $w$, such as \eqref{eq:tm}, \eqref{eq:wm} and \eqref{eq:lv}. As such, we provide a unified proof that these estimators, and many more, are sub-Gaussian (see \Cref{thm:shrinkage_concentration_adv} and \Cref{sec:examples}).

There are several advantages from this unified framework. First, we obtain non-asymptotic high-probability bounds for the error of \( \widehat{\mu }(\widehat{\kappa }; X_{1:n} )  \), under mild assumptions on \( w \) for this large class of estimators. In particular, we show that its concentration often improves upon that of \( \widehat{\kappa} \) and attains optimal rates when the error of \( \widehat{\kappa} \) is \( O(1) \) with high probability. Second, the possibility of iteratively shrinking the base estimator due to flexibility in the choice of \( \widehat{\kappa }  \); see \Cref{remark:iteration}. Third, these shrinkage estimators immediately inherit (or improve) other desirable properties from the base estimator, such as affine equivariance, high breakdown points, and asymptotic efficiency; see \Cref{sec:estimator_properties}. Fourth, users can select different weighting schemes and base estimators according to the relevant setting; we provide guidance as to the choice of $w$ in \Cref{sec:experiments}. Finally, we can generalize this class of estimators to the multivariate setting by downweighting deviations \( X_i -\widehat{\kappa }  \) by a factor of \( w(\widehat{\alpha } \lVert X_i - \widehat{\kappa } \rVert ) \), where we also obtain concentration bounds under adversarial contamination; see \Cref{sec:intro_multivariate}.

\subsection{Main results}
\subsubsection{Optimal concentration via shrinkage}\label{sec:optimal_concentration}
We prove concentration for \( \widehat{\mu }  \) under adversarial contamination; that is, we assume that an underlying sample \( X_{1:n} \sim P^{\otimes n}\) is adversarially contaminated with contamination level \( \varepsilon \in [0,1] \), resulting in a corrupted sample \( X_{1:n}^\varepsilon \in \sA(X_{1:n},\varepsilon ):=\{Y_{1:n}: \#\{i \in \{1,\ldots,n\}: X_i \neq Y_i \}\le n \varepsilon  \}  \). The statistician has access only to the corrupted sample \( X_{1:n}^\varepsilon  \) and an independent estimate \( \widehat{\kappa }  \) (possibly computed in an independently contaminated sample). We let \( \widehat{\mu }^\varepsilon =\widehat{\mu }(\widehat{\kappa } ;X_{1:n}^\varepsilon )  \) as in \eqref{eq:def_estimator-intro}. The concentration rates of \( \widehat{\mu }^\varepsilon   \) depend upon the concentration of \( \widehat{\kappa } \). We make this dependence explicit through the following definition: for \( \delta \in (0,1) \), define 
\begin{equation}\label{eq:def-error-adv}
  R_{\widehat{\kappa }}(\delta) := \inf\left\{r>0 : \Pr\left[|\widehat{\kappa }- \mu| > r\right] \leq \delta \right\}.
\end{equation}
We then have the following informal version of the main result, presented in \Cref{sec:main_result}, which holds under mild assumptions on \( w \) that are fully delineated in \Cref{sec:largefamily}:
\begin{theorem}[Concentration under adversarial contamination --- informal statement]\label{thm:adv_informal}
  Assume \( \eta=\ln \frac{1}{\delta }+\frac{3}{2}n \varepsilon  \) and that \( \widehat{\mu }^\varepsilon=\widehat{\mu }(\widehat{\kappa };X_{1:n}^\varepsilon  )    \) is computed on the contaminated sample \( X_{1:n}^\varepsilon  \), with \( \widehat{\mu }  \)  given by \eqref{eq:def_estimator-intro} and $\widehat{\kappa}$ computed on an independent (possibly corrupted) sample. Further assume that \( R_{\widehat{\kappa } }(\delta ) \le C_{\widehat{\kappa } } \nu_p \) for all \( p>1 \). For \( \delta \in (0,1), n\in \N \) large enough, for all atom-free distributions \( P \), with probability at least \( 1-4\delta  \), for all possible contaminated samples \( X_{1:n}^\varepsilon  \in \sA(X_{1:n},\varepsilon ) \),   \begin{equation}\label{eq:adv_bound_intro}
    \lvert \widehat{\mu}^\varepsilon -\mu   \rvert \le C_{w,\widehat{\kappa }} \left[\inf_{1<p \le  2}\nu_p\left(\frac{1}{n}\ln \frac{1}{\delta }\right)^{1-\frac{1}{p}} + \inf_{p>1}\nu _p\varepsilon ^{1-\frac{1}{p}}\right],
  \end{equation}      
  where \( \nu_p = \mathbb{E}\left[\lvert X-\mu  \rvert^p \right]^{\frac{1}{p}} \), and \( C_{w,\widehat{\kappa }  } \) depends only on \( w \) and \( C_{\widehat{\kappa } } \).
\end{theorem}
 As we will show in \Cref{sec:main_result}, the sample median, hereafter denoted by \( M  \), satisfies the bound \( R_M (\delta )\le \frac{2\nu _p }{1 / 2 - \varepsilon } \) for \( \delta>e^{-c\left( 1-2 \varepsilon  \right) m} \) and all \( p>1 \) when computed on a \(\varepsilon\)-contaminated sample of size \(m\). The necessity of the terms inside the brackets in \eqref{eq:adv_bound_intro} follows from \cite[Theorem 3.1]{devroye2016sub} and \cite[Lemma 5.4]{minsker2019uniform}, respectively. Thus, for a wide class of reweighting schemes \( w \), the associated shrinkage estimator achieves optimal concentration bounds. It is worth noting that this result holds even when the population mean of \( P \) is far from its population median, implying that the shrinkage estimators are able to substantially improve upon the concentration rates of an inadequate mean estimator, often leading to optimal bounds. 

\subsubsection{Improved concentration rates}
Indeed, our main result is more general; it guarantees concentration bounds without necessarily assuming that \( R_{\widehat{\kappa }}\) is conveniently bounded. For the sake of exposition, let us consider the uncontaminated case (\( \varepsilon =0 \)) and assume that the underlying distribution \( P \) has finite variance. 
Our main result, \Cref{thm:shrinkage_concentration_adv}, implies the following.

\begin{theorem}[Main result --- informal statement]\label{thm:main_informal}

  Let $\widehat{\kappa}$ be independent of $X_{1:n}$, $n \in \N$ and $\delta > e^{-cn}$. Take \( \eta=\ln \frac{1}{\delta } \) and \( \widehat{\mu }=\widehat{\mu }(\widehat{\kappa };X_{1:n} )   \) as in \eqref{eq:def_estimator-intro}. If $X_{1:n} \sim P^{\otimes n}$ for some atom-free distribution P satisfying $\nu_2(P) := \nu_2 = \mathbb{E}\left[\lvert X_1-\mu  \rvert^2 \right]^{1/2} < \infty$, then there exists a constant $C_{w}>0 $ depending only on the choice of $w$ such that, with probability at least \( 1 -4\delta , \) 

    \begin{equation}\label{eq:bound_simple_intro}
    \lvert \widehat{\mu }-\mu   \rvert \le  C_{w}  (\nu_2 + R_{\widehat{\kappa } }(\delta))\sqrt{ \frac{1}{n}\ln \frac{1}{\delta }}.
  \end{equation}
\end{theorem}
In other words, if \( \widehat{\kappa }  \) is such that \( R_{\widehat{\kappa } }(\delta ) \gg \sqrt{ \frac{1}{n} \ln \frac{1}{\delta }}\) for some regime of \( (n,\delta ) \) such that \( \delta  \gg e^{-cn} \) , then the estimator \( \widehat{\mu }  \) offers an improvement over \( \widehat{\kappa }  \) in terms of concentration rates. In particular, let \( \widehat{\kappa } \) be the empirical average of some independent sample of size $m= \left\lfloor nq \right\rfloor$ for some constant \( q \in (0,1) \), let \( \{P_k\}_{k\ge 1}  \) a sequence of distributions for which Chebyshev's inequality is sharp \citep{catoni2012challenging}, and take \( \delta =\frac{1}{16m} \) such that \(\delta  \gg e^{-cn}  \). One has that \( R_{\overline{X}}(\delta) = \Theta(1) \), implying that the shrinkage estimator improves upon the base estimate in terms of concentration rates, which is often the case even for already optimally robust estimators such as the trimmed mean (see also the computational experiments in \Cref{sec:experiments}).

Contrastingly with respect to the result presented in \Cref{thm:adv_informal}, the bound given in \Cref{thm:main_informal} is still informative if \( R_{\widehat{\kappa } }(\delta ) \gg 1\): in the setup of the previous paragraph, if we instead take $m=n$ and \( \delta =cn^{-\frac{5}{4}} \), we would have that \( R_{\overline{X}} = \Theta(n^{\frac{1}{4}}) \gg 1 \), while \[
  C_{w,\widehat{\kappa } }(\nu _2 +R_{\widehat{\kappa } }(\delta ))\sqrt{\frac{1}{n} \ln \frac{1}{\delta }} \le C n^{-\frac{1}{4}}\sqrt{\ln n} \ll 1.  
\] 

\subsubsection{Breakdown point, asymptotic efficiency and affine equivariance}\label{sec:estimator_properties}
Beyond optimal concentration rates under mild assumptions on \( \widehat{\kappa }  \), our shrinkage estimator \( \widehat{\mu}  \) also inherits other desirable properties from the base estimator. For example, if \( \widehat{\kappa }  \) is affine-equivariant then so is \( \widehat{\mu }(\widehat{\kappa } )  \). Another desirable property that \( \widehat{\mu }(\widehat{\kappa }; X_{1:n} )  \) enjoys is asymptotic efficiency; fix \( \eta >0 \) and denote \( \widehat{\mu }_n = \widehat{\mu }(\widehat{\kappa }_n; X_{1:n} )    \) for a sequence of \(\widehat{\kappa }_n\) independent of $X_{1:n}$. If we assume that \( \nu_{p}< \infty \) for any \( p>2 \) and that \( (\widehat{\kappa }_n - \mu)n^{-\frac{1}{2}} \to 0  \) in probability, then \( \sqrt{n}(\widehat{\mu }_n - \mu  ) \overset{d}{\to}  N(0,\nu_2^2) \). That is, the shrinkage estimator \( \widehat{\mu }_n  \) is undistinguishable from the empirical mean in terms of asymptotic distribution. Notice that the assumption that \( (\widehat{\kappa }_n - \mu  )n^{-\frac{1}{2}} \to  0 \) is very weak: in particular it holds for any constant base estimator \( \widehat{\kappa}_n  \).  

Finally, we also show that, under mild assumptions on \( w \), the breakdown point \( \varepsilon^\star= \varepsilon^\star(\widehat{\mu }(\widehat{\kappa } ), (X_{1:n},Y_{1:m}) )\) of the estimator \( \widehat{\mu }(\widehat{\kappa } )  \), given by \( (X_{1:n},Y_{1:m}) \mapsto \widehat{\mu}(\widehat{\kappa }(Y_{1:m});X_{1:n} )  \), is defined as \[
  \varepsilon^\star := \inf \left\{\varepsilon\in [0,1] : \sup_{\substack{X_{1:n}^\varepsilon \in \sA(X_{1:n}, \varepsilon )\\Y_{1:m}^\varepsilon \in \sA(Y_{1:m}, \varepsilon ) }} \lvert \widehat{\mu }(\widehat{\kappa }(Y_{1:m}); X_{1:n})-\widehat{\mu }(\widehat{\kappa }(Y_{1:m}^\varepsilon ); X_{1:n}^\varepsilon )   \rvert = \infty \right\}. 
\] is at least as high as that of \( \widehat{\kappa }  \), as long as \( \eta \ge n \varepsilon ^\star(\widehat{\kappa }; Y_{1:m} ) \). When \( \widehat{\kappa }  \) is taken to be the sample median of \( Y_{1:m} \) and \( \eta \ge n / 2 \), this implies that \( \varepsilon ^\star(\widehat{\mu }(\widehat{\kappa } ),(X_{1:n},Y_{1:m})) = 1 / 2 \) for every pair of samples \( (X_{1:n},Y_{1:m}) \). %
See \Cref{sec:properties} for more details.

\subsubsection{Extension to $\R^d$}\label{sec:intro_multivariate}

Let \( X_{1:n} \sim P^{\otimes n}\), where \( P \) is a distribution on \( \R^d \) with mean \( \mu  \) and well-defined covariance matrix \( \Sigma  \). Let \( \eps \in [0,\frac{1}{2}) \). As before, we want to estimate the mean \( \mu  \) based on an adversarially contaminated version of the sample, \( X_{1:n}^\varepsilon  \in \sA(X_{1:n},\varepsilon ) \). In order to extend our shrinkage estimators to the multivariate setting, we consider estimators of the same form as before, except that the weights are now calculated based on the norm of the deviations from the base estimator \( \widehat{\kappa }  \) instead of their absolute value. More precisely, we consider estimators of the form
\begin{equation}
  \widehat{\mu }(\widehat{\kappa };X_{1:n} )=\widehat{\mu }(\widehat{\kappa } ):= \widehat{\kappa }+ \frac{1}{n}\sum_{i=1}^n (X_i -\widehat{\kappa } )w(\widehat{\alpha }(\widehat{\kappa } ) \lVert X_i -\widehat{\kappa }  \rVert  ),\label{eq:def_estimator_multivariate_intro}
\end{equation}
where \( \widehat{\alpha}(\kappa )   \) is such that \( \sum_{i=1}^n w(\widehat{\alpha }(\kappa ) \lVert X_i - \kappa \rVert ) = n -\eta  \). 
Let $\tr(\Sigma)$ be the trace of \( \Sigma \). We extend our main result as follows:
\begin{theorem}[Concentration in \( \R^d \) --- informal statement]\label{thm:multivariate_concentration_adv_informal}
  Let $X_{1:n}\sim P^{\otimes n}$, where \( P \) has well-defined covariance matrix \( \Sigma\)  and is absolutely continuous with respect to the Lebesgue measure. Fix \( \delta \in (0,1) \) and \( \varepsilon \in [0,1/2) \). Set $\eta = \ln\frac{4}{\delta}+\frac{3}{2} \eps n $. If \( n \) is large enough and  $\widehat{\mu}^\eps  =  \widehat{\mu}(\widehat{\kappa };X_{1:n}^\eps) $ is as in \eqref{eq:def_estimator_multivariate_intro} and $R_{\widehat{\kappa }}(\delta) := \inf\left\{r>0 : \Pr\left[\lVert\widehat{\kappa }- \mu\rVert > r\right] \leq \delta \right\},$ then, with probability at least $1-\delta $, for all \( X_{1:n}^\varepsilon \in \sA(X_{1:n},\varepsilon ) \), \( \lVert \widehat{\mu }^\varepsilon -\mu   \rVert  \) is at most  
  \begin{equation}
    C_{w}\left\{ \left(\sqrt{\tr(\Sigma)}+ R_{\widehat{\kappa } }\left(\frac{\delta}{4}\right)\right) \sqrt{\frac{1}{n}\ln \frac{4}{\delta } + \varepsilon }  \right\},\label{eq:conclusion_concentration_multivariate_informal}
  \end{equation}
where \( C_{w } \) is a constant depending only on \( w \).
\end{theorem}
While this guarantee is sub-optimal in general, it attains the best concentration rates among known estimators of linear computational complexity (i.e., \( O(nd) \)), matching the geometric median-of-means \cite{minsker2015geometric} in both senses.

\subsection{Related works}\label{sec:related_works}

Estimators satisfying concentration bounds such as the ones given in \eqref{eq:adv_bound_intro} have been the focus of many recent papers in robust mean estimation. In particular, for the case \( \varepsilon =0 \) and restricting the infima to \( p=2 \), these estimators are called \( \delta  \)-dependent sub-Gaussian mean estimators. 
Indeed, \cite{devroye2016sub} showed that such guarantees are optimal in terms of probabilistic concentration. Recent efforts have established similar high-probability bounds. The first work to study whether estimators are sub-Gaussian in the above sense was \cite{catoni2012challenging}. In this paper, Catoni shows that Chebyshev's inequality is sharp, implying that the empirical mean concentrates badly for certain distributions with finite variance. He then proposes a sub-Gaussian estimator that is constructed as a $Z$-estimator satisfying $\sum_{i=1}^n \psi( X_i - \mu ) = 0$, for $\psi(x) = \phi(\beta x)$ where $\beta > 0$ is a parameter, $\phi$ is increasing and satisfies $- \ln( 1 - x + x^2/2) \leq \phi(x) \leq \ln( 1 + x + x^2/2)$. The parameter $\beta$ plays the role of a scaling factor and the bounds on $\phi$ enforce diminishing contributions from outliers. \cite{catoni2012challenging} only manages to obtain sub-Gaussian deviations assuming knowledge of an upper bound $\nu \geq \sigma^2$ and setting $\beta \approx \sqrt{(2/(n\nu) \ln (1/\delta)}$. With no prior knowledge of $\sigma^2$, Lepski's method \citep{lepskii1991problem} can be employed to select $\beta$, yielding an almost sub-Gaussian bound. In fact, \cite{devroye2016sub} showed that, in order for an estimator to attain these guarantees without prior knowledge of \( \sigma^2 \), it needs to be \( \delta-  \)dependent, justifying our choice of the parameter \( \eta  \) as a function of \( \delta  \) in \Cref{thm:adv_informal,thm:main_informal}.   

Following \cite{catoni2012challenging}, several papers have shown that classical estimators are also sub-Gaussian (e.g., see the survey \cite{lugosi2019mean}). For instance, \cite{lerasle2011robust} showed that the median-of-means \citep{nemirovskij1983problem, jerrum1986random, alon1996space} is sub-Gaussian when $P$ has second moment; \cite{bubeck2013bandits} showed that \eqref{eq:adv_bound_intro} also holds when the infima are restricted to any single \( p \in (1,2) \). Recently, \cite{oliveira2025finite} showed the general adversarially contaminated case \eqref{eq:adv_bound_intro} for the trimmed mean \eqref{eq:tm}, and \cite{lugosi2019mean} proved that the Winsorized mean \eqref{eq:wm} is sub-Gaussian assuming the existence of second moment; \cite{kock2025winsorizedmeanestimationheavy} then extended this result to allow for adversarial corruption and infinite variance.

Still in the finite variance setting, \cite{lee2022optimal} showed that the estimator \eqref{eq:lv} is sub-Gaussian. While they do not provide constants in finite samples, they show that the constant approaches the optimal $L = \sqrt{2}$ in the asymptotic limit $\big(\frac{1}{n}\ln\frac{1}{\delta}, \delta\big)\to (0,0)$. They do so through a reweighting scheme on top of an already sub-Gaussian base estimator $\widehat{\kappa}$, exploiting a data-dependent scaling factor analogous to the Z-estimator \( \widehat{\alpha }  \). More recently, \cite{lee2025} showed that this estimator preserves optimal concentration bounds under adversarial contamination and infinite-variance distributions under the same asymptotics and assuming $\eps n \in O\big(\ln\frac{1}{\delta}\big)$. Our framework is fundamentally more general and conceptually distinct in two key aspects. First, we make no structural assumptions on the base estimator \(\widehat{\kappa}\), which can be any preliminary estimate of location, and may not even be consistent. Second, while \cite{lee2022optimal} considers only the choice $w(t) = (1-t^2)_+$, our work applies to a broad class of shrinkage functions \(w(\cdot)\) and proves that the resulting estimators $\widehat{\mu}$ concentrate well through a single unified device. %

The idea of using an initial location estimate and then reweighting the sample points to obtain a better estimate is not new. W-estimators were introduced by \cite{mosteller1977data} as estimators \( (\sum_{i=1}^n w_i)^{-1}\sum_{i=1}^nw_i X_i \), where the weights \( w_i \) are a non-increasing function of the normalized deviation \( \lvert X_i - \widehat{\kappa }  \rvert  \), and \( \widehat{\kappa }  \) is chosen to be the sample median. \cite{donoho1992breakdown} generalized this idea to higher dimensions, providing breakdown point guarantees for such estimators. Our work can also be cast as studying the properties of such estimators, with the weights possibly depending on a scale estimate, and providing concentration properties, asymptotic efficiency, and breakdown point. See \Cref{sec:largefamily} for more details.

Other novel sub-Gaussian estimators have also been proposed. \cite{minsker2021robust} introduce an estimator that splits the sample in $K$ buckets, and use the empirical deviation $\widehat{\sigma}_j$ and empirical mean $\widehat{\mu}_j$ of each bucket $j = 1,2,\dots, K$ to construct the weighted average $\sum_{j=1}^K \widehat{\mu}_j \widehat{\sigma}_j^{-r} / \sum_{j=1}^K \widehat{\sigma}_j^{-r}$ for some parameter $r \geq 1$, showing that it is robust to a weaker form of contamination.
\cite{minsker2023ustatistics} studies some modifications to the median-of-means construction to attain asymptotically optimal constants even when \( \delta \not\to 0 \), widening the regime of \( (n,\delta ) \) considered by \cite{lee2022optimal} (see also \cite{lecue2020robust}). This work was improved by \cite{minsker2023efficient}, which provide a more computationally efficient estimator while also alleviating moment assumptions. More recently, \cite{li2025jointrobustestimation} have proposed a joint estimator of the mean and variance by leveraging ideas introduced by Catoni while dispensing with the reliance on an appropriately chosen parameter.

A natural continuation of this line of work provides similarly optimal concentration bounds for mean estimators in \( \R^d \) assuming only finite covariance matrix. Accordingly, \cite{minsker2015geometric} provided an algorithm that is linear on $d$, but with sub-optimal concentration bounds. Optimal concentration was first obtained by \cite{lugosi2019vector}, but with a computationally infeasible estimator. \cite{hopkins2020, cherapanamjeri2019fast, depersin2022robust, mathieu2022concentration} then proposed novel estimators exhibiting optimal concentration bounds, while improving computational efficiency. Still, the computational complexity of these estimators scale non-linearly in $d$, revealing a tension between computational cost and optimality. Beyond improved computational efficiency, \cite{depersin2022robust} also proved that their estimator is resistant to adversarial contamination, although with an assumption that the contamination level \( \varepsilon  \) is not too large. Independent multivariate versions of the trimmed mean were then shown by \cite{lugosi2021, oliveira2025} to attain optimal concentration bounds under this contamination model, at the cost of computational feasibility.
In Section \ref{sec:extension_multivariate} we provide an extension of our method to $\R^d$, with computational cost that is linear on $d$. Moreover, it attains the best-known concentration guarantee over linear-time algorithms under adversarial contamination.

\section{Preliminaries}\label{sec:largefamily}

\subsection{Notation and setup} Define \( [m]:=\{1,\ldots,m\}  \). Let $\widehat{P}_n$ denote the empirical distribution of $X_{1:n}\sim P^{\otimes n}$, and let $X\sim P$ and $Pf = \E[f(X)]$, with $\lVert f\rVert_{L_p(P)}=(\E[\lvert f(X)\rvert ^p])^{1/p}$, $\lVert f\rVert_{\infty}=\sup_{x\in\sX}\lvert f(x)\rvert $. Denote \( \sA(X_{1:n},\varepsilon ) \) the set of all possible \( \varepsilon \)-contaminated samples \( \{X_{1:n}^\varepsilon : \#\{i \in [n]: X_i \neq X_i^\varepsilon \}\le n \varepsilon  \}  \). Let \( \mu = \mathbb{E}X_1, \nu_p = \mathbb{E}\left[\lvert X_1-\mu  \rvert^p  \right]^{\frac{1}{p}} \).    

Given $p>1$, let $\sP$ denote a class of distributions in $\sP_p:=\{P: P\text{ has finite }p\text{-th moment}\}$. We have access to an adversarially contaminated sample \( X_{1:n}^\varepsilon  \in \sA(X_{1:n},\varepsilon ) \) 
where $X_{1:n}\sim P^{\otimes n}$ for some unknown $P\in\sP$ and we aim to estimate $\mu$. We are given a base estimator $\widehat{\kappa}$, and are interested in improving it through shrinkage. As in \Cref{sec:intro}, we consider $w : [0,\infty) \to [0,1]$, a right-continuous non-increasing function satisfying $w(0) = 1$ and $\lim_{t\to\infty}w(t)=0$. When \( w \) is continuous, the data-dependent scaling factor \( \widehat{\alpha} \) may be defined as a $Z$-estimator. However, for the sake of generality, we will consider the following alternative definition:
\begin{equation}
	\label{eq:def_alpha}
  \widehat{\alpha}(\kappa; X_{1:n})= \widehat{\alpha}(\kappa) := \inf\left\{ \alpha > 0 : \sum_{i=1}^n w(\alpha |X_i-\kappa|) \leq n - \eta \right\},
\end{equation}
where \( \eta \) is deemed the shrinkage level, since it determines how much of the sample will be discarded. When \( w \) is continuous, the $Z$-estimator defined by $\frac{1}{n}\sum_{i=1}^n w(\widehat{\alpha}(\widehat{\kappa } )|X_i - \widehat{\kappa}|) = 1 - \frac{\eta}{n}$, as proposed in \Cref{sec:intro}, is a particular case of the more general definition given by \eqref{eq:def_alpha} above. In that sense, $w(\widehat{\alpha}(\widehat{\kappa } )|X_i - \widehat{\kappa}|)$ can be thought of as the amount by which $X_i$ will be shrunk. Our proposed estimator is then given as in \eqref{eq:def_estimator-intro}, or, equivalently,
\begin{equation}
	\label{eq:def_estimator}
  \widehat{\mu}(\widehat{\kappa }; X_{1:n} )=\widehat{\mu}(\widehat{\kappa } )
    := \widehat{\kappa} + \frac{1}{n}\sum_{i=1}^n(X_i - \widehat{\kappa})w(\widehat{\alpha}(\widehat{\kappa } )|X_i-\widehat{\kappa}|)
    =\frac{\eta}{n}\widehat{\kappa} + \frac{1}{n}\sum_{i=1}^n X_i w(\widehat{\alpha}(\widehat{\kappa } ) |X_i-\widehat{\kappa}|).
\end{equation}

The last equality in \eqref{eq:def_estimator} holds if \( w \) is continuous. Thus, the resulting estimator can be thought of as an interpolation between \( \widehat{\kappa }  \) (when $\eta=n$) and the empirical mean (when $\eta=0$). For the contaminated versions of the aforementioned estimators, we denote \( \widehat{\alpha }^\varepsilon(\kappa )=\widehat{\alpha }(\kappa ; X_{1:n}^\varepsilon )\), \(  \widehat{\mu }^\varepsilon(\kappa )=\widehat{\mu }(\kappa; X_{1:n}^\varepsilon )\), \(  \widehat{\alpha  }^\varepsilon=\widehat{\alpha  }(\widehat{\kappa } ; X_{1:n}^\varepsilon )\), and \(  \widehat{\mu }^\varepsilon=\widehat{\mu }(\widehat{\kappa } ; X_{1:n}^\varepsilon )\).

\begin{remark}
The alternative estimator  \begin{equation}\label{eq:def_estimator_weighted}
  \widehat{\mu}_W(\kappa;X_{1:n} ):=\sum_{i=1}^n\frac{X_i w(\widehat{\alpha}(\kappa ) |X_i-\kappa|)}{\sum_{i=1}^nw(\widehat{\alpha }(\kappa )\lvert X_i-\kappa   \rvert  )},
\end{equation} 
where \( \widehat{\alpha }  \) is the one given in \eqref{eq:def_alpha}, also satisfies the concentration bounds given in \Cref{thm:shrinkage_concentration_adv} as a corollary. The proof is relegated to the appendix. This estimator is a W-estimator when \( \widehat{\kappa }  \) is chosen to be the sample median.
\end{remark}
The function $w$ controls how we treat deviant sample points: a compactly-supported $w$ may discard a few sample points altogether, whereas a strictly positive $w$ includes some proportion of each sample point in the average of the deviations. Naturally, the performance of \( \widehat{\mu}^\varepsilon   \) is necessarily dependent on the quality of this base estimator, which is made clear in our concentration bounds given by \Cref{thm:shrinkage_concentration_adv}. Nevertheless, we will show that this dependence is not overly strong and, therefore, that we can achieve good concentration bounds for \( \widehat{\mu}^\varepsilon   \) even if $\widehat{\kappa} $ does not concentrate well.

\vspace{1em}

\subsection{Assumptions} To obtain concentration bounds for $\widehat{\mu}^\varepsilon $, we will assume the following hypotheses.

\begin{assumption}\label{assum:malpha_finite}
	The function $w(\cdot)$ satisfies $\sup_{t\geq0}tw(t)<\infty$.
\end{assumption}
This hypothesis ensures that $w(t)$ decays at least as fast as $1/t$ for large $t$. In practice, it ensures that no single sample point can have an arbitrarily large impact on the estimator \( \widehat{\mu }  \) in \eqref{eq:def_estimator}.
In other words, the quantity \( \sup_{t\ge 0}t w(t) \) measures the maximum contribution that a sample point can have towards the sum in \eqref{eq:def_estimator}, up to the data-dependent scaling factor \( \widehat{\alpha }  \). Note that Assumption \ref{assum:malpha_finite} is satisfied whenever $w$ has compact support, which is the case for many natural choices of $w$ (see Section \ref{sec:examples}).

\begin{assumption}\label{assum:rho_bound}
	The function $w(\cdot)$ satisfies $w(t)\geq (1-t^p)_+$ for all $t\geq0$.
\end{assumption}
This hypothesis ensures $w(t)$ does not decay too quickly for small $t$. 
In particular, for differentiable $w$, this hypothesis guarantees that $w'(0)=0$ so the closest sample points to the base estimate contribute similarly and \eqref{eq:def_estimator-intro} becomes close to the sample mean. Note that if this assumption holds for \( q>1 \), it then holds for all \( p \in (1,q] \). In particular, one may choose \( w  \) such that \( w(t)\ge \mathbf{1}_{t<1} \), which implies that \Cref{assum:rho_bound} holds for all \( p>1 \).

\begin{assumption}\label{assum:kappa_independent}
	The base estimator \( \widehat{\kappa }  \) is independent from the sample \( X_{1:n} \).
\end{assumption}
This simplifying assumption allows us to state our results with mild assumptions on \( w \), and can be dispensed with through a case-by-case analysis. For instance, we show in \Cref{sec:kappa_dependent} that this assumption may be altogether avoided for the trimmed mean \eqref{eq:tm} and its variants (see \Cref{ex:rho_atm}). Alternatively, sample splitting and averaging can be employed to satisfy this assumption while retaining the guarantee \eqref{eq:bound_simple_intro}.

\begin{assumption}\label{assum:no_point_mass}
	The distributions $P\in\sP$ are atom-free.
\end{assumption}
This assumption is commonly found in the literature to avoid specifying tie-breaking criteria (but we do allow ties in contaminated samples). 
It is relatively weak and can be dispensed with when $w$ is continuous or by convolving the data with a continuous random variable with arbitrarily small variance.

\section{Concentration of shrinkage estimators}\label{sec:main_result}
We now state the main result of this paper in full generality, which provides concentration bounds for the shrinkage estimator $\widehat{\mu}^\varepsilon $ defined in \eqref{eq:def_estimator}. The main ingredients for the proof are discussed in \Cref{sec:proof_of_theorem_3_1}, with details included in \Cref{appendix:proofs}.
\begin{theorem}[Concentration of shrinkage estimators under adversarial contamination]\label{thm:shrinkage_concentration_adv}
  Let $X_{1:n}\sim P^{\otimes n}$ with $P\in\sP$. Let \( \eps \in [0,\frac{1}{2}) \). Under \Cref{assum:no_point_mass,assum:malpha_finite,assum:rho_bound,assum:kappa_independent}, set $\eta = \ln\frac{4}{\delta}+(1+\xi )\eps n $ for some \( \xi>0 \). Assume \( (1+2\xi)\eps + \frac{\overline{c}}{n}\ln \frac{4}{\delta } < 1 \) , where \( \overline{c}=16+4\xi^{-1} \). If $\widehat{\mu}^\eps  =  \widehat{\mu}(\widehat{\kappa };X_{1:n}^\eps) $ is as in \eqref{eq:def_estimator} and $R_{\widehat{\kappa}}=R_{\widehat{\kappa } }\left(\frac{\delta}{4}\right)$ is as in \eqref{eq:def-error-adv}, then, with probability at least $1-\delta$, for all \( X_{1:n}^\varepsilon \in \sA(X_{1:n},\varepsilon ) \), \( \lvert \widehat{\mu }^\varepsilon -\mu   \rvert  \) is at most  
	\begin{equation}\label{eq:conclusion_concentration}
    C_{w,\xi }\left\{\inf_{1<q\le 2\wedge p} \left(\nu _q + \nu _q^{\frac{q}{2}}R_{\widehat{\kappa } } ^{1-\frac{q}{2}}\right)\left(\frac{1}{n }\ln \frac{4}{\delta }\right)^{1-\frac{1}{q}}+\inf_{1<q\le p} \left(\nu _q +R_{\widehat{\kappa } }  \right)\left(\frac{1}{n }\ln \frac{4}{\delta } + \varepsilon \right)^{1-\frac{1}{q}}\right\}
	\end{equation}
  for $C_{w,\xi }>0$ depending only on $w$ and \( \xi  \).

  In particular, if there exist constants \( C , \delta_\textup{min}>0 \) depending only on the base estimator such that \( R_{\widehat{\kappa}}(\frac{\delta}{4})  \le C \nu _p \) for \( \delta \ge \delta_\textup{min} \) and \( P\in\sP \) , then \( \widehat{\mu}^\varepsilon   \) attains for \( \delta \ge \delta_{\textup{min}}\vee 4e^{-\frac{1-(1+2\xi)\varepsilon}{\overline{c}}n} \): \begin{equation}\label{eq:conclusion_concentration2}
    \Pr\left[ \sup_{X_{1:n}^\varepsilon \in \sA(X_{1:n},\varepsilon )}|\widehat{\mu}^\varepsilon  -\mu |  > \tilde{C}_{w,\xi}\left\{ \nu_{2\wedge p} \left(\frac{1}{n}\ln\frac{4}{\delta}\right)^{1-\frac{1}{2\wedge p}} + \nu^p \varepsilon^{1-\frac{1}{p}}\right\} \right] \leq \delta,
  \end{equation} 
  where \( \tilde{C}_{w,\xi}=C_{w,\xi}(2C+3) \).

\end{theorem}

Informally, \Cref{thm:shrinkage_concentration_adv} describes the error associated to the estimation of the population mean \( \mu \) through the estimator \( \widehat{\mu}^\varepsilon  \). The error in \eqref{eq:conclusion_concentration} can be decomposed into two terms: one that depends on $R_{\widehat{\kappa }} $ and one that does not. The latter term corresponds to the optimal order of mean estimation. The former term is associated with the quality of the base estimate combined with the contamination level and is significantly mitigated when \( \delta \gg e^{-cn}\), since the term \( \left(n^{-1}\ln (4\delta^{-1})\right)^{1-1/p}\) is small in this regime, leaving only the error associated to the contamination \( \varepsilon^{1-\frac{1}{p}} \). This means that for sufficiently large sample sizes, \( \widehat{\mu}^\varepsilon \) may concentrate well around the mean \( \mu \) even if \( \widehat{\kappa }   \) does not. \Cref{remark:higher_moments} emphasizes the magnification of this effect as \( P \) presents finite moments of higher order.

\begin{remark}\label{remark:higher_moments} 
While \( R_{\widehat{\kappa } } = O(1) \) is sufficient for attaining sub-Gaussian guarantees, one can conclude from \eqref{eq:conclusion_concentration} that having moments $p>2$ leads to an improved bound, in the sense that $\widehat{\mu}^\varepsilon $ might retain light tails even if $\widehat{\kappa}$ has poorer concentration. For example, if \( p=4,\varepsilon=0 \), and \( \widehat{\kappa }  \) is such that $R_{\widehat{\kappa } }(\delta ) = \Theta((n^{-1}\ln \delta^{-1})^{-1/4}),$ which diverges as \( n^{-1} \ln \delta^{-1}  \to 0 \), then \( \widehat{\mu }^\varepsilon   \) still retains sub-Gaussian rate tails.
  Notice also that if $w(t) \geq \mathbf{1}_{t < 1}$, \Cref{assum:rho_bound} is satisfied for every $p>1$ so, in particular, the infimum in \eqref{eq:conclusion_concentration} can be taken in \( (1,2] \) for the first term and in \( (1,\infty) \) for the second. It is under this scenario that \Cref{thm:adv_informal} holds (taking \( \xi =\frac{1}{2} \)). We explore choices of $w$ satisfying this lower bound in \Cref{sec:examples}.
\end{remark}

In \Cref{sec:examples_kappa} we provide examples of base estimators \( \widehat{\kappa }  \) that satisfy the condition \( R_{\widehat{\kappa } }(\delta ) \le C \nu _p \), in turn allowing us to directly apply the second part of \Cref{thm:shrinkage_concentration_adv}. In particular, we show that the sample median satisfies this condition and allows us to derive an optimal-rate estimator from an inconsistent one. We also provide a choice of \( \widehat{\kappa }  \), given by \( \widehat{\kappa } = (X_{(\eta / 2 + 1)} + X_{ n - \eta /2}) / 2\), which allows us to recover the trimmed mean \eqref{eq:tm} and show that it concentrates well enough to imply that \( \widehat{\mu }(\widehat{\kappa})    \) attains optimal concentration bounds under adversarial contamination, without the need for sample splitting (see \Cref{sec:kappa_dependent} for details).

\subsection{Proof of \Cref{thm:shrinkage_concentration_adv}} \label{sec:proof_of_theorem_3_1}
We now state the results that are used to prove \Cref{thm:shrinkage_concentration_adv}. 
For this section and the appendices, denote the uncontaminated shrinkage estimator based on \( \widehat{\kappa }   \) given by \( \widehat{\mu}(\widehat{\kappa } ; X_{1:n} )  \) as \( \widehat{\mu }  \). Similarly, denote \(\widehat{\alpha}:=\widehat{\alpha}(\widehat{\kappa};X_{1:n})\). The proof of \Cref{thm:shrinkage_concentration_adv} follows by showing that \( \widehat{\mu }  \) concentrates well and that the error \( \lvert \widehat{\mu}^\eps  - \widehat{\mu}   \rvert  \) is small. \Cref{theorem:bias_variance} and \Cref{lemma:alphaorder,lemma:bias_order} are dedicated to showing the former, while \Cref{lemma:contamination_error} justifies the latter. The main difficulty in proving concentration bounds for \( \widehat{\mu}  \)  is controlling the dependence between the parameter \( \widehat{\alpha}  \) and the sample \( X_{1:n} \).
To that end, we will first assume that, for each \( \kappa \in \R \), \( \widehat{\alpha }(\kappa ):= \widehat{\alpha}(\kappa;X_{1:n})\) is, with high probability, in an interval \( I_\alpha(\kappa ) \) such that the shrinkage estimator's error may be properly controlled, for fixed \( \kappa  \). This is done by bounding $|\widehat{\mu}(\kappa;X_{1:n}) - \mu|$ by $ \sup_{\alpha\in I_\alpha (\kappa )} |s(\alpha,\kappa ) - \mu|$ where
\[
	s(\alpha,\kappa ):=\kappa+\frac{1}{n}\sum_{i=1}^n(X_i-\kappa)w(\alpha|X_i-\kappa|).
\]
The bias induced by a fixed choice of $\alpha$ and $\kappa$ on the weighted sum $s(\alpha,\kappa)$ is
\[ b(\alpha,\kappa) := \E s(\alpha,\kappa) - \mu = \E\left( \kappa + (X-\kappa)w(\alpha |X-\kappa|) \right) - \mu. \]
The estimator \( \widehat{\mu }  \) is obtained by taking $\kappa = \widehat{\kappa}$ and $\alpha = \widehat{\alpha}$. Formally, if it holds that $\widehat{\alpha}(\kappa ) \in I_\alpha(\kappa) $ with high probability for some set $I_\alpha (\kappa )$, then, with high probability, 
\begin{equation}
	\label{eq:basic_biasvariance}
  |s(\widehat{\alpha}(\kappa ), \kappa ) - \mu| \leq \sup_{\alpha \in I_\alpha (\kappa )} |s(\alpha,\kappa ) - \mu| \leq \sup_{\alpha \in I_\alpha (\kappa )} |b(\alpha,\kappa )| + \sup_{\alpha \in I_\alpha (\kappa )} |s(\alpha,\kappa ) - \mu - b(\alpha,\kappa )|.
\end{equation}
We may then use the concentration of \( \widehat{\kappa }  \) in some interval \( I_{\widehat{\kappa } } \) along with its independence with respect to the sample \( X_{1:n} \) to derive a high-probability bound on the error \( \lvert \widehat{\mu } -\mu   \rvert  \). 
The upper bound in \eqref{eq:basic_biasvariance} can be thought of as a bias-variance decomposition. 
\Cref{theorem:bias_variance} below shows that, assuming the existence of the appropriate set $I_\alpha (\kappa )$, we can control the variance term $\sup_{\alpha \in I_\alpha (\kappa )}\lvert  s(\alpha, \kappa )-\mu-b(\alpha,\kappa )\rvert $ in \eqref{eq:basic_biasvariance} via $\|F_\kappa \|_{L^2(P)}^2$ and
\begin{equation}\label{eq:def_malpha}
  m(\alpha):=\sup_{t \ge 0} t w(\alpha t) = \alpha^{-1} \sup_{t \ge 0} t w(t) < \infty,
\end{equation}
where the finiteness is ensured by \Cref{assum:malpha_finite}. \Cref{lemma:alphaorder} then guarantees the existence of a useful choice of $I_\alpha (\kappa )$, which allows for a good bound on the bias and variance terms through \Cref{lemma:bias_order}.

\begin{theorem}[Bias-variance decomposition]\label{theorem:bias_variance} Under \Cref{assum:malpha_finite,assum:kappa_independent}, let \( I_{\widehat{\kappa } } \) be a deterministic set such that \( \widehat{\kappa } \in I_{\widehat{\kappa } }  \) with probability at least \( 1-\frac{\delta}{4} \) and $I_\alpha (\kappa )$ be a family of sets indexed by \( \kappa \in I_{\widehat{\kappa } } \) such that, for every \( \kappa \in I_{\widehat{\kappa } } \), 
	\begin{equation}
		\label{eq:alpha_in_alphak_whp}
    \Pr\left[\widehat{\alpha }(\kappa )\in I_\alpha (\kappa ) \right] \geq 1 - \frac{\delta}{2} .
	\end{equation}
	Define $\sF_\kappa  = \{ x \mapsto \kappa + (x-\kappa)w(\alpha |x-\kappa|)-\mu : \alpha  \in I_\alpha (\kappa )\}$ and $F_\kappa (x) = \sup_{f \in \sF_\kappa } |f(x)|$ for all $x \in \R$.
	Then, with probability at least $1-\delta$,
	\begin{equation}\label{eq:conclusion_bias_var}
    |\widehat{\mu} -\mu \vert \leq \sup_{\kappa \in I_{\widehat{\kappa } }} \left\{\sup_{\alpha \in I_\alpha (\kappa )} \vert b(\alpha,\kappa ) \vert + \left(2K+\sqrt{2}\right) \| F_\kappa  \|_{L_2(P)}\sqrt{\frac{1}{n} \ln \frac{4}{\delta}}+\sup_{\alpha \in I_\alpha (\kappa )}\frac{8 m(\alpha)}{3n}\ln \frac{4}{\delta }\right\}.
	\end{equation}
\end{theorem}

The proof of \Cref{theorem:bias_variance} is achieved through a Bennett-type concentration inequality for the suprema of empirical processes (\Cref{lemma:bousquet}) applied to $\sup_{\alpha \in I_\alpha (\kappa ) }\lvert  s(\alpha,\kappa )-\mu-b(\alpha,\kappa)\rvert $ followed by an integration over the distribution of \( \widehat{\kappa }\). To that end, our empirical process must be indexed by uniformly bounded functions, which is guaranteed by \Cref{assum:malpha_finite}, since \begin{align*}
  \|s(\alpha,\kappa  )-\mu -b(\alpha,\kappa  )\|_{\infty}&=\left\|\frac{1}{n}\sum_{i=1}^n(X_i-\kappa )w(\alpha |X_i-\kappa |)-\E[(X_i-\kappa )w(\alpha |X_i-\kappa |)]\right\|_{\infty}\\
&\le 2m(\alpha ), \end{align*}
    assuming that \( \sup_{\alpha \in I_\alpha (\kappa )} m(\alpha)<\infty \). The application of \Cref{lemma:bousquet} inequality relies on this last bound, but it is easy to see that the inequality \eqref{eq:conclusion_bias_var} holds trivially when \( \sup_{\alpha \in I_\alpha (\kappa )} m(\alpha ) \) is infinite, although we will show that, by the construction of the appropriate set \( I_\alpha (\kappa ) \) in \Cref{lemma:alphaorder} and the bounds given in \Cref{lemma:bias_order} , this is finite and properly controlled.   
	A detailed proof is presented in \Cref{appendix:proofs}.

The next lemma provides the definition of the set \( I_\alpha (\kappa )\) satisfying \eqref{eq:alpha_in_alphak_whp}. 
Intuitively, we will construct sets \( I_\alpha(\kappa  ):=(\underline{\alpha }(\kappa ),\overline{\alpha }(\kappa )) \), with $\underline{\alpha}(\kappa),\,\overline{\alpha}(\kappa)$
given implicitly through $w$ such that \( \widehat{\alpha}(\kappa )\in I_\alpha(\kappa )  \) for all \( \kappa  \in I_{\widehat{\kappa } } \) with high probability. Note that for most of the following bounds we do not require a direct bound on the values of $\widehat{\alpha}(\kappa)$, since the bias and variance terms depend on $\widehat{\alpha}(\kappa)$ only through $w$ (except for $m(\underline{\alpha}(\kappa))$, which is dealt with in \Cref{lemma:bias_order}). This allows us to attain good bounds without restrictive conditions on $w$.

\begin{lemma}\label[lemma]{lemma:alphaorder}
	Under \Cref{assum:no_point_mass}, there exist absolute constants \( \underline{c}, \overline{c}>0 \) such that \( \eta =\ln \frac{4}{\delta }+(1+\xi )\eps n \) and $\underline{\alpha}(\kappa) < \overline{\alpha}(\kappa)$ are implicitly defined by 
	\[ \E w(\underline{\alpha}(\kappa) |X-\kappa|) = 1 - \frac{\underline{c}\ln \frac{4}{\delta }}{n}-\frac{\xi}{7}\varepsilon  \quad \text{ and } \quad  \E w(\overline{\alpha}(\kappa) |X-\kappa|) = 1 - \frac{\overline{c}\ln \frac{4}{\delta }}{n}-(1+2\xi)\eps  \]
  for every \( \kappa \in I_{\widehat{\kappa}} \) and \eqref{eq:alpha_in_alphak_whp} holds with $I_\alpha(\kappa) = (\underline{\alpha}(\kappa), \overline{\alpha}(\kappa))$ whenever \( \frac{\overline{c} \ln \frac{4}{\delta }}{n}+(1+2\xi )\eps < 1 \). Moreover, for all \( \kappa \in I_{\widehat{\kappa } } \), \begin{equation}\label{eq:alpha_contamination}
    \Pr\left[\underline{\alpha }(\kappa )<\inf_{X_{1:n}^\varepsilon \in \sA(X_{1:n},\varepsilon ) }\widehat{\alpha}^\eps (\kappa )\le \widehat{\alpha}(\kappa )< \overline{\alpha}(\kappa )\right]\geq 1-\frac{\delta}{2}  
  \end{equation}
\end{lemma}

The inequality \eqref{eq:alpha_contamination} implies \eqref{eq:alpha_in_alphak_whp}. Accordingly, the proof of \Cref{lemma:alphaorder} consists in showing the former. We prove the stronger statement \eqref{eq:alpha_contamination} in order to provide a good bound for the contamination error \( \lvert\widehat{\mu}^\eps -\widehat{\mu}\rvert   \). To that end, we observe that $\inf_{X_{1:n}^\varepsilon \in \sA(X_{1:n},\varepsilon ) }\widehat{\alpha}^\eps (\kappa )\le \widehat{\alpha}(\kappa )$ holds by definition and so the probability of the complementary event to that of \eqref{eq:alpha_contamination} may be bounded above by
\[
\Pr\left[\inf_{X_{1:n}^\varepsilon \in \sA(X_{1:n},\varepsilon ) }\widehat{\alpha }^\eps _n(\kappa ) \le \underline{\alpha }(\kappa ) \right]+\Pr\left[\widehat{\alpha} (\kappa )\ge \overline{\alpha }(\kappa ) \right]\] 
For the first term, notice that
\[
  \Pr\left[\inf_{X_{1:n}^\varepsilon \in \sA(X_{1:n},\varepsilon )}\widehat{\alpha}^\eps _n(\kappa)\leq\underline{\alpha}(\kappa)\right] \leq\Pr\left[\inf_{X_{1:n}^\varepsilon \in \sA(X_{1:n},\varepsilon ) }\sum_{i=1}^nw(\underline{\alpha}(\kappa)\lvert  X_i^\eps -\kappa\rvert )\leq n-\eta\right],
\]
where we used \Cref{lemma:rho_bounds} and the fact that \( w \) is non-increasing to obtain the last inequality. We are then able to replace \( X_i^\eps  \) with \( X_i \) by noting that
\begin{align*}
  \sum_{i=1}^nw(\underline{\alpha}(\kappa)\lvert  X_i^\eps -\kappa\rvert )-\sum_{i=1}^nw(\underline{\alpha}(\kappa)\lvert  X_i -\kappa\rvert )&=\sum_{i: X_i \neq X_i^\eps } [w (\underline{\alpha }(\kappa )\lvert X_i^\eps -\kappa  \rvert )-w (\underline{\alpha }(\kappa )\lvert X_i - \kappa  \rvert )]\\ 
&\geq-\eps n.
\end{align*}
The same reasoning applies to bound \( \Pr[\widehat{\alpha}(\kappa)\geq\overline{\alpha}(\kappa)] \), except that we do not have to deal with the contamination. Thus, we may control the probability of \( \widehat{\alpha}\notin I_{\alpha }(\kappa ) \) by controlling the concentration of sums of independent random variables through Bernstein's inequality.
See \Cref{appendix:proofs} for details.

The remaining terms $\sup_{\alpha\in I_\alpha (\kappa )}\lvert  b(\alpha,\kappa )\rvert \text{, } \lVert F_\kappa \rVert_{L_2(P)}\text{ and } \sup_{\alpha\in I_\alpha (\kappa )}m(\alpha )= m(\underline{\alpha }(\kappa ))$ are now bounded by employing the appropriate definition of the set $I_\alpha (\kappa )$ presented in \Cref{lemma:alphaorder}. %

\begin{lemma}[Bounding bias and variance]
	\label[lemma]{lemma:bias_order}
	Under \Cref{assum:no_point_mass,assum:malpha_finite,assum:rho_bound}, let $I_\alpha (\kappa )$ and $\eta$ be as in \Cref{lemma:alphaorder}. Set $c_w:=\sup_{t\geq0}tw(t)$. It holds that
	\begin{gather*}
    \sup_{\alpha\in I_\alpha (\kappa )}\lvert  b(\alpha,\kappa )\rvert \le \inf_{q \in (1,p]} \nu_q \left(\frac{\overline{c}}{n} \ln \frac{4}{\delta}+(1+2\xi )\varepsilon  \right)^{1-\frac{1}{q}} +  \lvert \kappa -\mu  \rvert  \left(\frac{\overline{c}}{n} \ln \frac{4}{\delta}+(1+2\xi )\eps \right), \\
		m(\underline{\alpha}(\kappa))\leq\inf_{q \in (1,p]} c_w \left(\nu_q+\lvert \kappa -\mu  \rvert  \right)\left(\frac{\underline{c} }{n}\ln\frac{4}{\delta}+\frac{\xi}{7}\varepsilon \right)^{-\frac{1}{q}},\\
\| F_\kappa \|_{L_2(P)}^2 \leq \inf_{1<q\le 2\wedge p }2\nu_q^q\left( m(\underline{\alpha}(\kappa))^{2-q}+\lvert \kappa -\mu  \rvert  ^{2-q}\right)+2 \lvert \kappa -\mu  \rvert^2 \left(\frac{\overline{c} \ln \frac{4}{\delta } }{n}+(1+2\xi )\varepsilon \right). 
	\end{gather*}
\end{lemma}
The inequalities for $\sup_{(\alpha ,\kappa )\in \Lambda}\lvert  b(\alpha,\kappa )\rvert $ and $\lVert F\rVert_{L_2(P)}$ follow from the implicit definitions of $\underline{\alpha}(\kappa),\overline{\alpha}(\kappa)$. Since $m(\underline{\alpha}(\kappa))$ depends directly on $\underline{\alpha}(\kappa)$ through the identity $m(\underline{\alpha}(\kappa))=c_w/\underline{\alpha}(\kappa)$ (see \eqref{eq:def_malpha}), we have to explicitly lower-bound $\underline{\alpha}(\kappa) $. \Cref{assum:rho_bound} enables us to use the implicit definition of $\underline{\alpha}(\kappa)$ through $w$ in order to obtain the desired explicit bound. See \Cref{appendix:proofs} for details.

The bounds presented thus far refer to the estimator \( \widehat{\mu}  \) unaffected by adversarial contamination. The following lemma shows that the error introduced by contamination given by \( \lvert \widehat{\mu}^\eps -\widehat{\mu}   \rvert  \) is well-controlled.
\begin{lemma}[Bounding the error introduced by contamination]\label[lemma]{lemma:contamination_error} Let any \( \kappa \in \R , \eta \in (0,n), \varepsilon \in [0,1/2) \). Consider any sample \( X_{1:n} \) and \( X_{1:n}^\varepsilon \in \sA(X_{1:n},\varepsilon ) \). Then, \(   \lvert \widehat{\mu}^\eps(\kappa ) - \widehat{\mu}(\kappa )   \rvert  \) is at most
\[
\left(\eps+\frac{2n-\sum_{i=1}^n w(\widehat{\alpha }^\varepsilon (\kappa )\lvert X_i^\varepsilon -\kappa  \rvert  )-\sum_{i=1}^n w(\widehat{\alpha }(\kappa ) \lvert X_i -\kappa  \rvert  )}{n}\right) \left(m(\widehat{\alpha}(\kappa  ) )+m(\widehat{\alpha }^\eps(\kappa  )  )\right) .
\]  
\end{lemma}
\begin{proof}
We will ommit the dependence on \( \kappa  \) throughout the proof to simplify notation. The proof follows directly by the definitions of \( \widehat{\mu}  \), \( \widehat{\mu}^\eps  \), \( \widehat{\alpha}  \), and \( \widehat{\alpha}^\eps   \):
\begin{align*}
  \lvert \widehat{\mu}^\eps -\widehat{\mu}  \rvert &= \left\lvert\frac{1}{n}\sum_{i=1}^n (X_i-\kappa)w(\widehat{\alpha}\lvert X_i -\kappa  \rvert )-(X_i^\eps -\kappa  )w (\widehat{\alpha}^\eps \lvert X_i^\eps -\kappa  \rvert  ) \right\rvert \\
&\le  \eps (m(\widehat{\alpha} )+m(\widehat{\alpha}^\eps  )) +\frac{1}{n}\sum_{i = 1 }^n \lvert X_i-\kappa\rvert\lvert w(\widehat{\alpha}\lvert X_i -\kappa  \rvert )-w (\widehat{\alpha}^\eps \lvert X_i -\kappa  \rvert  )\rvert.
\end{align*}
In the last inequality, we used the definition \( \sup_{x\in \R}\lvert x w (\alpha x) \rvert = m(\alpha )  \) to bound the summands corresponding to the contaminated points. For the remaining term, notice that, for any \( a,b \in [0,1] \), it holds that \( \lvert a - b \rvert \le  a (1-b) + b(1-a) \). Therefore \( \sum_{i = 1 }^n \lvert X_i-\kappa\rvert\lvert w(\widehat{\alpha}\lvert X_i -\kappa  \rvert )-w (\widehat{\alpha}^\eps \lvert X_i -\kappa  \rvert  )\rvert  \) is upper-bounded by \begin{equation*}
\sum_{i = 1 }^n\lvert X_i-\kappa\rvert w(\widehat{\alpha}\lvert X_i -\kappa  \rvert )(1-w (\widehat{\alpha}^\eps \lvert X_i^\eps  -\kappa  \rvert  ))+ \sum_{i = 1 }^n \lvert X_i-\kappa\rvert w(\widehat{\alpha }^\eps\lvert X_i -\kappa  \rvert )(1-w (\widehat{\alpha} \lvert X_i  -\kappa  \rvert  )).
\end{equation*}
Bounding $\lvert X_i-\kappa\rvert w(\widehat{\alpha}\lvert X_i -\kappa  \rvert ) \leq m(\widehat{\alpha} )$ and $\lvert X_i-\kappa\rvert w(\widehat{\alpha }^\eps\lvert X_i -\kappa  \rvert ) \leq m(\widehat{\alpha}^\eps)$ for all $i$ and combining the resulting bounds, we obtain the desired result.
\end{proof}

Under \Cref{assum:no_point_mass}, \Cref{lemma:rho_bounds} yields, for all \( \kappa \in \R \) and \( X_{1:n}^\varepsilon \in \sA(X_{1:n}, \varepsilon ) \) with probability \( 1 \), \[
n-\eta-1-\left\lfloor n \varepsilon  \right\rfloor \leq \sum_{i=1}^n w\left(\widehat{\alpha}\left(\kappa; X_{1:n}^\varepsilon \right) |X_i^\varepsilon  - \kappa|\right),
\] 
and so \[
  \lvert \widehat{\mu}^\eps(\kappa ) - \widehat{\mu}(\kappa )   \rvert \le \left(\eps+\frac{2 \eta + 2 + 2n \varepsilon }{n}\right) \left(m(\widehat{\alpha}(\kappa  ) )+m(\widehat{\alpha }^\eps(\kappa  )  )\right) .
\]  
In addition, when \( \underline{\alpha }(\kappa ) < \inf_{X_{1:n}^\varepsilon \in \sA(X_{1:n},\varepsilon )  }\widehat{\alpha }^\eps (\kappa )\le \widehat{\alpha}(\kappa )    \), it holds that the error \( \lvert \widehat{\mu }^\varepsilon  - \widehat{\mu }   \rvert  \) is bounded by \(  \left(6\eps + \frac{4\eta+4}{n}\right)m(\underline{\alpha }(\kappa ))  \). Then, for \( \eta  \) as in \Cref{lemma:alphaorder}, we may use the bounds on \( m(\underline{\alpha }(\kappa ))  \) from \Cref{lemma:bias_order} to control the error introduced by contamination, concluding that, with high probability, \[
  \sup_{X_{1:n}^\varepsilon \in \sA(X_{1:n},\varepsilon )}\lvert \widehat{\mu}^\eps -\widehat{\mu}   \rvert \le \left( (10+4\xi) \eps + \frac{8}{n} \ln \frac{4}{\delta }\right)c_w (\nu _p + \lvert \kappa -\mu  \rvert )\left(\frac{\underline{c}}{n}\ln \frac{4}{\delta }+\frac{\xi}{7}\varepsilon \right)^{-\frac{1}{p}}. 
\] 
\Cref{thm:shrinkage_concentration_adv} then follows by combining the previous results. Details are presented in \Cref{appendix:proofs}.

\section{Properties of shrinkage estimators}\label{sec:properties}
In this section, we formally state and prove the desirable properties of the shrinkage estimators \( \widehat{\mu}(\widehat{\kappa } )  \) presented in \Cref{sec:estimator_properties}.
\begin{proposition}[Affine-equivariance]
		Assume \( \widehat{\kappa} \) depends on the sample \( Y_{1:m} \) and is affine-equivariant, that is, \( \widehat{\kappa }(\lambda Y_{1:m}+t )= \lambda \widehat{\kappa }(Y_{1:m})+t \), for every \( \lambda \neq 0, t \in
    \R\), where \( \lambda X_{1:n}+t = (\lambda X_1+t,\ldots, \lambda X_n +t ) \). Denote by \( \widehat{\mu }(X_{1:n},Y_{1:m})=\widehat{\mu }(\widehat{\kappa }(Y_{1:m}); X_{1:n} )  \) the shrinkage estimator from \eqref{eq:def_estimator} applied to the base estimate \( \widehat{\kappa }(Y_{1:m}) \) and the sample \( X_{1:n} \). Then, \( \widehat{\mu }(\lambda (X_{1:n},Y_{1:m})+t)=\lambda \widehat{\mu }(X_{1:n},Y_{1:m})+t \).
	\end{proposition}
	\begin{proof} Note that
		\begin{align*}
      \widehat{\mu }&(\lambda (X_{1:n},Y_{1:m})+t )  =\widehat{\kappa }(\lambda Y_{1:m}+t )\\&\hspace{1cm}+\frac{1}{n}\sum_{i=1}^n (\lambda X_i+t -\widehat{\kappa }(\lambda Y_{1:m}+t ))w (\widehat{\alpha }(\lambda (X_{1:n},Y_{1:m})+t )\lvert \lambda X_i+t -\widehat{\kappa }(\lambda Y_{1:m}+t ) \rvert ) \\
                                                   & =\lambda \widehat{\kappa }(Y_{1:m})+t +\frac{\lambda }{n}\sum_{i=1}^n (X_i -\widehat{\kappa }(Y_{1:m}))w (\widehat{\alpha }(\lambda (X_{1:n},Y_{1:m})+t )\lvert \lambda  \rvert  \lvert X_i-\widehat{\kappa }(Y_{1:m}) \rvert )                             \\
                                                   & =\lambda\left( \widehat{\kappa }(Y_{1:m}) +\frac{1 }{n}\sum_{i=1}^n (X_i -\widehat{\kappa }(Y_{1:m}))w (\widehat{\alpha }(\lambda (X_{1:n},Y_{1:m})+t )\lvert \lambda  \rvert \lvert X_i-\widehat{\kappa }(Y_{1:m}) \rvert )\right) +t.
		\end{align*}
    It is then enough to check that \( \lvert \lambda  \rvert \widehat{\alpha }(\lambda (X_{1:n},Y_{1:m})+t )  = \widehat{\alpha }(X_{1:n},Y_{1:m}) \). Recall that \( \widehat{\alpha }(X_{1:n},Y_{1:m}) \) is such that \begin{align*}
      \widehat{\alpha }(X_{1:n},Y_{1:m}) & =\inf \left\{\alpha >0: \sum_{i=1}^n w (\alpha \lvert X_i-\widehat{\kappa }(Y_{1:m}) \rvert )\le n-\eta \right\} \\
			                           & =\inf \left\{\alpha >0: \sum_{i=1}^n w \left(\frac{\alpha}{\lvert \lambda  \rvert } \lvert \lambda X_i+t -\widehat{\kappa }(\lambda Y_{1:m}+t ) \rvert \right)\le n-\eta \right\} \\
                                 & =\lvert \lambda  \rvert\inf \left\{\alpha >0: \sum_{i=1}^n w \left(\alpha \lvert \lambda X_i+t -\widehat{\kappa }(\lambda Y_{1:m}+t ) \rvert \right)\le n-\eta \right\}\\
                                 &=\lvert \lambda  \rvert\widehat{\alpha }(\lambda (X_{1:n},Y_{1:m})+t). \qedhere
		\end{align*}
	\end{proof}
  \begin{proposition}[Distributional limit of shrinkage estimators]\label[proposition]{prop:distributional_limit} Let $\eta$ be fixed and take $ \widehat{\mu}_n=\widehat{\mu }(\widehat{\kappa }_n;X_{1:n} ) $ as in \eqref{eq:def_estimator}. Under \Cref{assum:no_point_mass}, further assume that $\nu_{p+\gamma } < \infty$ for some $p > 1,\gamma  > 0$, that $(\widehat{\kappa}_n-\mu) n^{-\frac{1}{p}}\to 0$ in probability and that \( \widehat{\kappa }_n  \) is independent from \( X_{1:n} \). Then, $n^{1-\frac{1}{p}} (\widehat{\mu}_n-\mu ) \overset{d}{\to} D$ for some distribution $D$ if, and only if, $n^{1-\frac{1}{p}} (\overline{X}-\mu ) \overset{d}{\to} D$.
	\end{proposition}
	\begin{proof}
		Consider \( n^{1-\frac{1}{p}}(\widehat{\mu}_n - \overline{X}) \). We show that it converges in probability to zero. We start by writing \( \lvert n^{1-\frac{1}{p}}(\widehat{\mu}_n -\overline{X} ) \rvert \) as \begin{align*}
      &\left\lvert n^{-\frac{1}{p}}\left(n-\sum_{i=1}^nw(\widehat{\alpha }_n \lvert X_i - \widehat{\kappa }_n  \rvert  )\right)(\widehat{\kappa }_n-\mu )+n^{-\frac{1}{p}} \sum_{i=1}^n (X_i-\mu )(w (\widehat{\alpha}_n \lvert X_i-\widehat{\kappa }_n \rvert )-1) \right\rvert \\
   & \le (\eta+1) n^{-\frac{1}{p}} \lvert \widehat{\kappa }_n-\mu \rvert + (\eta+1) n^{-\frac{1}{p}}\max_{1\le i\le n} \lvert X_i -\mu \rvert,
		\end{align*}
    where \( \widehat{\alpha }_n=\widehat{\alpha }(\widehat{\kappa }_n;X_{1:n} )   \) and we used \Cref{lemma:rho_bounds} with \( \varepsilon = 0 \) to attain the last inequality (which holds almost surely).
		To finish the proof, it is enough to show that \( n^{-\frac{1}{p}}\max_{1\le i\le n} \lvert X_i -\mu \rvert \) converges in probability to zero. Indeed,
		\begin{align*}
      \ \Pr\left[ n^{-\frac{1}{p}}\max_{1\le i\le n} \lvert X_i -\mu \rvert > t \right] &= 1-\left(1-\Pr\left[\lvert X_1 -\mu \rvert > tn^{\frac{1}{p}}\right]\right)^n \\
 & = 1-\left(1-\Pr\left[\lvert X_1 -\mu \rvert^{p+\gamma  } > t^{p+\gamma }n^{\frac{p+\gamma }{p}}\right]\right)^n \\
        & \le 1-\left(1-\frac{\nu _{p+\gamma }^{p+\gamma }}{t^{p+\gamma }n^{1+\frac{\gamma}{p}}}\right)^n \underset{n\to \infty}{\to } 0. \qedhere
		\end{align*}
	\end{proof}

  \begin{proposition}[\( \widehat{\mu }  \) preserves breakdown point]
    Under \Cref{assum:rho_bound}, let \(  \eta \ge  n\varepsilon^\star (\widehat{\kappa }, Y_{1:m} )  \). Then \(  \varepsilon^\star(\widehat{\mu }(\widehat{\kappa } ),(X_{1:n},Y_{1:m}) ) \ge \varepsilon^\star(\widehat{\kappa},Y_{1:m} ) \). 
    \begin{proof}
    We need to show that, when \( \varepsilon < \varepsilon^\star(\widehat{\kappa }, Y_{1:m} )  \),  \[
      \sup_{\substack{X_{1:n}^\varepsilon \in \sA(X_{1:n}, \varepsilon )\\Y_{1:m}^\varepsilon \in \sA(Y_{1:m}, \varepsilon )}} \lvert \widehat{\mu}(\widehat{\kappa }(Y_{1:m}^\varepsilon); X_{1:n}^\varepsilon  ) - \widehat{\mu }(\widehat{\kappa}(Y_{1:m});X_{1:n})   \rvert < \infty.
    \] 
    Denote \( \widehat{\kappa }^\varepsilon = \widehat{\kappa }(Y_{1:m})   \) and \( \widehat{\kappa }=\widehat{\kappa }(Y^\varepsilon _{1:m})   \). We use that \( \lvert \widehat{\mu }(\widehat{\kappa }^\varepsilon ;X_{1:n}^\varepsilon  ) - \widehat{\mu }(\widehat{\kappa };X_{1:n} )   \rvert \le \lvert \widehat{\mu }(\widehat{\kappa }^\varepsilon; X_{1:n}^\varepsilon  ) - \widehat{\mu }(\widehat{\kappa }^\varepsilon ; X_{1:n} )   \rvert +  \lvert \widehat{\mu }(\widehat{\kappa }^\varepsilon; X_{1:n}  ) - \widehat{\mu }(\widehat{\kappa } ; X_{1:n} )   \rvert  \). For the first part, we use \Cref{lemma:contamination_error} and the fact that \( w \) is non-negative to attain the inequality \[
      \lvert \widehat{\mu }(\widehat{\kappa }^\varepsilon ;X_{1:n}^\varepsilon  ) -\widehat{\mu }(\widehat{\kappa }^\varepsilon ; X_{1:n} )  \rvert  \le \left(\eps+2\right) \left(m(\widehat{\alpha}(\widehat{\kappa } ^\varepsilon; X_{1:n}   ) )+m(\widehat{\alpha }(\widehat{\kappa }^\varepsilon ; X_{1:n}^\varepsilon    )  )\right)
    \] 
    We then use \Cref{assum:rho_bound} to conclude that, for any \( \kappa \in \R \),  
    \begin{align*}
      n-\eta &\ge  \sum_{i=1}^nw(\widehat{\alpha }(\kappa; X^\varepsilon _{1:n})\lvert X_i^\varepsilon  - \kappa   \rvert ) \ge  \sum_{i=1}^nw(\widehat{\alpha }(\kappa; X_{1:n}^\varepsilon )\lvert X_i - \kappa   \rvert ) - n\varepsilon \\&\ge  n(1-\varepsilon )- \widehat{\alpha }(\kappa ; X^\varepsilon _{1:n} )^p\sum_{i=1}^n \lvert X_i - \kappa   \rvert^p.
    \end{align*}
    which implies the bound
    \[
      (\widehat{\alpha } (\widehat{\kappa } ^\varepsilon ; X^\varepsilon _{1:n}) )^{-p} \le (\eta-n\varepsilon)^{-1}\sum_{i=1}^n \lvert X_i- \widehat{\kappa }^\varepsilon    \rvert^p \le 2^{p-1}(\eta -n \varepsilon )^{-1}\left(\sum_{i=1}^n \lvert X_i - \widehat{\kappa }  \rvert^p + n\lvert \widehat{\kappa } -\widehat{\kappa }^\varepsilon    \rvert^p  \right).
    \]
    It is then immediate that
    \[
      \sup_{\substack{X_{1:n}^\varepsilon \in \sA(X_{1:n},\varepsilon )\\Y_{1:m}^\varepsilon \in \sA(Y_{1:m},\varepsilon )}} (\widehat{\alpha }(\widehat{\kappa }^\varepsilon ; X_{1:n}^\varepsilon  ) )^{-1} < \infty,
    \]
    because \( \widehat{\kappa}\) has breakdown point greater than \( \varepsilon  \) and \( \eta \ge n \varepsilon ^\star(\widehat{\kappa }, Y_{1:m} )>n \varepsilon  \) by assumption.   
    It follows that
    \[
      \sup_{\substack{X_{1:n}^\varepsilon \in \sA(X_{1:n},\varepsilon )\\Y_{1:m}^\varepsilon \in \sA(Y_{1:m},\varepsilon )}} m(\widehat{\alpha }^\varepsilon (\widehat{\kappa }^\varepsilon  )) <\infty,
    \]
    since \( m(\alpha ) = c_w\alpha ^{-1} \). The bound for \( m(\widehat{\alpha }(\widehat{\kappa }^\varepsilon ; X_{1:n} ) ) \) is obtained similarly. 

    We now focus on the term \( \lvert \widehat{\mu }(\widehat{\kappa }^\varepsilon; X_{1:n}  ) - \widehat{\mu }(\widehat{\kappa } ; X_{1:n} )   \rvert  \). Its supremum is easily bounded: 
    \[
      \sup_{\substack{X_{1:n}^\varepsilon \in \sA(X_{1:n},\varepsilon )\\Y_{1:m}^\varepsilon \in \sA(Y_{1:m},\varepsilon )}}\lvert \widehat{\mu }(\widehat{\kappa }^\varepsilon; X_{1:n}  ) - \widehat{\mu }(\widehat{\kappa } ; X_{1:n} )   \rvert \le \sup_{\substack{X_{1:n}^\varepsilon \in \sA(X_{1:n},\varepsilon )\\Y_{1:m}^\varepsilon \in \sA(Y_{1:m},\varepsilon )}} \lvert \widehat{\kappa } - \widehat{\kappa }^\varepsilon    \rvert + 2n \max_{1 \le i \le n} \lvert X_i \rvert  < \infty.
    \] 
    The proof is complete by collecting the bounds. \qedhere
    \end{proof}
  \end{proposition}
\section{Examples}
We now explore the consequences of \Cref{thm:shrinkage_concentration_adv} for important examples of shrinkage estimators with specific choices of base estimators that yield mean-estimators with optimal concentration rates. We will generally consider estimates computed on an independent sample \( Y_{1:m} \sim P^{\otimes m } \). To account for possible adversarial contamination (with the same level \( \varepsilon  \)) on the sample \( Y_{1:m} \), we allow \( \widehat{\kappa } \) to take the form of \( \widehat{\kappa }(f(Y_{1:m}, \varepsilon ))  \), where \( f\colon \R^m \times [0,1]\to  \R^m \) is a Borel-measurable function such that \( f(y_{1:m}, \gamma ) \in \sA(y_{1:m}, \gamma ) \) for every possible sample \( y_{1:m} \in \R^m \) and contamination level \( \gamma \in [0,1] \). Notice that \( f(y_{1:m}, 0) = y_{1:m} \). These technicalities are not necessary in the particular case of \Cref{ex:exact_tm}, where we assume that the base estimator is dependent on the same contaminated sample \( X^\varepsilon _{1:n} \).

We also provide examples of shrinkage functions \( w \) that satisfy \Cref{assum:malpha_finite,assum:rho_bound} and that can be combined with any base estimator \( \widehat{\kappa }  \) to yield optimal bounds in terms of order. 
\subsection{Examples of base estimators \( \widehat{\kappa }  \)}\label{sec:examples_kappa}

\begin{example}[Empirical mean]\label{ex:empirical_mean_sub_gaussian}
  Assume that \( \varepsilon = 0 \). Let \( \widehat{\kappa}=\widehat{\kappa }(Y_{1:m})=\overline{Y}   \). Recall that the empirical mean $\overline{Y}$ satisfies $R_{\overline{Y}}(\delta) \leq \nu_2 \sqrt{\frac{1}{\delta m}} $ for \( \delta \in (0,1) \). Then, \eqref{eq:conclusion_concentration2} holds for \( p=2 \), \( C=1 \), and \( \delta_{\textup{min}}= \frac{4}{m} \). Notice that \( \varepsilon = 0 \) is necessary in this case, since the empirical mean is not robust to adversarial contamination. In fact, \( \sup_{X_{1:n}^\varepsilon \in \sA(X_{1:n},\varepsilon )}\widehat{\mu}^{\varepsilon }\left(\overline{Y^\varepsilon };X_{1:n}^\varepsilon \right) \to \infty \) as any \( Y_i^{\varepsilon}\to \infty \), rendering this estimator useless under adversarial contamination.
\end{example}

\begin{example}[Empirical quantiles]\label{ex:empirical_quantiles}
  Let \( \widehat{\kappa }= \widehat{\kappa }(f(Y_{1:n},\varepsilon ))= \widehat{Q}_\gamma(f(Y_{1:n},\varepsilon )) \) be an empirical quantile of a possibly contaminated version of the sample \( Y_{1:n} \). It is easy to see that \( \widehat{Q}_{\gamma -\varepsilon}(Y_{1:n})\le \widehat{Q}_{\gamma }(Y_{1:n}^{\varepsilon}) \le \widehat{Q}_{\gamma +\varepsilon}(Y_{1:n})   \) when \( \gamma \in (\varepsilon,1-\varepsilon ) \) and \( Y_{1:m}^\varepsilon \in \sA(Y_{1:m},\varepsilon ) \). By the proof of \Cref{prop:empirical_quantiles}, which provides concentration bounds for empirical quantiles, \( R_{\widehat{\kappa }}  (\delta )\le \nu _p\left(\frac{2}{\gamma \wedge (1-\gamma )-\varepsilon  }\right) \) when \( \delta \ge \exp\left\{-\frac{m}{16}(\gamma \wedge (1-\gamma )-\varepsilon)^2 \right\} \), for every \( p>1 \). In particular, the sample median satisfies the condition \( R_{\widehat{\kappa}}(\delta)\le 4\nu _p(1-2\varepsilon )^{-1}    \) for any \( \varepsilon \in [0, 1/2) \). Then, if \( \delta \ge 4\exp\left\{-\frac{m}{16}\left(1 / 2-\varepsilon\right)^2 \right\}\vee 4\exp \left\{-\frac{1-(1+2\xi )\varepsilon }{\overline{c}}n\right\}   \), with probability at least \( 1-\delta  \), 
  \begin{equation}\label{eq:median_concentration}
    \sup_{X_{1:n}^\varepsilon \in \sA(X_{1:n},\varepsilon )}\lvert \widehat{\mu }^\varepsilon-\mu   \rvert  \le C_{w,\xi ,\varepsilon} \left\{ \inf_{1<q\le 2 \wedge p} \nu _q\left(\frac{1}{n} \ln \frac{4}{\delta }\right)^{1-\frac{1}{q}}+ \inf_{1<q \le p} \nu_q \varepsilon^{1-\frac{1}{q}}\right\},  
  \end{equation}  
  for some \( C_{w,\xi ,\varepsilon }>0 \) depending only on \( w \), \( \xi  \), and \( \varepsilon  \).
\end{example}

\begin{example}[Trimmed mean]\label[example]{ex:exact_tm}
  In order to recover the trimmed mean estimator given by \eqref{eq:tm}, one may choose \( w(t)=\mathbf{1}_{t<1} \). For the sake of simplicity, assume that \( \eta  \) is an even integer. Then, take \( \widehat{\kappa }=\frac{X_{(\eta /2 + 1)} + X_{(n- \eta /2)}}{2}  \), where \( X_{(i)} \) denotes the \( i \)-th order statistic of the sample \( X_{1:n} \). If the same sample is used for both the base estimate and the shrinkage estimator, then it is possible to show \( \widehat{\mu }_W(\widehat{\kappa }; X_{1:n})\) is the trimmed mean (since the smallest interval centered on \( \widehat{\kappa }  \) that contains at least \( n-\eta  \) sample points is exactly \( [X_{(\eta /2 + 1)}, X_{(n-\eta /2 )}] \)). Accordingly, \Cref{thm:trimmed_concentration_adv} generalizes \Cref{thm:shrinkage_concentration_adv} by allowing $\widehat{\kappa}$ to be calculated on the same sample as $\widehat{\mu}$ and with a guarantee under a lighter dependence on $R_{\widehat{\kappa}}$ (see \Cref{remark:worse_kappa}).
  Through \Cref{thm:trimmed_concentration_adv}, it is possible to obtain the optimal concentration bounds for the trimmed mean since $R_{\widehat{\kappa}}(\delta / 2) \le \nu_p \left( \frac{1}{12n} \ln \frac{4}{\delta } + \frac{\xi}{4} \varepsilon \right)^{-1 / p} $ for \( p>1 \)  as \Cref{prop:exact_tm} presented in the appendix shows.   %

The combined application of \Cref{thm:trimmed_concentration_adv} and \Cref{prop:exact_tm} implies that the estimator \( \widehat{\mu }^\varepsilon:=\widehat{\mu }_W(\widehat{\kappa }(X_{1:n}^\varepsilon ); X_{1:n}^\varepsilon  )   \) attains optimal concentration. More specifically, under the assumptions of \Cref{thm:trimmed_concentration_adv}, fix \( \varepsilon \in [0,1 / 2) \) and take \( \eta  \) and \( \xi  \) as in the statement of \Cref{prop:exact_tm}. For $\delta \ge 4\exp\left\{-n(6-3\xi \varepsilon )\right\}\vee 4\exp \left\{-\frac{c_0-2\cdot(1+\xi )\varepsilon }{\overline{c}}n\right\},$ with probability at least \( 1- \frac{3 \delta }{2} \) 
  \begin{equation}\label{eq:exact_tm_concentration}
    \sup_{X_{1:n}^\varepsilon \in \sA(X_{1:n},\varepsilon )}\lvert \widehat{\mu }^\varepsilon-\mu   \rvert  \le C_{\xi ,\varepsilon} \left\{ \inf_{1<q\le 2 \wedge p} \nu _q\left(\frac{1}{n} \ln \frac{4}{\delta }\right)^{1-\frac{1}{q}}+ \inf_{1<q \le p} \nu_q \varepsilon^{1-\frac{1}{q}}\right\},  
  \end{equation}  
  for some \( C_{\xi ,\varepsilon }>0 \) depending only on \( \xi  \), and \( \varepsilon  \), recovering the optimal concentration rates for the trimmed mean under adversarial contamination and infinite variance.
\end{example}

\begin{example}[Constant base estimate]\label{ex:constant_base_estimate}
  When \( \widehat{\kappa }=\lambda \) is constant-valued, \( R_{\widehat{\kappa  } }(\delta )=\lvert \lambda -\mu  \rvert  \). In particular, if \( \lvert \lambda - \mu  \rvert \le \nu _p  \) for all \( p>1 \), then, for \( \delta \ge 4\exp \left\{-\frac{1-(1+2\xi )\varepsilon }{\overline{c}}n\right\} \), with probability at least \( 1-\delta  \), \[
 \sup_{X_{1:n}^\varepsilon \in \sA(X_{1:n},\varepsilon )}   \lvert \widehat{\mu }^\varepsilon -\mu   \rvert  \le C_{w,\xi } \left\{ \inf_{1<q\le 2\wedge p} \nu _q\left(\frac{1}{n} \ln \frac{4}{\delta }\right)^{1-\frac{1}{q}}+ \inf_{1<q \le p} \nu_q \varepsilon^{1-\frac{1}{q}}\right\},
  \]
  for some \( C_{w,\xi}>0 \) depending only on \( w \) and \( \xi  \). 
\end{example}

\begin{example}[Optimal \( \delta  \)-dependent mean-estimators]\label{ex:optimal_base_estimator}
  Assume that the \( \delta  \)-dependent base estimator \( \widehat{\kappa} \) is optimal in the sense that there exist \( L_{\widehat{\kappa } }, \delta_{\textup{min}}(\widehat{\kappa },m)>0 \) such that 
\[
  \Pr\left[ \sup_{Y_{1:m}^\varepsilon \in \sA(Y_{1:m},\varepsilon )}|\widehat{\kappa }(Y_{1:m}^\varepsilon )  -\mu |  > L_{\widehat{\kappa } }\left\{ \inf_{q\in (1,2]}\nu_q \left(\frac{1}{m}\ln\left(1+\frac{1}{\delta}\right)\right)^{1-\frac{1}{q}}+\inf_{q>1}\nu_q\eps^{1-\frac{1}{q}}\right\} \right] \leq \delta,
\] 
for all \( \delta \ge \delta_{\textup{min}}(\widehat{\kappa },m ) \). Then \( \widehat{\kappa }=\widehat{\kappa }(f(Y_{1:m},\varepsilon ))   \) is such that
\[
  R_{\widehat{\kappa } }(\delta ) \le L_{\widehat{\kappa } }\left\{ \inf_{q\in (1,2]}\nu_q \left(\frac{1}{m}\ln\left(1+\frac{1}{\delta}\right)\right)^{1-\frac{1}{q}}+\inf_{q>1}\nu_q\eps^{1-\frac{1}{q}}\right\},\]
  and we may attain similar bounds to \eqref{eq:median_concentration}. 
\end{example}
\subsection{Examples of shrinkage functions \( w \)}\label{sec:examples}
In this section, we show that suitable choices of $w : [0,\infty) \to [0,1]$ in \eqref{eq:def_estimator} lead to known robust estimators or variants thereof. In particular, the theoretical results in \Cref{sec:main_result} can be directly applied to establish their concentration bounds. \Cref{assum:malpha_finite,assum:rho_bound} are shown to hold in each example below.

\begin{example}[The generalized trimmed mean]\label{ex:rho_atm} The choice of $w(t) = \mathbf{1}_{t < 1}$ yields
  \[ \widehat{\mu} = \frac{\lfloor \eta \rfloor}{n}\widehat{\kappa} + \frac{1}{n} \sum_{i=1}^{n - \lfloor \eta \rfloor} X_{(i)},    \]
where $X_{(1)}, \ldots, X_{(n)}$ are ordered such that $|X_{(1)} - \widehat{\kappa}| \leq |X_{(2)} - \widehat{\kappa}| \leq \cdots \leq |X_{(n)} - \widehat{\kappa}|$. This estimator discards the sample points further from the base estimate $\widehat{\kappa}$. Recall that, for a particular choice of \( \widehat{\kappa }  \), \( \widehat{\mu} (\widehat{\kappa } )  \) is exactly the trimmed mean (\Cref{ex:exact_tm}). \Cref{assum:malpha_finite} follows from $\sup_{t \geq 0} t w(t) = 1 < \infty$. Regardless of which $p$-th moment is available, \Cref{assum:rho_bound} is satisfied. Thus, the estimator achieves optimal concentration in an appropriate regime of \( \delta  \) depending on the concentration of \( \widehat{\kappa }  \) (see \Cref{thm:shrinkage_concentration_adv}). The classical trimmed mean \citep{oliveira2025finite} attains sub-Gaussian concentration rates and optimal bounds for the infinite variance scenario. In these cases it has been established that $\delta_{\textup{min}}$ is of order $e^{-c^\prime n}$ for some $c^\prime>0$, which is optimal. The estimator proposed in \eqref{eq:def_estimator_weighted} in this case takes the form of \[
  \widehat{\mu }_W = \frac{1}{n-\left\lfloor \eta  \right\rfloor }\sum_{i=1}^{n-\left\lfloor \eta  \right\rfloor} X_{(i)},
  \] 
  effectively excluding \( \widehat{\kappa }  \) from its explicit formula and using it only for the reweighting scheme. %
\end{example}

\begin{example}[Winsorized mean]\label{ex:rho_win} If $w(t) = 1 \wedge t^{-1}$, then
	\begin{equation}
		\label{eq:def_estimator_winsorized}
		\widehat{\mu} = \widehat{\kappa} + \frac{1}{n} \left( M(n_+ - n_-) + \sum_{|X_i - \widehat{\kappa}| \leq M} X_i - \widehat{\kappa}\right) \text{ and } \widehat{\alpha}^{-1} = M,
	\end{equation}
	where $n_+ = \left|\left\{ i : X_i - \widehat{\kappa} > M \right\}\right|$, $n_- = \left|\left\{ i : X_i - \widehat{\kappa} < -M \right\}\right|$ and $M$ is the (unique) solution of
	\begin{equation}
		\label{eq:def_alpha_winsorized}
		n - \eta = \left| \left\{ i : |X_i - \widehat{\kappa}| \leq M \right\} \right| +  \sum_{
			|X_i - \widehat{\kappa}| > M } \frac{M}{|X_i - \widehat{\kappa}|}.
	\end{equation}
	The resulting estimator is the same regardless of $p$. It corresponds to the Winsorized mean (\cite{huber1964robust}), but with the threshold $M$ chosen via \eqref{eq:def_alpha_winsorized}. \Cref{assum:rho_bound} follows directly for all $p\in(1,2]$ and, since $\sup_{t\geq0} tw(t) = 1 < \infty$, \Cref{assum:malpha_finite} holds. This function attains the slowest possible decay in the sense that $\limsup_{t\to\infty}tw(t)\in(0,\infty)$ and any decay of a slower order would violate \Cref{assum:malpha_finite}.
\end{example}

\begin{example}[Lee and Valiant's estimator \cite{lee2022optimal}]\label{ex:rho_lv}
	Taking $w(t)=(1-t^p)_+$ yields the estimator $\widehat{\mu}$ such that $(\widehat{\alpha},\widehat{\mu})$ solves
	\[\begin{cases}
			\sum_{i=1}^n\min(\alpha^p\lvert X_i-\widehat{\kappa}\rvert^p,1)=\eta, \\
			\mu=\widehat{\kappa}+\frac{1}{n}\sum_{i=1}^n(X_i-\widehat{\kappa})(1-\min(\alpha^p\lvert X_i-\widehat{\kappa}\rvert^p,1)).
		\end{cases} \]
    \cite{lee2022optimal} introduced this estimator for the case $p=2$ and showed that if $\widehat{\kappa}$ is sub-Gaussian, then $\widehat{\mu}$ is sub-Gaussian with asymptotically optimal constants as $\big(\frac{1}{n}\ln\frac{1}{\delta}, \delta\big) \to (0,0)$. In our framework, this estimator satisfies \Cref{assum:rho_bound} with equality and \Cref{assum:malpha_finite} follows for every $p \in (1, 2]$ since
    \[
      \sup_{t\geq0}tw(t)=\sup_{t\geq0}t(1-t^p)_+=p(p+1)^{-(p+1)/p}<\infty.
    \]
    Thus, our main result ensures that \( \widehat{\mu }  \) attains sub-Gaussian concentration even under weaker assumptions on the base estimator \( \widehat{\kappa}\): we require only that its error is \( O(1) \) with high probability (see \Cref{ex:empirical_quantiles}). 
\end{example}

Beyond well-studied estimators, our framework also facilitates the search for new ones with possibly different desirable properties. For example, we may look for shrinkage estimators that do not shrink large sample points as aggressively as the trimmed mean but more aggressively than the Winsorized mean. One way to do this is by tuning the $w$ function, as it determines the shrinkage associated with each sample point.

\begin{example}[Fast decay]\label{ex:rho_inv} A possible choice of estimator that shrinks large sample points polynomially would entail choosing $w(t)=(1+t^p)^{-1}$. It satisfies \Cref{assum:malpha_finite,assum:rho_bound} for all $p\in(1,2]$ since $\sup_{t\geq0}t(1+t^p)^{-1}\le 1$ and \( (1+x)^{-1}\ge (1-x)_+ \) for all \( x \in \R \). Notice that this choice of \( w  \) interpolates between \Cref{ex:rho_win} and \Cref{ex:rho_lv} in the sense that $1-x^p\le \frac{1}{1+x^p}\le 1\wedge x^{-1} \text{ for all } x\ge 0.$ In fact, we observe empirically that this interpolation also takes place in terms of performance, as is detailed in the appendix.
\end{example}

The examples above showcase the flexibility of our framework, allowing us to explore a wide range of estimators by tuning the base estimator \( \widehat{\kappa }  \) and the shrinkage function \( w \). For example, although the sample median is generally inconsistent, it often excels as a base estimator. \Cref{sec:experiments} further considers how to make these choices.

\section{Extension to \( \R^d \)}\label{sec:extension_multivariate}

We now discuss an extension of our estimator to the multivariate setting. Let \( X_{1:n} \sim P^{\otimes n}\), where \( P \) is a distribution on \( \R^d \) with mean \( \mu  \) and well-defined covariance matrix \( \Sigma  \). Let \( \eps \in [0,\frac{1}{2}) \). As before, we want to estimate the mean \( \mu  \) based on an adversarially contaminated version of the sample, \( X_{1:n}^\varepsilon  \in \sA(X_{1:n},\varepsilon ) \). 
We define
\begin{align}
  \widehat{\alpha }(\kappa;X_{1:n})=\widehat{\alpha }(\kappa ):=\inf \left\{\alpha > 0: \sum_{i=1}^n w(\alpha \lVert X_i -\kappa  \rVert  ) \le n - \eta  \right\},\label{eq:def_alpha_multivariate}\\   
  \widehat{\mu }(\widehat{\kappa };X_{1:n} )=\widehat{\mu }(\widehat{\kappa } ):= \widehat{\kappa }+ \frac{1}{n}\sum_{i=1}^n (X_i -\widehat{\kappa } )w(\widehat{\alpha }(\widehat{\kappa } ) \lVert X_i -\widehat{\kappa }  \rVert  )\label{eq:def_estimator_multivariate}.
\end{align}
We will require a slight strengthening of \Cref{assum:no_point_mass}, in that we take \( P \) to be absolutely continuous with respect to the Lebesgue measure, though this can be dispensed with if \( w \) is continuous.
Let $\tr(\Sigma)$ and $\|\Sigma\|$ be the trace and operator norm of $\Sigma$, respectively. We extend our main result as follows:
\begin{theorem}\label{thm:multivariate_concentration_adv}
Let $X_{1:n}\sim P^{\otimes n}$, where \( P \) has well-defined covariance matrix \( \Sigma\)  and is absolutely continuous with respect to the Lebesgue measure. Let \( \eps \in [0,\frac{1}{2}) \). Under \Cref{assum:kappa_independent}, set $\eta = \ln\frac{4}{\delta}+(1+\xi )\eps n $ for some \( \xi>0 \) and let \( w \) be as in \Cref{assum:malpha_finite,assum:rho_bound}. Further assume that \( (1+\xi)\eps + \frac{\overline{c}}{n}\ln \frac{4}{\delta } < 1 \). If $\widehat{\mu}^\eps  =  \widehat{\mu}(\widehat{\kappa };X_{1:n}^\eps) $ is as in \eqref{eq:def_estimator_multivariate} and 
\[
  R_{\widehat{\kappa }}(\delta) := \inf\left\{r>0 : \Pr\left[\lVert\widehat{\kappa }- \mu\rVert > r\right] \leq \delta \right\},
\]
then, with probability at least $1-\delta $, for all \( X_{1:n}^\varepsilon \in \sA(X_{1:n},\varepsilon ) \), \( \lVert \widehat{\mu }^\varepsilon -\mu   \rVert  \) is at most  
  \begin{equation}\label{eq:conclusion_concentration_multivariate}
    C^M_{w,\xi }\left\{ \left(\sqrt{\tr(\Sigma)}+ R_{\widehat{\kappa } }\left(\frac{\delta}{4}\right)\right) \sqrt{\frac{1}{n}\ln \frac{4}{\delta } + \varepsilon }  \right\},
  \end{equation}
where \( C^M_{w,\xi } \) is a constant depending only on \( w \) and \( \xi  \).
\end{theorem}
The proof follows essentially the same steps as the proof of \Cref{thm:shrinkage_concentration_adv} and is provided in \Cref{sec:proof_multivariate}. The bound \eqref{eq:conclusion_concentration_multivariate} closely resembles the one obtained by \cite{minsker2015geometric} for the geometric median-of-means. Both estimators have linear time-complexity in $d$ --- i.e., they run in $O(nd)$ --- but are known to yield sub-optimal concentration bounds. The sub-optimality of \eqref{eq:conclusion_concentration_multivariate} is best understood by comparing it to the empirical sample mean $\overline{X}$ computed on Gaussian samples, which achieves the optimal high-probability bound
\begin{equation}\label{eq:multivariate_bound_optimal}
\left\lVert \overline{X} - \mu \right\rVert \le \sqrt{\frac{\tr(\Sigma)}{n}} + \sqrt{\frac{2 \lVert \Sigma\rVert \ln \frac{1}{\delta}}{n}}.
\end{equation}

The first robust estimator to obtain \eqref{eq:multivariate_bound_optimal} was the median-of-means tournament \cite{lugosi2019vector}. This estimator requires minimizing a non-convex function over the set of directions $a \in \mathbb{S}^{d-1}$ and is not typically computationally feasible. Assuming no contamination and keeping the optimal bound \eqref{eq:multivariate_bound_optimal}, \cite{hopkins2020} first obtained a computational complexity of $O\left( nd + \left(d \ln\frac{1}{\delta}\right)^C \right)$ for some universal constant $C>0$, with \cite{cherapanamjeri2019fast} later obtaining a time complexity of $O(n^{3.5}+n^2d)$. Allowing for contamination and using SDP relaxation techniques from \cite{cheng2019high}, \cite{depersin2022robust} obtained \eqref{eq:multivariate_bound_optimal} with a time complexity of $O\left(nd + \left(\eps n \vee \ln\frac{1}{\delta}\right) d \ln\frac{1}{\delta} \right)$. These results highlight a tension between statistical optimality and computational feasibility, with \Cref{thm:multivariate_concentration_adv} attaining the best-known concentration guarantee over linear-time algorithms.

\section{Computational experiments}\label{sec:experiments}

We now empirically study the impact of shrinking towards different base estimators through a variety of choices of shrinkage functions \( w \). We show that improvements are not only in terms of order (see \Cref{thm:shrinkage_concentration_adv}), but also hold for small values of $n$. We also provide some guidance as to the choice of $w$. A GitHub repository with the necessary code to reproduce all figures and tables is made available at \url{https://github.com/lucasresenderc/shrink}.

\subsection{Improved concentration}\label{subsec:exp_improved} 
Our first experiment is related to \Cref{thm:main_informal} and aims to study whether the proposed shrinkage estimators \( \widehat{\mu }(\widehat{\kappa } )  \)  indeed improves upon the performance of their base estimators \( \widehat{\kappa }  \) in practice. To that end, we consider four base estimators: the empirical mean \( \overline{X} \), the sample median \( M \), the trimmed mean TM with trimming level $k = \lceil \ln \frac{1}{\delta} \rceil$ \citep{oliveira2025finite}, and the median-of-means MoM with $K = \lceil \ln \frac{1}{\delta} \rceil$ buckets \citep{lerasle2011robust}. For this experiment, we compute these base estimators on a sample of size \( m=25 \), whereas the shrinkage estimators \( \widehat{\mu} \)  are computed on an independent sample of size \( n=475 \), using the shrinkage functions given by \( w(t)=e^{-t^p},w(t)=(1+t^p)^{-1}\), \(w(t)=1\wedge t^{-1}\), \(w(t)=\mathbf{1}_{t<1}\), and \( w(t)=(1-t^p)_+ \), where \( p=2 \) since we only consider distributions with finite variance (a standard normal distribution N, a skewed normal SN, a Student's \( t \)-distribution with \( 2.01 \) degrees of freedom T, and a skewed \( t \)-distribution with \( 2.01 \) degrees of freedom ST). We fix \( \delta=0.05 \) and \( \eta = \ln \frac{1}{\delta } \).  

To evaluate the performance of each estimator, we compute the \( 1-\delta  \)-quantile of the errors computed over \( 10000 \) independent trials for the four different distributions. We then calculate the relative difference of this quantile with respect to that of the base estimate over the full sample of size \( N=500 \). The results are presented in \Cref{tab:is_shrinkage_good}.

We show that, in most cases, shrinking does not hurt the performance of the trimmed mean, and consistently improves the median-of-means. For both the empirical mean and the sample median, \( \widehat{\mu } (\widehat{\kappa } ) \) dramatically improves their performance except when the base estimator is already excellent (e.g., the empirical mean under light-tails and the sample median under symmetrical, heavy-tailed distributions, see \Cref{fig:is_shrinkage_good}). These results corroborate \Cref{thm:main_informal} and suggest that shrinking is a generally beneficial procedure. 
Finally, note that estimators with similar features (such as those with $w(t) = (1-t^p)_+$, $(1+t^p)^{-1}$, and \( e^{-t^p} \), which decay quickly for small $t$) also display similar performances. These shrinkage functions provide the best results under skewed distributions, while those with plateaus for small $t$ (such as $w(t) = 1 \wedge t^{-1}$ and $w(t) = \mathbf{1}_{t < 1}$) perform better under symmetric distributions. This suggests that, under prior evidence of skewness, it may be beneficial to choose a shrinkage function that decays more quickly for small $t$, favoring our novel estimators. Alternatively, under prior evidence of symmetry, it may be beneficial to choose a shrinkage function with a plateau for small $t$.

In order to compare the different base estimators, we also display in \Cref{fig:is_shrinkage_good} the error for each base estimator and, for the sake of a clearer graph, only the shrunk version of the sample median with \( w(t) = (1+t^p)^{-1} \), since the performance of the shrinkage estimators is similar, as \Cref{tab:is_shrinkage_good} suggests.

\begin{table}[t]
\tiny
\centering
\begin{tabular}{l@{\extracolsep{\fill}}|cccc|cccc|cccc|cccc}
& \multicolumn{4}{c|}{$\overline{X}$}& \multicolumn{4}{c|}{$M$}& \multicolumn{4}{c|}{TM}& \multicolumn{4}{c}{MoM}
\\
& N& SN& T& ST& N& SN& T& ST& N& SN& T& ST& N& SN& T& ST
\\
\hline
  $e^{-t^p}$ & \(1\) & \(\mathbf{2}\) & \(-33\) & \(\mathbf{-7}\) & \(-19\) & \(\mathbf{-69}\) & \(55\) & \(\mathbf{-65}\) & \(1\) & \(\mathbf{-1}\) & \(4\) & \(\mathbf{-7}\) & \(\mathbf{-12}\) & \(\mathbf{-11}\) & \(-31\) & \(\mathbf{-5}\) \\
$(1+t^p)^{-1}$ & \(1\) & \(\mathbf{2}\) & \(-32\) & \(\mathbf{-9}\) & \(-19\) & \(\mathbf{-69}\) & \(59\) & \(\mathbf{-66}\) & \(1\) & \(\mathbf{-1}\) & \(6\) & \(\mathbf{-9}\) & \(\mathbf{-12}\) & \(\mathbf{-11}\) & \(-30\) & \(\mathbf{-7}\) \\
$1\wedge t^{-1}$ & \(\mathbf{0}\) & \(9\) & \(\mathbf{-38}\) & \(2\) & \(\mathbf{-19}\) & \(-67\) & \(\mathbf{43}\) & \(-62\) & \(\mathbf{1}\) & \(5\) & \(\mathbf{-4}\) & \(-0\) & \(-12\) & \(-5\) & \(\mathbf{-37}\) & \(4\) \\
$\mathbf{1}_{t<1}$ & \(2\) & \(11\) & \(\mathbf{-37}\) & \(8\) & \(-17\) & \(-66\) & \(\mathbf{47}\) & \(-60\) & \(2\) & \(7\) & \(\mathbf{-2}\) & \(6\) & \(-10\) & \(-3\) & \(\mathbf{-35}\) & \(10\) \\
$(1-t^p)_+$ & \(\mathbf{1}\) & \(2\) & \(-34\) & \(-5\) & \(\mathbf{-19}\) & \(-69\) & \(53\) & \(-64\) & \(\mathbf{1}\) & \(-1\) & \(2\) & \(-5\) & \(-12\) & \(-11\) & \(-32\) & \(-3\) \\
\hline
\end{tabular}
\caption{Relative difference (in \%) of the empirical \( 1-\delta  \) confidence interval error of the shrunk estimators (for each $w$ (line) and $\widehat{\kappa}$ (column)) with respect to the base estimator $\widehat{\kappa}$ under different distributions. The entries corresponding to the two best performing shrinkage functions compared to the base estimator for each fixed base estimator and distribution are highlighted in bold.}\label{tab:is_shrinkage_good} 
\end{table}
\begin{figure}[t!]
	\centering
	\includegraphics[width=\linewidth]{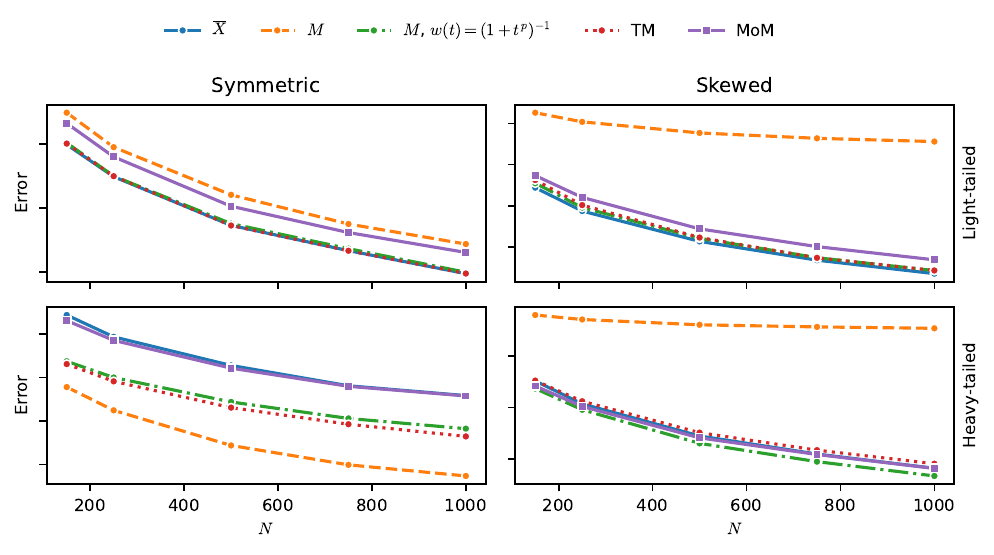}
	\caption{Empirical $1-\delta$ confidence interval errors plotted against sample sizes, for different distributions (skew/symmetry is represented by the columns, while tail weight is represented by the rows), and mean estimators. The shrinkage estimator based on the sample median (green) has excellent performance across all distributions and sample sizes considered: it always outperforms the median-of-means (purple) and is competitive against the trimmed mean (red) in terms of robustness.}
	\label{fig:is_shrinkage_good}
\end{figure}

\subsection{Splitting the sample}\label{subsec:exp_split}
We now study the effect of sample splitting, which is a natural way to fulfill \Cref{assum:kappa_independent}. Nevertheless, it is not clear how much of the available data should be dedicated to the base estimator, or even if computing the base estimator and the shrinkage estimator on the full sample is the best approach. To evaluate the performance of shrinkage estimators with varying sample splitting strategies, we compute the errors of these estimators as described in \Cref{subsec:exp_improved} for varying split ratios given by \( \left\lfloor m / N \right\rfloor \), along with fully utilizing the sample for both computations. The performances are displayed in \Cref{tab:split_eval}, for \( \delta =0.05 \), \( \eta = \ln \frac{1}{\delta } \), and \( N=500 \). 

These results suggest that, as long as we want to keep a splitting strategy, we should dedicate as few sample points as possible to the base estimator. A closer look at small values for \( m \) (and, thus, of split ratios) suggests that there often is an optimal value of \( m \) greater than \( 1 \). See \Cref{fig:split_eval,fig:best_split}. For \Cref{fig:split_eval}, the full sample size is kept at \( N=500 \), while \( N \) varies from \( 100 \) to \( 1000 \) in \Cref{fig:best_split}. 
To obtain \Cref{fig:best_split}, we computed the error and the value of \( m \) that minimizes the error (denoted by \( m^\star \)) for \( 50 \) independent trials and plotted its mean and error band determined by the \( 2.5\% \) and \( 97.5\% \) quantiles. The resulting graph shows that \( m^\star \) remains approximately constant even as \( N \) increases.
This corroborates \Cref{thm:adv_informal}, which states that \( \widehat{\mu }(\widehat{\kappa } )  \) attains optimal concentration rates as long as \(\widehat{\kappa }\) is not too far from the population mean. Hence, the practitioner should choose a relatively small value of \( m \) to maximize the impact of the shrinkage procedure.

\begin{table}[t]
\tiny
\centering
\begin{tabular}{l@{\extracolsep{\fill}}|cccc|cccc|cccc|cccc}
& \multicolumn{4}{c|}{N}& \multicolumn{4}{c|}{SN}& \multicolumn{4}{c|}{T}& \multicolumn{4}{c}{ST}
\\
& NA& $0.05$& $0.5$& $0.95$& NA& $0.05$& $0.5$& $0.95$& NA& $0.05$& $0.5$& $0.95$& NA& $0.05$& $0.5$& $0.95$
\\
\hline
  $\overline{X}$ & \multicolumn{4}{c|}{0} & \multicolumn{4}{c|}{0} & \multicolumn{4}{c|}{0} & \multicolumn{4}{c}{0}  \\
 $e^{-t^p}$ & \(-0\) & \(1\) & \(34\) & \(221\) & \(1\) & \(3\) & \(37\) & \(253\) & \(-34\) & \(-32\) & \(-17\) & \(49\) & \(-8\) & \(-8\) & \(19\) & \(124\) \\
 $(1+t^p)^{-1}$ & \(-0\) & \(1\) & \(34\) & \(226\) & \(1\) & \(3\) & \(37\) & \(247\) & \(-32\) & \(-31\) & \(-15\) & \(55\) & \(-11\) & \(-10\) & \(16\) & \(114\) \\
 $1\wedge t^{-1}$ & \(0\) & \(1\) & \(30\) & \(205\) & \(8\) & \(12\) & \(55\) & \(250\) & \(-39\) & \(-38\) & \(-24\) & \(42\) & \(-0\) & \(1\) & \(32\) & \(118\) \\
$\mathbf{1}_{t<1}$ & \(2\) & \(2\) & \(33\) & \(224\) & \(11\) & \(15\) & \(61\) & \(361\) & \(-38\) & \(-37\) & \(-22\) & \(52\) & \(5\) & \(7\) & \(39\) & \(158\) \\
$(1-t^p)_+$ & \(-0\) & \(1\) & \(34\) & \(217\) & \(1\) & \(3\) & \(38\) & \(266\) & \(-35\) & \(-33\) & \(-19\) & \(49\) & \(-6\) & \(-5\) & \(22\) & \(132\) \\
\hline
$M$ &  \multicolumn{4}{c|}{0}  &  \multicolumn{4}{c|}{0}  &  \multicolumn{4}{c|}{0}  &  \multicolumn{4}{c}{0} \\
$e^{-t^p}$ & \(-21\) & \(-20\) & \(6\) & \(155\) & \(-70\) & \(-69\) & \(-59\) & \(22\) & \(52\) & \(57\) & \(92\) & \(245\) & \(-65\) & \(-65\) & \(-53\) & \(6\) \\
$(1+t^p)^{-1}$ & \(-21\) & \(-20\) & \(6\) & \(159\) & \(-70\) & \(-69\) & \(-59\) & \(18\) & \(56\) & \(60\) & \(97\) & \(256\) & \(-66\) & \(-65\) & \(-54\) & \(2\) \\
$1\wedge t^{-1}$ & \(-21\) & \(-20\) & \(4\) & \(143\) & \(-68\) & \(-67\) & \(-54\) & \(27\) & \(41\) & \(45\) & \(78\) & \(228\) & \(-63\) & \(-62\) & \(-50\) & \(7\) \\
$\mathbf{1}_{t<1}$ & \(-19\) & \(-19\) & \(5\) & \(158\) & \(-67\) & \(-66\) & \(-52\) & \(59\) & \(43\) & \(46\) & \(83\) & \(244\) & \(-60\) & \(-60\) & \(-47\) & \(18\) \\
$(1-t^p)_+$ & \(-21\) & \(-20\) & \(6\) & \(152\) & \(-70\) & \(-69\) & \(-59\) & \(28\) & \(51\) & \(54\) & \(90\) & \(239\) & \(-64\) & \(-64\) & \(-52\) & \(9\) \\
\hline
TM & \multicolumn{4}{c|}{0}  &  \multicolumn{4}{c|}{0}  &  \multicolumn{4}{c|}{0}  &  \multicolumn{4}{c}{0} \\
$e^{-t^p}$ & \(-0\) & \(1\) & \(34\) & \(221\) & \(-4\) & \(-1\) & \(31\) & \(240\) & \(3\) & \(5\) & \(29\) & \(132\) & \(-10\) & \(-7\) & \(19\) & \(130\) \\
$(1+t^p)^{-1}$ & \(-0\) & \(2\) & \(34\) & \(226\) & \(-4\) & \(-1\) & \(31\) & \(233\) & \(6\) & \(8\) & \(33\) & \(140\) & \(-12\) & \(-9\) & \(16\) & \(119\) \\
$1\wedge t^{-1}$ & \(0\) & \(1\) & \(30\) & \(205\) & \(3\) & \(7\) & \(48\) & \(238\) & \(-4\) & \(-3\) & \(19\) & \(119\) & \(-3\) & \(0\) & \(30\) & \(125\) \\
$\mathbf{1}_{t<1}$ & \(2\) & \(2\) & \(33\) & \(224\) & \(6\) & \(10\) & \(54\) & \(344\) & \(-3\) & \(-2\) & \(23\) & \(132\) & \(3\) & \(5\) & \(36\) & \(165\) \\
$(1-t^p)_+$ & \(-0\) & \(1\) & \(34\) & \(217\) & \(-4\) & \(-1\) & \(32\) & \(253\) & \(2\) & \(4\) & \(28\) & \(129\) & \(-7\) & \(-5\) & \(22\) & \(137\) \\
\hline
MoM & \multicolumn{4}{c|}{0}  &  \multicolumn{4}{c|}{0}  &  \multicolumn{4}{c|}{0}  &  \multicolumn{4}{c}{0} \\
$e^{-t^p}$ & \(-14\) & \(-12\) & \(16\) & \(178\) & \(-14\) & \(-12\) & \(17\) & \(201\) & \(-31\) & \(-30\) & \(-14\) & \(56\) & \(-7\) & \(-6\) & \(21\) & \(128\) \\
$(1+t^p)^{-1}$ & \(-14\) & \(-12\) & \(16\) & \(181\) & \(-14\) & \(-12\) & \(17\) & \(196\) & \(-29\) & \(-29\) & \(-12\) & \(61\) & \(-10\) & \(-8\) & \(17\) & \(118\) \\
$1\wedge t^{-1}$ & \(-13\) & \(-13\) & \(13\) & \(164\) & \(-8\) & \(-5\) & \(33\) & \(199\) & \(-37\) & \(-36\) & \(-20\) & \(48\) & \(1\) & \(3\) & \(33\) & \(123\) \\
$\mathbf{1}_{t<1}$ & \(-11\) & \(-11\) & \(15\) & \(182\) & \(-5\) & \(-2\) & \(38\) & \(295\) & \(-36\) & \(-35\) & \(-18\) & \(57\) & \(6\) & \(8\) & \(40\) & \(163\) \\
$(1-t^p)_+$ & \(-14\) & \(-12\) & \(16\) & \(175\) & \(-13\) & \(-12\) & \(17\) & \(215\) & \(-32\) & \(-31\) & \(-15\) & \(55\) & \(-5\) & \(-4\) & \(24\) & \(136\) \\
\hline
\end{tabular}
\caption{Relative error of \( \widehat{\mu }(\widehat{\kappa } )  \) with respect to \( \widehat{\kappa }  \) (found in the first row of each block) for different distributions and split ratios (columns). NA means that both \( \widehat{\kappa }  \) and \( \widehat{\mu  }(\widehat{\kappa } )  \) were computed on the same full sample. The performance seems to decrease as the split ratio increases. Fully utilizing the sample for both computations yields the best results, although splitting the data with \( 5\% \) of the data dedicated to the base estimator leads to similar performances.}
\label{tab:split_eval}
\end{table}
\begin{figure}[t!]
	\centering\includegraphics[width=\linewidth]{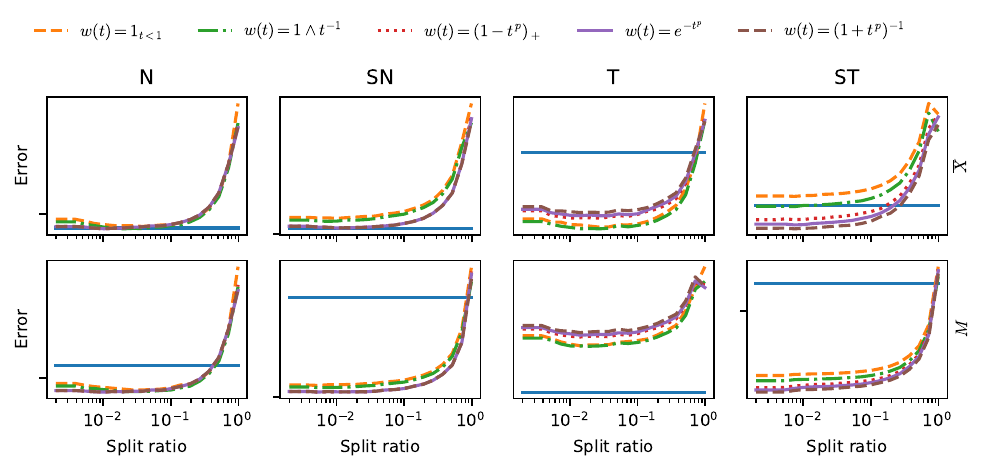}
	\caption{Error of shrinkage estimators plotted against split ratio for different base estimators (rows), distributions (columns) and shrinkage functions (curves). The solid blue line represents the performance of the base estimator computed on the full sample. The best performing split ratios are consistently below \( 10\% \) (\( m =50 \)).}
	\label{fig:split_eval}
\end{figure}
\begin{figure}[t!]
	\centering\includegraphics[width=\linewidth]{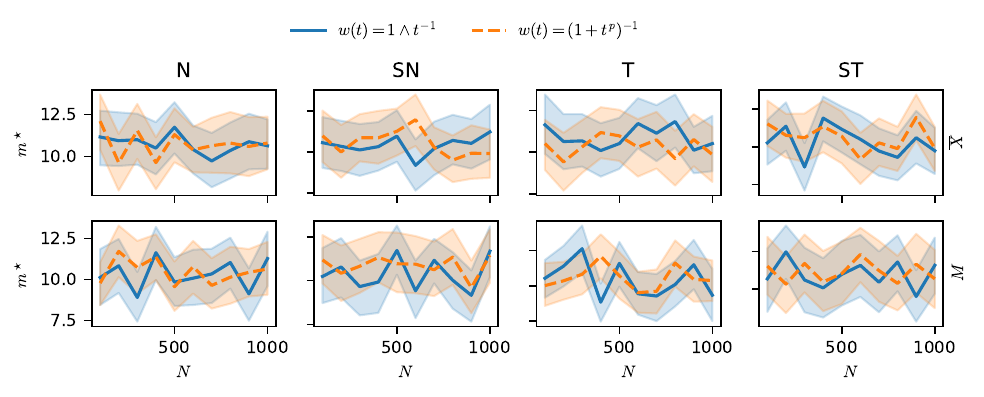}
  \caption{Best performing sample size \( m^\star \) for base estimators \( \overline{X} \) and \( M \) (rows), plotted against \( N \), for different distributions (columns) and shrinkage functions (curves), for fixed \( \delta=0.05  \). The value \( m^\star \) does not increase as \( N \) increases, in accordance with \Cref{thm:adv_informal}.}
	\label{fig:best_split}
\end{figure}
\subsection{Robustness under contamination}
To evaluate robustness under contamination, we sample our data as in \Cref{subsec:exp_improved,subsec:exp_split} and randomly substitute \( \left\lfloor \varepsilon m \right\rfloor \) of the sample points destined for the computation of \( \widehat{\kappa }  \) and \( \left\lfloor \varepsilon n  \right\rfloor \) of the sample points destined for the computation of \( \widehat{\mu }(\kappa )  \) by \( 10^6 \). We display the performance of the base estimators under this contamination as well as that of \( \widehat{\mu }  \) in \Cref{tab:contamination}. The results confirm in practice that our shrinkage estimators \( \widehat{\mu }  \) indeed retains robustness under contamination as long as the base estimator is robust. Is is noteworthy that the worst-performing shrinkage estimators under heavy contamination are given by those based on the shrinkage function \( w(t)=1 \wedge t^{-1} \), which has the slowest asymptotic decay contemplated by our theory. Therefore, in alignment with our intuition, the user should employ aggressive shrinkage functions --- that is, ones with fast decay --- under situations where extreme outliers are expected.
\begin{table}[t]
\tiny
\centering
\begin{tabular}{l@{\extracolsep{\fill}}|ccc|ccc|ccc|ccc}
& \multicolumn{3}{c|}{N}& \multicolumn{3}{c|}{SN}& \multicolumn{3}{c|}{T}& \multicolumn{3}{c}{ST}
\\
& $0.0$& $0.05$& $0.2$& $0.0$& $0.05$& $0.2$& $0.0$& $0.05$& $0.2$& $0.0$& $0.05$& $0.2$
\\
\hline
$\overline{X}$ & \(\mathbf{-1.1}\) & \(4.7\) & \(5.3\) & \(\mathbf{-1.0}\) & \(4.7\) & \(5.3\) & \(-0.5\) & \(4.7\) & \(5.3\) & \(-0.3\) & \(4.7\) & \(5.3\) \\
$\mathbf{1}_{t<1}$ & \(-1.0\) & \(3.7\) & \(4.6\) & \(-0.9\) & \(3.7\) & \(4.6\) & \(-0.7\) & \(3.7\) & \(4.6\) & \(-0.3\) & \(3.7\) & \(4.6\) \\
$1\wedge t^{-1}$ & \(-1.1\) & \(3.9\) & \(4.9\) & \(-0.9\) & \(3.9\) & \(4.9\) & \(-0.7\) & \(3.9\) & \(4.9\) & \(-0.3\) & \(3.9\) & \(4.9\) \\
$(1-t^p)_+$ & \(-1.1\) & \(3.8\) & \(4.9\) & \(\mathbf{-1.0}\) & \(3.8\) & \(4.9\) & \(-0.7\) & \(3.8\) & \(4.9\) & \(-0.3\) & \(3.8\) & \(4.9\) \\
$e^{-t^p}$ & \(-1.1\) & \(4.1\) & \(5.0\) & \(-1.0\) & \(4.1\) & \(5.0\) & \(-0.7\) & \(4.1\) & \(5.0\) & \(-0.3\) & \(4.1\) & \(5.0\) \\
$(1+t^p)^{-1}$ & \(-1.1\) & \(4.3\) & \(5.1\) & \(-1.0\) & \(4.3\) & \(5.1\) & \(-0.7\) & \(4.3\) & \(5.1\) & \(\mathbf{-0.3}\) & \(4.3\) & \(5.1\) \\
\hline
$M$ & \(-1.0\) & \(-0.8\) & \(-0.4\) & \(-0.4\) & \(-0.5\) & \(-0.5\) & \(\mathbf{-0.9}\) & \(\mathbf{-0.7}\) & \(-0.3\) & \(0.1\) & \(0.1\) & \(-0.1\) \\
$\mathbf{1}_{t<1}$ & \(-1.0\) & \(-1.0\) & \(-0.7\) & \(-0.9\) & \(-0.9\) & \(\mathbf{-0.7}\) & \(-0.7\) & \(\mathbf{-0.8}\) & \(-0.6\) & \(-0.3\) & \(-0.2\) & \(-0.2\) \\
$1\wedge t^{-1}$ & \(-1.1\) & \(-0.7\) & \(-0.2\) & \(-0.9\) & \(-0.7\) & \(-0.2\) & \(-0.7\) & \(-0.2\) & \(0.3\) & \(-0.3\) & \(-0.0\) & \(0.5\) \\
$(1-t^p)_+$ & \(-1.1\) & \(-1.0\) & \(-0.7\) & \(-1.0\) & \(-0.9\) & \(-0.7\) & \(-0.7\) & \(-0.7\) & \(-0.6\) & \(-0.3\) & \(-0.3\) & \(-0.2\) \\
$e^{-t^p}$ & \(-1.1\) & \(-1.0\) & \(-0.7\) & \(-1.0\) & \(-0.9\) & \(-0.7\) & \(-0.7\) & \(-0.7\) & \(-0.6\) & \(-0.3\) & \(-0.3\) & \(-0.2\) \\
$(1+t^p)^{-1}$ & \(-1.1\) & \(-1.0\) & \(-0.7\) & \(-1.0\) & \(-0.9\) & \(-0.7\) & \(-0.7\) & \(-0.7\) & \(-0.6\) & \(-0.3\) & \(-0.3\) & \(-0.2\) \\
\hline
TM & \(\mathbf{-1.1}\) & \(-0.7\) & \(-0.3\) & \(-1.0\) & \(-0.8\) & \(-0.3\) & \(-0.7\) & \(-0.4\) & \(-0.1\) & \(-0.3\) & \(\mathbf{-0.4}\) & \(-0.1\) \\
$\mathbf{1}_{t<1}$ & \(-1.0\) & \(-1.0\) & \(-0.7\) & \(-0.9\) & \(-0.9\) & \(\mathbf{-0.8}\) & \(-0.7\) & \(-0.7\) & \(-0.6\) & \(-0.3\) & \(-0.3\) & \(-0.2\) \\
$1\wedge t^{-1}$ & \(-1.1\) & \(-0.7\) & \(-0.2\) & \(-0.9\) & \(-0.7\) & \(-0.2\) & \(-0.7\) & \(-0.2\) & \(0.3\) & \(-0.3\) & \(-0.0\) & \(0.5\) \\
$(1-t^p)_+$ & \(-1.1\) & \(\mathbf{-1.0}\) & \(-0.7\) & \(-1.0\) & \(-1.0\) & \(-0.7\) & \(-0.7\) & \(-0.7\) & \(-0.6\) & \(-0.3\) & \(-0.3\) & \(-0.3\) \\
$e^{-t^p}$ & \(-1.1\) & \(\mathbf{-1.0}\) & \(\mathbf{-0.7}\) & \(-1.0\) & \(\mathbf{-1.0}\) & \(-0.7\) & \(-0.7\) & \(-0.7\) & \(-0.6\) & \(-0.3\) & \(-0.3\) & \(\mathbf{-0.3}\) \\
$(1+t^p)^{-1}$ & \(-1.1\) & \(-1.0\) & \(\mathbf{-0.7}\) & \(-1.0\) & \(\mathbf{-1.0}\) & \(-0.7\) & \(-0.7\) & \(-0.7\) & \(-0.6\) & \(-0.3\) & \(\mathbf{-0.3}\) & \(\mathbf{-0.3}\) \\
\hline
MoM & \(-1.0\) & \(-0.3\) & \(-0.2\) & \(-0.9\) & \(-0.2\) & \(-0.1\) & \(-0.6\) & \(-0.1\) & \(0.0\) & \(-0.3\) & \(0.1\) & \(0.2\) \\
$\mathbf{1}_{t<1}$ & \(-1.0\) & \(-0.9\) & \(-0.7\) & \(-0.9\) & \(-0.9\) & \(-0.4\) & \(-0.7\) & \(-0.6\) & \(-0.5\) & \(-0.3\) & \(-0.3\) & \(-0.2\) \\
$1\wedge t^{-1}$ & \(-1.1\) & \(-0.7\) & \(-0.1\) & \(-0.9\) & \(-0.6\) & \(-0.2\) & \(\mathbf{-0.7}\) & \(-0.2\) & \(0.3\) & \(-0.3\) & \(0.0\) & \(0.5\) \\
$(1-t^p)_+$ & \(-1.1\) & \(-1.0\) & \(-0.6\) & \(-1.0\) & \(-0.9\) & \(-0.6\) & \(-0.7\) & \(-0.6\) & \(\mathbf{-0.6}\) & \(-0.3\) & \(-0.3\) & \(-0.2\) \\
$e^{-t^p}$ & \(-1.1\) & \(-1.0\) & \(-0.5\) & \(-1.0\) & \(-0.9\) & \(-0.6\) & \(-0.7\) & \(-0.6\) & \(-0.5\) & \(-0.3\) & \(-0.3\) & \(-0.2\) \\
$(1+t^p)^{-1}$ & \(-1.1\) & \(-1.0\) & \(-0.6\) & \(-1.0\) & \(-0.9\) & \(-0.5\) & \(-0.7\) & \(-0.6\) & \(\mathbf{-0.6}\) & \(\mathbf{-0.3}\) & \(-0.3\) & \(-0.2\) \\
\hline
\end{tabular}
\caption{Base \( 10 \) logarithm of the errors of the mean estimators (rows) for different distributions and contamination levels (columns). The two best performing estimator for each distribution and contamination level are highlighted in bold. All estimators except for those based on the sample mean are reasonably robust to contamination, with the shrinkage estimators \( \widehat{\mu }  \) often attaining the best performance under heavy contamination.}
\label{tab:contamination}
\end{table}
\section{Discussion}
This paper shows that downweighting sample points far from a location estimate is a fundamental source of robustness in terms of concentration rates, in particular achieving sub-Gaussian behavior for many $(n,\delta)$ regimes as well as asymptotic efficiency, breakdown point, and affine equivariance. While most works in robust statistics focus on a single estimator, our results apply to a broad class of mean estimators, providing unified proof techniques and a wide-ranging set of tools that users can apply to varying problems. Indeed, our experiments reveal that different choices of $w$ excel in distinct regimes of sample size and confidence level $1-\delta$, exhibiting varying sensitivities to the base estimator. Thus, a refinement of our results beyond order bounds could guide the selection of $w$ for specific applications.

\begin{appendix}
	\section{Useful Lemmata and Definitions}\label[appendix]{appendix:usefullemmata}
\begin{lemma}[Bernstein's inequality] \label{lemma:bernstein} Let $Y_1, \dots, Y_n$ be i.i.d. random variables with mean zero and bounded by $1$. Then,
\[ \Pr\left( \frac{1}{n} \sum_{i=1}^n Y_i \geq \sigma \sqrt{\frac{2}{n} \ln\frac{1}{\delta}} + \frac{1}{3n}\ln\frac{1}{\delta} \right) \geq 1 - \delta . \]
\end{lemma}

	\begin{definition}[VC-subgraph dimension]\label{def:vc_sub}
		Given a class of functions \(\mathcal{F}\) mapping from \(\mathcal{X}\) to \(\R\), the VC-subgraph dimension of \(\mathcal{F}\) is the largest integer \(d\) such that there exist \(d\) pairs in \(\sX\times\R\) \((x_1, y_1), \dots, (x_d, y_d)\) for which the set of subgraphs $\left\{\left\{ (x, y)\in\sX\times\R : f(x) > y\right\}: f \in \mathcal{F} \right\}$ can realize all \(2^d\) possible labelings. That is, for every subset \(S \subseteq \{1, \dots, d\}\), there exists a function \(f \in \mathcal{F}\) such that $f(x_i) > y_i \text{ if and only if } i \in S.$
	\end{definition}

	\begin{lemma}[Concentration of suprema of empirical processes under VC assumptions]\label[lemma]{lemma:bousquet} Let $P$ be a distribution over a set $\bX$ and $\sG$ be a class of integrable functions $g : \bX \to \R$. Assume that $\sup_{g \in \sG}\| g - Pg \|_\infty \leq L$ for some $L>0$. Let $G(x) = \sup_{g \in \sG} | g(x) |$ be the envelope of $\sG$ and assume $\sG$ is a VC-subgraph class with dimension $d$. Assume also that exists a countable $\sH \subset \sG$ such that every $g \in \sG$ is the $P$-pointwise a.s. and $L_1(P)$ limits of functions $(h_k)_{k=1}^\infty \subset \sH$. For every $\delta \in (0,1)$, it holds with probability at least $1-\delta$, that
		\begin{align*}
			\sup_{g \in \sG} \left| \left(\Pmhat_n - P\right)g \right| & \leq K \| G \|_{L_2(P)} \sqrt{\frac{d}{n}} + \sqrt{\frac{2}{n}\ln\frac{1}{\delta} \left(\sigma^2 + 2 L K \| G \|_{L_2(P)} \sqrt{\frac{d}{n}} \right) } + \frac{L}{3n}\ln\frac{1}{\delta} \\
			                                                           & \leq 2 K \| G \|_{L_2(P)} \sqrt{\frac{d}{n}} + \sigma \sqrt{\frac{2}{n}\ln\frac{1}{\delta} } + \frac{4L}{3n}\ln\frac{1}{\delta}
		\end{align*}
		where $K>0$ is an absolute constant and $\sigma^2 = \sup_{g \in \sG} P(g - Pg)^2$.
	\end{lemma}
	\begin{proof}
		Since $\sG$ is pointwise $P$-a.s. and $L_1(P)$ limit of functions in $\sH$ we have
		\[ \sup_{g \in \sG} \left| \left(\Pmhat_n - P\right)g \right| = \sup_{h \in \sH} \left| \left(\Pmhat_n - P\right)h \right| \text{ a.s. }\]
		On the countable family $\sH$, apply Bousquet's version of Talagrand's inequality \cite[Theorem 2.3]{bousquet2002bennett} and use the entropy bound on the expectation of the empirical process (\Cref{lemma:integral_entropy}).
	\end{proof}

	\Cref{def:vc_sub} is closely related to the following definition of covering number. This relation yields the proof of \Cref{lemma:integral_entropy}.

	\begin{definition}[Covering number]\label{def:covering_number}
		Let \((\mathcal{X}, d)\) be a metric space. Given \(\eps > 0\), the \(\eps\)-covering number of \(\sX\) with respect to the metric \(d\), denoted by \(\sN(\sX, d, \eps)\), is the smallest number of elements \(x_1, \dots, x_n \in \sX\) such that for every \(x \in \sX\), there exists some \(x_i\) satisfying \( d(x, x_i) \leq \eps
		\). In other words, \(\sN(\sF, d, \eps)\) is the minimal number of points required to form an \(\eps\)-net that covers \(\sX\) with respect to the given metric.
	\end{definition}

	\begin{lemma}[Integral entropy bound]\label[lemma]{lemma:integral_entropy}
		Let $X,X_1,...,X_n$ be i.i.d., $\sX$-valued random variables with distribution $P$ and $\mathcal{F}$ a family of real functions define in $\sX$. Assume $\sF$ is VC-subgraph with dimension $d$ and an envelope function $F$ (that is, $|f(x)|\leq F(x)$ for all $f\in\mathcal{F}$ and $x\in\sX$).

		$$\E\sup_{f\in\sF}\left\lvert \frac{1}{n}\sum_{i=1}^nf(X_i)-\E f(X) \right\rvert \leq K\lVert F\rVert_{L_2(P)}\sqrt{\frac{d}{n}}$$
		where $K>0$ is an absolute constant upper-bounded by \( 384 \).
	\end{lemma}
	\begin{proof}
    A symmetrization argument followed by an application of Dudley's theorem (see Example 5.24 of \cite{wainwrighthigh} or, alternatively, the proof of Theorem 8.3.23 in \cite{vershynin2018high}) implies
		\begin{align}
			\E\sup_{f\in\sF}\left\lvert \frac{1}{n}\sum_{i=1}^nf(X_i)-\E f(X) \right\rvert & \leq \frac{48}{\sqrt{n}}\E\int_0^{2\lVert F\rVert _{L_2(\Pr_n)}}\sqrt{\ln\mathcal{N}(\sF,\lVert \cdot\rVert_{L_2(\Pr_n)},\eps)}d\eps\nonumber \\
			                                                                               & \leq \frac{48}{\sqrt{n}}\lVert F\rVert_{L_2(P)}\int_0^2\sup_Q\sqrt{\ln\mathcal{N}(\sF,\lVert  \cdot\rVert_{L_2(Q)},\eps\lVert  F\rVert_{L_2(Q)})}d\eps,\label{eq:entropy_bound}
		\end{align}
		where the supremum is taken over distributions $Q$. We proceed by using that $F$ is an envelope of $\sF$ to bound (proof in \Cref{lemma:envelope_bound})
		\[
			\lVert  f-g\rVert_{L_2(Q)}\leq2\lVert  F\rVert_{L_2(Q)}\lVert  \mathbf{1}_{\textup{sg}(f)}-\mathbf{1}_{\textup{sg}(g)}\rVert_{L_2(Q')}.
		\]
		It then follows that
		\[
			\mathcal{N}(\sF,\lVert  \cdot\rVert_{L_2(Q)},\eps\lVert  F\rVert_{L_2(Q)})\leq\mathcal{N}(\{\mathbf{1}_{\textup{sg}(f)}:f\in\sF\},\lVert  \cdot\rVert_{L_2(Q')},\eps/2)\leq\left(\frac{512}{\eps^4}\right)^{2d},
		\]
    where the last inequality follows by \cite[Theorem 8.3.18]{vershynin2018high} and $\textup{sg}(f):=\{(x,t)\in\sX\times\R:f(x)>t\}$. Plugging this inequality into \eqref{eq:entropy_bound} concludes the proof. Computing the integral \( \int_0^2 \sqrt{ \ln (512\varepsilon^{-4})} d \varepsilon \approx 5.30875 \) yields the numerical value for the constant \( K \). 
	\end{proof}

	\begin{lemma}\label[lemma]{lemma:envelope_bound}
		Let $\sF$ be a family of functions from $\sX$ to $\R$ and $F$ be its envelope. Let $Q$ be any distribution over $\sX$. Then, for any $f,g\in\sF$,
		\[
			\lVert  f-g\rVert_{L_2(Q)}\leq2\lVert  F\rVert_{L_2(Q)}\lVert  \mathbf{1}_{\textup{sg}(f)}-\mathbf{1}_{\textup{sg}(g)}\rVert_{L_2(Q')},
		\]
		where $Q'$ is a distribution over $\sX\times\R$.
	\end{lemma}
	\begin{proof}
		We follow the proof presented in \cite{zhou2020}.
		\begin{align*}
      \int\lvert  f-g\rvert ^2dQ & \leq\int2F(x)\lvert  f(x)-g(x)\rvert dQ(x)\\
                                 &=\int 2F(x)\left(\int\lvert  \mathbf{1}_{\textup{sg}(f)}(x,t)-\mathbf{1}_{\textup{sg}(g)}(x,t)\rvert dt\right)dQ(x) \\
                                   & =2\iint_{|t|\leq F(x)}F(x)\lvert  \mathbf{1}_{\textup{sg}(f)}(x,t)-\mathbf{1}_{\textup{sg}(g)}(x,t)\rvert dt\,dQ(x) \\
			                  & =2\iint_{|t|\leq F(x)}F(x)dQ(x)\,dt \\
                        & \hspace{1cm}\times\iint\lvert  \mathbf{1}_{\textup{sg}(f)}(x,t)-\mathbf{1}_{\textup{sg}(g)}(x,t)\rvert ^2\frac{F(x)\mathbf{1}_{\lvert  t\rvert \leq F(x)}}{\iint_{|t|\leq F(x)}F(x)dQ(x)\,dt}dt\,dQ(x) \\
			                  & =4\lVert  F\rVert^2_{L_2(Q)}\lVert  \mathbf{1}_{\textup{sg}(f)}-\mathbf{1}_{\textup{sg}(g)}\rVert_{L_2(Q')}^2,
		\end{align*}
		where $Q'$ is absolutely continuous with respect to $Q\times\text{Lebesgue}$ and its density is given by
    \[\frac{F(x)\mathbf{1}_{\lvert  t\rvert \leq F(x)}}{\iint_{|t|\leq F(x)}F(x)dQ(x)\,dt}. \qedhere\]
	\end{proof}
	\section{Technical proofs}\label[appendix]{appendix:proofs}

	\begin{lemma}\label[lemma]{lemma:rho_bounds}
    Assume \( X_{1:n} \sim P^{\otimes n} \) with \( P \) atom-free. Fix \( \kappa \in \R \). Then the following inequality holds almost surely for all \( X_{1:n}^\varepsilon  \in \sA(X_{1:n},\varepsilon ) \):
    \[ n-\eta-1-\left\lfloor n \varepsilon  \right\rfloor \leq \sum_{i=1}^n w\left(\widehat{\alpha}\left(\kappa; X_{1:n}^\varepsilon \right) |X_i^\varepsilon  - \kappa|\right)\leq n - \eta.  \]
	\end{lemma}
  \begin{proof}
    For the purpose of this proof, denote \( \widehat{\alpha }= \widehat{\alpha }(\kappa;X_{1:n}^\varepsilon )   \) Define
		\[
      S_{X^\varepsilon _{1:n}}(\alpha):=\sum_{i=1}^n w(\alpha |X_i^{\varepsilon} - \kappa|).
		\] The upper bound follows directly by the definition of $\widehat{\alpha}$ and right-continuity of $w$:
		\[
      S_{X_{1:n}^{\varepsilon}}\left(\widehat{\alpha}\right)=\lim_{\alpha\downarrow\widehat{\alpha}}S_{X_{1:n}^{\varepsilon}}(\alpha)\leq n-\eta
		\]We proceed by proving the lower bound.
    Assume that $S_{X_{1:n}^{\varepsilon}}\left(\widehat{\alpha}\right) < n-\eta-1-\left\lfloor n \varepsilon  \right\rfloor$, with positive probability. By definition of $\widehat{\alpha}$, $\lim_{\alpha\uparrow\widehat{\alpha}}S_{X_{1:n}^{\varepsilon}}(\alpha)\geq n-\eta.$ Thus,
		\[
      \lim_{\alpha\uparrow\widehat{\alpha}(\kappa)}S_{X_{1:n}^{\varepsilon}}(\alpha)-S_{X_{1:n}^{\varepsilon}}(\widehat{\alpha})>1+ \left\lfloor n \varepsilon  \right\rfloor
		\]
    and $\widehat{\alpha}$ is a discontinuity of $\alpha\mapsto w(\alpha \lvert  X_i^{\varepsilon}-\kappa\rvert )$ for at least \( 2 + \left\lfloor n \varepsilon  \right\rfloor \)  different indices. Since \( X_{1:n}^\varepsilon \neq X_{1:n} \) holds for at most \( \left\lfloor n \varepsilon  \right\rfloor \) indices, then there are at least two different indices \( i,j \) such that \( \widehat{\alpha }  \) is a discontinuity of \( \alpha  \mapsto w(\alpha \lvert X_l - \kappa  \rvert )  \) for \( l \in  \{i,j\} \).
Let $\sD$ be the set of discontinuities of $w$, which we know to be countable since $w$ is non-increasing. We will now show that, because $P$ has no point mass, the event
\[
  D_{k,l}:=\left\{(\widehat{\alpha} \lvert  X_k-\kappa\rvert ,\widehat{\alpha} \lvert  X_l-\kappa\rvert )\in\sD\times\sD\right\}
\]
has probability zero, where \( k\neq l \). Denote $\sD/\sD:=\left\{\frac{t_1}{t_2}:(t_1,t_2)\in\sD\times\sD\right\}$, which is well defined since $0\notin\sD$ ($w$ is defined on $[0,\infty)$ and right-continuous) and is countable. Notice that 
\[D_{k,l}\subseteq\left\{\frac{\lvert  X_k-\kappa\rvert }{\lvert  X_l-\kappa\rvert }\in\sD/\sD\text{ and } \lvert  X_l-\kappa\rvert >0\right\}.
\]
The proof is complete once we observe that $\frac{\lvert  X_k-\kappa\rvert }{\lvert  X_l-\kappa\rvert }$ has no point mass because \( X_k,X_l \) are independent, $\lvert  X_s-\kappa\rvert $ has no point mass for \( s\in \{k,l\} \), and is almost surely positive. We finish the proof with an union bound argument on the pairs \( (k,l) \), showing that $\Pr(\bigcup_{k\neq l} D_{k,l})=0$, which implies that $\widehat{\alpha}$ is a discontinuity of $\alpha\mapsto w(\alpha \lvert  X_i-\kappa\rvert )$ for at most one index almost surely.
	\end{proof}
\begin{lemma}
    \label{lemma:bound_vc_dim} For every $\kappa \in I_{\widehat{\kappa}}$, the family
    \[ \sF_\kappa = \{ x \mapsto \kappa + (x-\kappa)w(\alpha \lvert  x-\kappa\rvert ) - \mu : \alpha>0   \} \]
    has VC-subgraph dimension $1$.
\end{lemma}

\begin{proof}
  Without loss of generality, we may show that the family \[\sF'= \{ x \mapsto  xw(\alpha \lvert  x\rvert ) : \alpha>0  \}\] has VC-subgraph dimension $1$. Assume by contradiction that $((x_1,t_1),(x_2,t_2))$ is shattered by
\[
\left\{\{(x,t)\in \R^2: g(x)>t\}:g\in\sF'\right\}.
\]
Notice that $x\in\{x_1,x_2\}$ is not equal to zero, since that would imply that, for every $g\in\sF'$, $g(x)=0$. This leads to two cases 
\begin{itemize}
    \item $x_1x_2<0$:
    In this case, assume that $x_1<0$ and $x_2>0$. Since $((x_1,t_1),(x_2,t_2))$ is shattered, there exist $\alpha_1,\alpha_2>0$ such that
    \[
    \begin{cases}
        x_1w(\alpha_1 \lvert x_1 \rvert )>t_1\text{ and } x_1w(\alpha_2\lvert x_1 \rvert  )\leq t_1 \implies \alpha_1>\alpha_2\\
        x_2w(\alpha_1 \lvert x_2 \rvert )>t_2\text{ and } x_2w(\alpha_2 \lvert x_2 \rvert )\leq t_2 \implies \alpha_1<\alpha_2
    \end{cases},
    \]
    which is a contradiction. The implications above are a result of $w$ being non-increasing and $\lvert  x_1\rvert ,\lvert  x_2\rvert >0$.
    \item $x_1x_2>0$: Analogously to the first case, there exist $\alpha_1,\alpha_2>0$ such that
    \[
        \begin{cases}
        x_1w(\alpha_1 \lvert x_1 \rvert )>t_1\text{ and } x_1w(\alpha_2 \lvert x_1 \rvert )\leq t_1\\
        x_2w(\alpha_1 \lvert x_2 \rvert )\leq t_2\text{ and } x_2w(\alpha_2 \lvert x_2 \rvert )> t_2 
    \end{cases}.
    \]
    This implies a contradiction both in the case $x_1,x_2>0$ and $x_1,x_2<0$.\qedhere
\end{itemize}
\end{proof}

{\renewcommand{\proofname}{Proof of \Cref{theorem:bias_variance}}
	\begin{proof}
		By hypothesis \eqref{eq:alpha_in_alphak_whp}, we have that with probability at least \( 1-\frac{3\delta}{4} \), \begin{align}
      \lvert \widehat{\mu}(\kappa ) - \mu \rvert & \le \sup_{\alpha \in I_\alpha (\kappa )} \lvert  b(\alpha,\kappa )\rvert  + \sup_{\alpha \in I_\alpha (\kappa )}\lvert s(\alpha,\kappa ) - \mu - b(\alpha,\kappa ) \rvert \nonumber \\
                                                                   & =\sup_{\alpha \in I_\alpha (\kappa )} \lvert  b(\alpha,\kappa )\rvert  + \sup_{f \in \sF_\kappa } \left|\left(\Pmhat_n - P\right) f\right|.\label{eq:bias_var_bound}
		\end{align}
		We can use \Cref{lemma:bousquet} with $\sG=\sF_\kappa $ to bound the second term of \eqref{eq:bias_var_bound}. We proceed by checking the assumptions of \Cref{lemma:bousquet}. First, since $\alpha \mapsto m(\alpha)$ is non-increasing,
		\begin{align*}
			\sup_{f \in \sF_\kappa } \| f - Pf \|_\infty & = \sup_{\alpha \in I_\alpha (\kappa )} \sup_{x \in \R} | (x-\kappa) w(\alpha |x-\kappa|) - \E (X-\kappa) w(\alpha |X-\kappa|) | \\
			                                     & \leq \sup_{\alpha \in I_\alpha (\kappa )} 2\, m(\alpha) =  2\,m(\underline{\alpha}(\kappa)).
		\end{align*}

		Define
		$\sH := \{ x \mapsto \kappa + (x-\kappa) w(\alpha |x-\kappa|) - \mu : \alpha \in I_\alpha(\kappa) \cap \Q \}\subseteq\sF_\kappa .$
		$\sH$ is a countable subset of $\sF_\kappa $ and, since $w$ is right-continuous and non-increasing, every \( h\in \sH \) has a countable set of discontinuities. Combining this with \Cref{assum:no_point_mass} implies that each function $f \in \sF_\kappa $ is the $P$-a.s. pointwise limit of functions $(h_i)_{i=1}^\infty \in \sH$. Since each $f \in \sF_\kappa $ is bounded by $ m(\underline{\alpha}(\kappa)) + \lvert \kappa -\mu  \rvert  $, the dominated convergence theorem ensures that $f$ is also the $L_1(P)$ limit of functions $(h_i)_{i=1}^\infty \in \sH$. Thus, a straightforward application of \Cref{lemma:bousquet} gives
		\[
			\Pr\left[ \sup_{\alpha \in I_\alpha (\kappa )}\lvert s(\alpha,\kappa ) - \mu - b(\alpha,\kappa ) \rvert > \left(2K+\sqrt{2}\right) \| F_\kappa  \|_{L_2(P)}\sqrt{\frac{1}{n} \ln \frac{4}{\delta}}+\frac{8 m(\underline{\alpha}(\kappa))}{3n}\ln \frac{4}{\delta }  \right] \leq \frac{\delta}{4},
		\]
    where the VC-subgraph dimension of \( \sF_\kappa  \) is bounded by \Cref{lemma:bound_vc_dim}. The result then follows by integrating over \( \widehat{\kappa }  \).
	\end{proof}}

	{\renewcommand{\proofname}{Proof of \Cref{lemma:alphaorder}}
	\begin{proof}
		Let \( \underline{c}, \overline{c}>0 \) such that
		\begin{align*}
      1-\underline{c}&\ge  \sqrt{2\underline{c}} + \frac{2 }{3},\\
      \overline{c}-2&\ge \sqrt{2\overline{c}}+2\xi^{-1}+\frac{13}{3} , 
		\end{align*}

    To see that there is such a solution, it is enough to solve for \( \underline{c},\overline{c} \). Since \( \frac{\overline{c} \ln \frac{4}{\delta }}{n}+(1+2\xi )\eps < 1 \), it holds that \( 1-\frac{\underline{c}\ln \frac{4}{\delta } }{n}-\frac{\xi}{7}\varepsilon ,1-\frac{\overline{c} \ln \frac{4}{\delta }}{n}-(1+2\xi )\eps \in (0,1) \). Then, by \Cref{assum:no_point_mass}, \( \underline{\alpha}(\kappa),\overline{\alpha}(\kappa) \) are well defined, because $\alpha\mapsto\E[w(\alpha \lvert  X-\kappa\rvert )]$ is continuous (as $P$ has no point mass) and, by \Cref{assum:malpha_finite}, $\lim_{\alpha\to+\infty}\E[w(\alpha \lvert  X_i-\kappa\rvert )]=0$.

		Recall the definition given by \eqref{eq:def_alpha}. Using that $w$ is right-continuous, takes values on $[0,1]$ and that $P$ is atom-free, we can bound (\Cref{lemma:rho_bounds})
    \[ n - \eta - 1 \leq \sum_{i=1}^n w(\widehat{\alpha}(\kappa) |X_i  - \kappa|) \le \sup_{X_{1:n}^\varepsilon \in \sA(X_{1:n},\varepsilon ) }\sum_{i=1}^n w(\widehat{\alpha}^\eps (\kappa) |X_i^\eps   - \kappa|) \leq n - \eta, \text{ a.s.} \]

    Since $\alpha \mapsto \frac{1}{n} \sum_{i-1}^n w(\alpha |X_i - \kappa|)$ is a non-increasing function on $\alpha$, in order to show the inequalities $\underline{\alpha }(\kappa )<\inf_{X_{1:n}^\varepsilon \in \sA(X_{1:n},\varepsilon ) }\widehat{\alpha }^\eps(\kappa ) \le  \widehat{\alpha}(\kappa)< \overline{\alpha}(\kappa ) $, it suffices to obtain
    \[\frac{1}{n}\sum_{i=1}^{n} w(\overline{\alpha}(\kappa)\lvert X_i - \kappa \rvert) < 1-\frac{\eta+1}{n}<1-\frac{\eta}{n}<\inf_{X_{1:n}^\varepsilon \in \sA(X_{1:n},\varepsilon ) } \frac{1}{n}\sum_{i=1}^{n} w(\underline{\alpha}(\kappa)\lvert X_i^\eps - \kappa \rvert)
		\]
		with high probability. We want to bound
		\begin{align*}
			 & \Pr\left[ \frac{1}{n}\sum_{i=1}^{n} w(\overline{\alpha}(\kappa)\lvert X_i - \kappa \rvert) -\mathbb{E}w(\overline{\alpha}(\kappa)\lvert X_1-\kappa \rvert) \ge 1-\frac{\eta+1}{n}-\mathbb{E}w(\overline{\alpha}(\kappa)\lvert X_1-\kappa \rvert)\right] \\
       & =\Pr\left[\frac{1}{n}\sum_{i=1}^{n} w(\overline{\alpha}(\kappa)\lvert X_i - \kappa \rvert) -\mathbb{E}w(\overline{\alpha}(\kappa)\lvert X_1-\kappa \rvert) \ge \frac{\overline{c}-1 }{n}\ln \frac{4}{\delta}-\frac{1}{n}+\xi \eps \right].
		\end{align*}
    We apply \Cref{lemma:bernstein}, which implies that this probability is lesser than \( \frac{\delta}{4} \) if \[
			\frac{\overline{c}-1 }{n}\ln \frac{4}{\delta}-\frac{1}{n}+\xi \eps \ge \sigma\sqrt{\frac{2}{n}\ln \frac{4}{\delta}}+\frac{1}{3n}\ln \frac{4}{\delta},
		\]
		where \( \sigma^2:=\mathbb{E} \lvert w(\overline{\alpha}(\kappa)\lvert X -\kappa \rvert)- \mathbb{E}w(\overline{\alpha}(\kappa)\lvert X-\kappa \rvert)\rvert^2\). We bound
		\begin{align*}
			\sigma^2 & =\mathbb{E}\left[w(\overline{\alpha}(\kappa)\lvert X-\kappa \rvert)- \mathbb{E}w(\overline{\alpha}(\kappa)\lvert X-\kappa \rvert)\right]^2 = \mathbb{E}w(\overline{\alpha}(\kappa)\lvert X-\kappa \rvert)^2- (\mathbb{E}w(\overline{\alpha}(\kappa)\lvert X-\kappa \rvert))^2\\
               & \le 1-\mathbb{E}\left[w(\overline{\alpha }(\kappa ) \lvert X-\kappa  \rvert )\right] = \frac{\overline{c} \ln \frac{4}{\delta }}{n} + (1+2\xi )\varepsilon.  
		\end{align*}
    Thus, it suffices to prove that \[
    \frac{\overline{c}-2}{n}\ln \frac{4}{\delta }+\xi \eps \ge \frac{\sqrt{2\overline{c}}}{n}\ln \frac{4}{\delta }+\sqrt{ \frac{2(1+2\xi )\eps}{n} \ln \frac{4}{\delta }} +\frac{1}{3n}\ln \frac{4}{\delta }. 
		\] 
		Let us consider two different scenarios. The first is \( \sqrt{\frac{2(1+2\xi )\eps}{n}\ln \frac{4}{\delta }}\le \xi \eps \). In this case, the bound is satisfied whenever \[
    \overline{c}-2\ge \sqrt{2\overline{c}} +\frac{1}{3}. 
		\]   
    The second scenario concerns the inequality \( \sqrt{\frac{2(1+2\xi )}{n}\ln \frac{4}{\delta }} > \xi \sqrt{\eps} \). Then, the desired inequality follows if \[
      \overline{c}-2\ge \sqrt{2\overline{c}}+2\xi^{-1}+\frac{13}{3}, 
    \] 
    which is the stronger condition and holds by the definition of \( \overline{c} \). Arguing similarly for the lower bounds on \( \widehat{\alpha }  \), we would like to bound \begin{align*}
      &\Pr\left[\frac{1}{n}\sum_{i=1}^{n} w(\underline{\alpha}(\kappa)\lvert X_i^\eps  - \kappa \rvert)\le 1-\frac{\eta}{n}\right] \le  \Pr\left[\frac{1}{n}\sum_{i=1}^{n} w(\underline{\alpha}(\kappa)\lvert X_i  - \kappa \rvert)\le 1-\frac{\eta}{n}+\eps \right]\\
&=\Pr\left[ \frac{1}{n}\sum_{i=1}^{n} \mathbb{E}\left[w (\underline{\alpha }(\kappa )\lvert X_1-\kappa  \rvert )\right]-w(\underline{\alpha}(\kappa)\lvert X_i  - \kappa \rvert)\ge \mathbb{E}\left[w (\underline{\alpha }(\kappa )\lvert X_1-\kappa  \rvert )\right]-1+\frac{\eta}{n}-\eps \right]
		\end{align*}
    As before, this is bounded by \( \frac{\delta}{4} \) if \[
      \frac{1 -\underline{c}}{n}\ln \frac{4}{\delta}+\frac{6\xi}{7}\varepsilon  \ge \sigma\sqrt{\frac{2}{n}\ln \frac{4}{\delta}}+\frac{1}{3n}\ln \frac{4}{\delta}.
    \] 
    Here, \( \sigma \le \sqrt{\frac{\underline{c}\ln \frac{4}{\delta }}{n}}+\sqrt{\frac{\xi}{7}\varepsilon }   \). In a similar manner as before, we consider two cases. The first is given by \( \sqrt{\frac{2\xi \varepsilon }{7n}\ln \frac{4}{\delta }}\le  \frac{6\xi}{7}\varepsilon   \), while the second is given by \( \sqrt{\frac{2\xi }{7n}\ln \frac{4}{\delta }}>\frac{6\xi }{7}\sqrt{\varepsilon } \). In both cases, it suffices to prove \[ 
			1-\underline{c}\ge  \sqrt{2\underline{c}} + \frac{2}{3},
		\]
		which again is true by the definition of the constants.
	\end{proof}}

	{\renewcommand{\proofname}{Proof of \Cref{lemma:bias_order}}
	\begin{proof}
		We start decomposing the bias as
		\begin{align}
      \sup_{\alpha\in I_\alpha(\kappa)}\lvert  b(\alpha,\kappa )\rvert  & = \sup_{\alpha\in I_\alpha(\kappa)} \lvert  \E[(X-\kappa)w(\alpha |X-\kappa|)]+\kappa-\mu\rvert \nonumber \\
                                                                                     & \leq \sup_{\alpha\in I_\alpha(\kappa)}\lvert  \E[(X-\mu)w(\alpha |X-\kappa|)]\rvert \label{eq:est_bias} \\&+ \sup_{\alpha\in I_\alpha(\kappa)}\lvert  (\kappa-\mu)[1-\E w(\alpha |X-\kappa|)]\rvert \label{eq:kappa_bias},
		\end{align}
		where \eqref{eq:est_bias} estimates the bias introduced by shrinkage and \eqref{eq:kappa_bias} estimates the bias introduced by $\kappa$ weighted by how much $\kappa$ is taken into account in the final estimator. Now we bound each of these terms separately.

		To bound \eqref{eq:est_bias} we start observing that
		$\E[(X-\mu)w(\alpha |X-\kappa|)] = - \E[(X-\mu)(1-w(\alpha |X-\kappa|))]$
		and use Hölder's inequality with $q = \frac{p}{p-1} > 1$ to get
		\begin{align*}
      \sup_{\alpha\in I_\alpha(\kappa)}\lvert  \E[(X-\mu)w(\alpha |X-\kappa|)]\rvert 
			\leq\nu_p\sup_{\alpha\in I_\alpha(\kappa)}\lVert  1-w(\alpha |X-\kappa|)\rVert_{L_q(P)},
		\end{align*}

		To bound the remaining norm we use the definition of $\overline{\alpha}(\kappa)$ and that $w$ takes values on $[0,1]$:
		\[ \sup_{\alpha\in I_\alpha(\kappa)}\E |1 - w(\alpha |X-\kappa| )|^q \leq \sup_{\alpha\in I_\alpha(\kappa)}\E\left[ 1 - w(\alpha |X-\kappa| ) \right] = \frac{\overline{c} }{n}\ln\frac{4}{\delta }+(1+2\xi )\eps  \]

		Thus,
    \[ \sup_{\alpha\in I_\alpha(\kappa)}\lvert  \E[(X-\mu)w(\alpha \lvert X-\kappa \rvert )]\rvert 
			\leq\nu_p \left(\frac{\overline{c} }{n}\ln\frac{4}{\delta }+(1+2\xi )\eps \right)^{1-\frac{1}{p}}. \]
		The bound on \eqref{eq:kappa_bias} is now straightforward:
    \[ \sup_{\alpha\in I_\alpha(\kappa)}\lvert  (\kappa-\mu)[1-\E w(\alpha |X-\kappa|)]\rvert  \leq |\kappa - \mu| \left(\frac{\overline{c} }{n}\ln\frac{4}{\delta }+(1+2\xi )\eps\right) . \]
		Next, we bound $m(\underline{\alpha}(\kappa))$. To that end, we lower bound $\underline{\alpha}(\kappa)$ since
		\begin{equation}\label{eq:m_xp}
      m(\alpha)=\sup_{x\in\R}\lvert  xw(\alpha \lvert  x\rvert )\rvert =\alpha^{-1}\sup_{t\geq0}tw(t)=c_w\alpha^{-1}.
		\end{equation}
    Recall that, by definition, $\E[w(\underline{\alpha}(\kappa)\lvert  X-\kappa\rvert )]=1-\frac{\underline{c}}{n}\ln\frac{4}{\delta }-\frac{\xi}{7}\varepsilon $. Then the property $w(t)\geq(1-t^p)_+\geq1-t^p$ implies $1-\frac{\underline{c}}{n}\ln\frac{4}{\delta }-\frac{\xi}{7}\varepsilon \geq 1-\underline{\alpha}(\kappa)^p\E[\lvert  X-\kappa\rvert ^p]$ and so
		\begin{equation*}
      \underline{\alpha}(\kappa)\geq\left(\frac{\underline{c}}{n}\ln\frac{4}{\delta }+\frac{\xi}{7}\varepsilon \right)^{\frac{1}{p}}\frac{1}{\E[\lvert  X-\kappa\rvert ^p]^{\frac{1}{p}}}\geq\left(\frac{\underline{c} }{n}\ln\frac{4}{\delta }+\frac{\xi}{7}\varepsilon \right)^{\frac{1}{p}}\frac{1}{\nu_p+|\kappa-\mu|}.
		\end{equation*}
		By \eqref{eq:m_xp},
		\begin{equation}\label{eq:malpha_bound}
      m(\underline{\alpha}(\kappa))\leq c_w (\nu_p+\lvert  \kappa-\mu\rvert )\left(\frac{\underline{c} }{n}\ln\frac{4}{\delta }+\frac{\xi}{7}\varepsilon \right)^{-\frac{1}{p}}.
		\end{equation}
		We may bound $\lVert  F_\kappa\rVert_{L_2(P)}$ in the following way:
		\begin{align*}
      \lVert  F_\kappa\rVert^2_{L_2(P)} & =\E\sup_{\alpha\in I_\alpha (\kappa )}\lvert  (X-\kappa)w(\alpha \lvert  X-\kappa\rvert )+\kappa-\mu\rvert ^2\nonumber \\
                                           & =\E\sup_{\alpha\in I_\alpha (\kappa )}\lvert  (X-\mu)w(\alpha |X-\kappa|)+(\kappa-\mu)(1-w(\alpha |X-\kappa|))\rvert ^2\nonumber \\
                                           & \leq 2\E\sup_{\alpha\in I_\alpha (\kappa )}\lvert  (X-\mu)w(\alpha |X-\kappa|)\rvert ^2+2\E\sup_{\alpha\in I_\alpha (\kappa )}\lvert  (\kappa-\mu)(1-w(\alpha |X-\kappa|))\rvert ^2\nonumber \\
			                    & \leq
                          \underbrace{2\E\lvert  (X-\mu)w(\underline{\alpha}(\kappa) |X-\kappa|)\rvert ^2}_{\text{(I)}}
			+
      \underbrace{2\E\lvert  (\kappa-\mu)(1-w(\overline{\alpha}(\kappa) |X-\kappa|))\rvert ^2}_{\text{(II)}}.
		\end{align*}
		In the case \( p\in (1,2)\), we bound (I) using the centered \( p \)-th moment of \( X \) and the fact that \( xw (\alpha \lvert x \rvert ) \) is uniformly bounded over \( x\in \R \): \[
      2\E\lvert X-\mu \rvert^p \lvert  (X-\mu)w(\underline{\alpha}(\kappa) |X-\kappa|)\rvert ^{2-p} \le 2\nu_p^p\left(\lvert \kappa -\mu  \rvert  ^{2-p} + m(\underline{\alpha }(\kappa ))^{2-p}\right).
		\]
		When \( p\ge 2 \), we bound it directly by \( 2\nu_2^2 \).
		(II) is then bounded by means of the implicit definition of \( \overline{\alpha }(\kappa ) \):
    \begin{align*}
      2\E\lvert  (\kappa-\mu)(1-w(\overline{\alpha}(\kappa) |X-\kappa|))\rvert ^2  &\le 2\lvert \kappa -\mu  \rvert^2  \mathbb{E}\lvert 1-w (\overline{\alpha }(\kappa )\lvert X-\kappa \rvert ) \rvert \\
 &= 2 \lvert \kappa -\mu  \rvert^2 \left(\frac{\overline{c} \ln \frac{4}{\delta }}{n}+(1+2\xi )\varepsilon\right).\qedhere
    \end{align*}
	\end{proof}}

	{\renewcommand{\proofname}{Proof of \Cref{thm:shrinkage_concentration_adv}}
	\begin{proof}
    For this proof, denote \( R:=R_{\widehat{\kappa } }\left(\frac{\delta}{4}\right)  \). Then, we may take \( I_{\widehat{\kappa } } = [\mu - R, \mu  + R] \).
    We proceed by applying \Cref{theorem:bias_variance} and \Cref{lemma:bias_order} with \(I_\alpha (\kappa )\) as constructed in \Cref{lemma:alphaorder}. First notice that \Cref{lemma:bias_order} bounds \( \sup_{\kappa \in I_{\widehat{\kappa } }}\sup_{\alpha\in I_\alpha (\kappa )}\lvert b(\alpha ,\kappa ) \rvert  \) by  \begin{align*}
      \inf_{1<q\le p} &\nu_q \left(\frac{\overline{c}}{n} \ln \frac{4}{\delta}+(1+2\xi )\eps \right)^{1-\frac{1}{q}} +  R \left(\frac{\overline{c}}{n} \ln \frac{4}{\delta}+(1+2\xi )\eps \right)\\ &\le  \inf_{1<q\le p} (\nu_q +R)\left(\frac{\overline{c}}{n} \ln \frac{4}{\delta}+(1+2\xi )\eps \right)^{1-\frac{1}{q}}, 
  \end{align*} 
  where we used that \( \frac{\overline{c}}{n} \ln \frac{4}{\delta}+(1+2\xi )\eps <1 \) and that the infimum is over \( q>1 \) for the second inequality. Next, we apply \Cref{theorem:bias_variance}, the bound on \( m(\underline{\alpha }(\kappa )) \) from \Cref{lemma:bias_order} and the fact that \( \overline{c}>1 \) to conclude, with probability at least \( 1-\delta \), that \( \lvert \widehat{\mu} -\mu  \rvert  \) is at most
		\begin{equation}\label{eq:bias_var}
      \begin{split}
        \inf_{q \in (1,p]}&\left(\frac{8\underline{c}^{-\frac{1}{q}}c_w+3}{3}\right)(\nu_q+R)\left(\frac{\overline{c}}{n}\ln\frac{ 4}{\delta } + (1+2\xi )\eps \right)^{1-\frac{1}{q}}\\+ &\sup_{\kappa \in I_{\widehat{\kappa } }} \left(2K+\sqrt{2}\right) \| F_\kappa  \|_{L_2(P)}\sqrt{\frac{1}{n} \ln \frac{4}{\delta}}.
      \end{split}
		\end{equation}
    Now we bound
    \begin{equation}\label{eq:F_bound}
      \sup_{\kappa \in I_{\widehat{\kappa }} }\| F_\kappa \|_{L_2(P)} \leq \inf_{1<q\le 2\wedge p }\sqrt{2} \nu_q^{\frac{q}{2}}\left(\sup_{\kappa \in I_{\widehat{\kappa }}} m(\underline{\alpha}(\kappa))^{1-\frac{q}{2}}+R ^{1-\frac{q}{2}}\right)+\sqrt{2} R \sqrt{\frac{\overline{c}}{n}\ln\frac{4}{\delta}+(1+2\xi )\eps } . 
    \end{equation} 

    We want to bound \( \sup_{\kappa \in I_{\widehat{\kappa }} }\|F_\kappa \|_{L_2(P)}\sqrt{\frac{1}{n} \ln \frac{4}{\delta }}  \). The product of \( \sqrt{\frac{1}{n} \ln \frac{4}{\delta }}  \) with the second summand in \eqref{eq:F_bound} is bounded by
    \[
      \inf_{1<q\le p} \sqrt{2}R\left(\frac{\overline{c}}{n} \ln \frac{4}{\delta }+(1+2\xi )\varepsilon \right)^{1-\frac{1}{q}}.
    \] 
    For the first summand in \eqref{eq:F_bound}, we use \eqref{eq:malpha_bound} to get 
    \begin{align*}
      &\inf_{1<q\le 2\wedge p } \nu_q^{\frac{q}{2}}\left(\sup_{\kappa \in I_{\widehat{\kappa }}} m(\underline{\alpha}(\kappa))^{1-\frac{q}{2}}+R ^{1-\frac{q}{2}}\right)\sqrt{\frac{2}{n} \ln \frac{4}{\delta }} \\
      &\le    \inf_{1<q\le 2\wedge p } \nu_q^{\frac{q}{2}}\left(c_w^{1-\frac{q}{2}}(\nu _q+R)^{1-\frac{q}{2}}\left(\frac{\underline{c}}{n} \ln \frac{4}{\delta }+\frac{\xi}{7}\varepsilon \right)^{\frac{1}{2}-\frac{1}{q}} +R ^{1-\frac{q}{2}}\right)\sqrt{\frac{2}{n} \ln \frac{4}{\delta }}  \\
      &\le \inf_{1<q\le 2\wedge p}\sqrt{\frac{2}{\underline{c}}} ((c_w\vee 1)+1) \left(\nu_q+\nu _q^{\frac{q}{2}}R^{1-\frac{q}{2}}\right)\left(\frac{\underline{c}}{n} \ln \frac{4}{\delta }\right)^{1-\frac{1}{q}},
    \end{align*} 
    where we used \( \frac{\underline{c}}{n}\ln \frac{4}{\delta }+\frac{\xi}{7}\varepsilon  \le  1 \) and that \( \frac{1}{2}-\frac{1}{q}\le 0 \) in the last inequality. 
    We conclude that, with probability at least \( 1-\delta \), for all \( X_{1:n}^\varepsilon \), \( \lvert \widehat{\mu} -\mu  \rvert  \) is at most \[
      C_{w,\xi}^{\prime}\left\{\inf_{1<q\le 2\wedge p} \left(\nu _q + \nu _q^{\frac{q}{2}}R^{1-\frac{q}{2}}\right)\left(\frac{1}{n }\ln \frac{4}{\delta }\right)^{1-\frac{1}{q}}+\inf_{1<q\le p} \left(\nu _q + R\right)\left(\frac{1}{n }\ln \frac{4}{\delta } + \varepsilon \right)^{1-\frac{1}{q}}\right\}
    \] 
    for some $C^\prime_{w,\xi}>0$ depending only on $w$ and \( \xi  \).

		The conclusion \eqref{eq:conclusion_concentration} is then proved by using the bound on the error \[
      \sup_{X_{1:n}^\varepsilon \in \sA(X_{1:n},\varepsilon )} \lvert \widehat{\mu}^\eps -\widehat{\mu}   \rvert  \le \inf_{1<q\le p} \left( 6\eps + \frac{8}{n} \ln \frac{4}{\delta }\right)c_w (\nu _q + R)\left(\frac{\underline{c}}{n}\ln \frac{4}{\delta }+\frac{\xi}{7}\varepsilon \right)^{-\frac{1}{q}},
		\]  
given by \Cref{lemma:contamination_error}.
	\end{proof}}
\begin{proposition}[Concentration of empirical quantiles]\label[proposition]{prop:empirical_quantiles}
  Let $\gamma \in (0,1)$. Then, for every \( \delta \ge 2e^{-c_\gamma m} \), the empirical $\gamma$-quantile $\widehat{Q}_{\gamma }(Y_{1:m}) $ satisfies $R_{\widehat{\kappa}}\left(\delta\right) \leq C_\gamma \nu_p,$ for every $P\in\sP_p$, where \( c_\gamma= \frac{((1-\gamma )\wedge \gamma)^2}{16} \) and \( C_\gamma =\frac{2}{\gamma \wedge (1-\gamma )} \). 
  
\end{proposition}
\begin{proof}
  The proof follows by a simple application of Bernstein's inequality. Let \( F^-(x)=\Pr[X_1<x] \text{ and } Q^-_t = \sup \{x \in \R: F^-(x)\le t\}  \). It holds that \( F^-(Q^-_t)=t \) for \( t \in (0,1) \).  
 \[
   \Pr\left[\widehat{\kappa } \ge Q_t^- \right]\le \Pr\left[\frac{1}{n}\sum_{i=1}^n \mathbf{1}_{X_i\ge Q_t^-}\ge 1-\gamma  \right]=\Pr\left[\frac{1}{n}\sum_{i=1}^n \mathbf{1}_{X_i\ge Q_t^-} - (1-t)\ge  -\gamma +t \right]
 \]
 Take \( t \) to be \( \gamma + \sqrt{\frac{1}{n}\ln \frac{2}{\delta }}  \). Then, \[
\Pr\left[\widehat{\kappa } \ge Q_t^- \right] \le  \Pr\left[\frac{1}{n}\sum_{i=1}^n \mathbf{1}_{X_i\ge Q_t^-} - (1-t)\ge \sqrt{\frac{1}{n}\ln \frac{2}{\delta }}  \right]\le \frac{\delta}{2},
 \] 
 as long as \( \delta \ge 2e^{-\frac{7}{10}n} \). 
By Markov's inequality, we can bound \( Q_t^- \) as follows:
\begin{align*}
  1-t = \Pr[X_1 \ge  Q_t^-] &\le  \frac{\nu_p^p}{\lvert Q_t^--\mu \rvert^p}, \text{ if \( Q_t^- \ge \mu  \) }\\
  t = \Pr[X_1 <  Q_t^-] &\le  \frac{\nu_p^p}{\lvert Q_t^--\mu \rvert^p}, \text{ if \( Q_t^- < \mu  \) }
\end{align*} 
implying that \( Q_t^- \le \mu + \nu_p \left[\left(\gamma + \sqrt{\frac{1}{n}\ln \frac{2}{\delta }}\right)\wedge \left(1-\gamma - \sqrt{\frac{1}{n}\ln \frac{2}{\delta }}\right)  \right]^{-1/p}. \) In particular, if \( \delta  \) is greater than  \( 2\exp\left\{-\frac{(1-\gamma )^2n}{4}\right\} \), then \( \widehat{\kappa } -\mu \le \nu_p\left(\frac{2}{\gamma \wedge (1-\gamma ) }\right)^{\frac{1}{p}}  \). 
We provide a lower-bound similarly: let \( F(x)=\Pr[X_1 \le x], Q_t = \inf \{x \in  \R: F(x)\ge  t\}  \). Then,
 \[
   \Pr\left[\widehat{\kappa } \le Q_t \right]\le \Pr\left[\frac{1}{n}\sum_{i=1}^n \mathbf{1}_{X_i\le Q_t}\ge \frac{n-1}{n}\gamma \right]=\Pr\left[\frac{1}{n}\sum_{i=1}^n \mathbf{1}_{X_i\le Q_t} - t\ge \frac{n-1}{n}\gamma -t \right].
 \]
 Taking \( t=\gamma -2\sqrt{\frac{1}{n}\ln \frac{2}{\delta }}  \) leads to \( \Pr[\widehat{\kappa }\le Q_t ] \le \frac{\delta}{2} \) as before. Using Markov's inequality again yields the bound \( Q_t \ge \mu - \nu _p\left[\left(\gamma - 2\sqrt{\frac{1}{n}\ln \frac{2}{\delta }}\right)\wedge \left(1-\gamma + 2\sqrt{\frac{1}{n}\ln \frac{2}{\delta }}\right)  \right]^{-1/p} \). If \( \delta \ge \exp \left\{-\frac{\gamma^2n}{16}\right\}  \), \( \widehat{\kappa } -\mu \ge -\nu_p\left(\frac{2}{\gamma \wedge (1-\gamma )}\right)^{\frac{1}{p}}  \) holds with high probability. Combining the lower and upper bounds, when \( \delta \ge 2e^{-\frac{((1-\gamma )\wedge \gamma)^2n}{16}} \), with probability at least \( 1-\delta  \), \[
   \lvert \widehat{\kappa }-\mu   \rvert \le \nu_p \left(\frac{2}{\gamma\wedge (1-\gamma ) }\right)^{\frac{1}{p}}\le C_\gamma \nu _p \qedhere
 \] 

\end{proof}

\begin{proposition}\label[proposition]{prop:exact_tm}
  Let \( \widehat{\kappa }=(Y_{(\eta /2 + 1 )} + Y_{(n- \eta /2)})/2  \) be the average of two empirical quantiles of a possibly contaminated version of the sample \( Y_{1:n} \), where \( \eta = \ln \frac{4}{\delta } +  \left\lfloor (1+\xi ) \varepsilon n \right\rfloor\), for \( \xi  \) such that \( \xi \varepsilon n \ge 2 \). Then, for every \( \delta \ge 4\exp\left\{-n(6-3\xi \varepsilon )\right\} \), \( R_{\widehat{\kappa } }\left(\frac{\delta}{2}\right) \le \nu_p\left(\frac{1}{12n}\ln \frac{4}{\delta }+\frac{\xi}{4} \varepsilon \right)^{-\frac{1}{p}}    \) for every \( p>1 \).
\end{proposition}
\begin{proof}
 The proof follows similarly to the one of \Cref{prop:empirical_quantiles}. We start by noting that \[
   L=X_{(\frac{1}{2}\ln \frac{4}{\delta }+ \left\lfloor (1+\xi )\varepsilon n \right\rfloor - \left\lfloor \varepsilon n \right\rfloor)} \le Y_{(\eta  /2 +1)} \le \widehat{\kappa } \le Y_{(n - \eta /2)} \le X_{(n - \frac{1}{2}\ln \frac{4}{\delta } - \left\lfloor (1+\xi ) \varepsilon n \right\rfloor + \left\lfloor \varepsilon n \right\rfloor)}=U.  
 \] 
 We then lower-bound \( X_{(\frac{1}{2}\ln \frac{4}{\delta }+ \left\lfloor (1+\xi )\varepsilon n \right\rfloor - \left\lfloor \varepsilon n \right\rfloor)} \) and upper-bound \( X_{(n - \frac{1}{2}\ln \frac{4}{\delta } - \left\lfloor (1+\xi ) \varepsilon n \right\rfloor + \left\lfloor \varepsilon n \right\rfloor)} \) by quantiles of the distribution of \( X_1 \). For the lower bound, we have that:
 \begin{align}
   \mathbb{P}\left[L \le Q_t\right] &\le \mathbb{P}\left[\sum_{i=1}^n\mathbf{1}_{X_i \le Q_t} \ge \frac{1}{2}\ln \frac{4}{\delta }+ \left\lfloor (1+\xi ) \varepsilon n \right\rfloor - \left\lfloor \varepsilon n \right\rfloor\right] \label{eq:exact_tm1}\\
                                    & \le \mathbb{P}\left[\frac{1}{n} \sum_{i=1}^n \mathbf{1}_{X_i \le Q_t} - t \ge \frac{1}{2n} \ln \frac{4}{\delta } + \frac{\xi}{2} \varepsilon  - t\right] \le \frac{\delta}{4},\label{eq:exact_tm2} 
 \end{align} 
 whenever \( \frac{1}{2n} + \frac{\xi}{2}\varepsilon \ln \frac{4}{\delta } -t \ge \frac{1}{3n}\ln \frac{4}{\delta } + \sqrt{\frac{2t(1-t)}{n}\ln \frac{4}{\delta }}  \), by Berstein's inequality. This holds true for \( t =  \frac{1}{12n} \ln \frac{4}{\delta } + \frac{\xi}{4}\varepsilon   \), so let us fix such \( t \). To obtain \eqref{eq:exact_tm2} from \eqref{eq:exact_tm2}, we used that \( \left\lfloor (1+\xi ) \varepsilon n \right\rfloor - \left\lfloor \varepsilon n \right\rfloor \ge \xi \varepsilon n -1 \ge \frac{\xi}{2} \varepsilon n\), since, by assumption, \( \xi  \varepsilon n \ge 2\). We proceed exactly the same as in the proof of \Cref{prop:empirical_quantiles}, concluding that 
 \[
   Q_t -\mu \ge - \nu_p (t \wedge (1-t))^{-\frac{1}{p}} =  -\nu_p\left(\frac{1}{12n} \ln \frac{4}{\delta } + \frac{\xi}{4} \varepsilon \right)^{-\frac{1}{p}},
 \] 
 where we used that \( t \le \frac{1}{2} \le 1-t \) which is a consequence of the assumption on the confidence level \( \delta \ge 4\exp\left\{-n(6-3\xi \varepsilon )\right\} \). 
 The upper bound follows symmetrically, with \( t=1-\frac{1}{12n}\ln \frac{4}{\delta } -\frac{\xi}{4}\varepsilon  \). We conclude that \( R_{\widehat{\kappa } }\left(\frac{\delta}{2}\right) \le \nu_p\left(\frac{1}{12n} \ln \frac{4}{\delta } + \frac{\xi}{4}\varepsilon \right)^{-\frac{1}{p}}    \) for every \( p>1 \).
\end{proof}

Finally, we show that the normalized shrinkage estimator \( \widehat{\mu }_W  \) defined in \eqref{eq:def_estimator_weighted} behaves at least as well as the non-normalized one:

\begin{proposition}\label[proposition]{prop:normalized_shrinkage}
  Fix \( \kappa \in \R \). Then \( \lvert\widehat{\mu}(\kappa;X_{1:n}) - \widehat{\mu }_W(\kappa;X_{1:n})  \rvert \le m(\widehat{\alpha }(\kappa; X_{1:n} ) )\frac{n-S}{S},   \) where \( S \) is defined by \( S:=\sum_{i=1}^n w(\widehat{\alpha }(\kappa ) \lvert X_i-\kappa  \rvert  ) \).    
\end{proposition}
\begin{proof}
  By the definitions of \( \widehat{\mu }  \) and \( \widehat{\mu }_W  \), we have that \( \lvert \widehat{\mu }_W(\kappa;X_{:n})-\widehat{\mu }(\kappa ;X_{1:n}) \rvert \) is equivalently written as
  \begin{align*}
    &\left\lvert \frac{1}{S}\sum_{i=1}^n X_iw(\widehat{\alpha }(\kappa; X_{1:n} ) \lvert X_i-\kappa  \rvert  ) - \kappa - \frac{1}{n}\sum_{i=1}^n (X_i-\kappa )w(\widehat{\alpha }(\kappa; X_{1:n} ) \lvert X_i-\kappa  \rvert  ) \right\rvert \\
       &\le \left(\frac{1}{S}-\frac{1}{n}\right)\sum_{i=1}^n \lvert X_i-\kappa \rvert w(\widehat{\alpha }(\kappa;X_{1:n} ) \lvert X_i-\kappa  \rvert  )\le m(\widehat{\alpha }(\kappa;X_{1:n} ) )\frac{n-S}{S}.\qedhere
  \end{align*} 
\end{proof}
It immediately follows that the concentration guarantee provided for \( \widehat{\mu }^\varepsilon   \) also holds for \( \widehat{\mu }_W^\varepsilon=\widehat{\mu }_W(\widehat{\kappa }; X_{1:n}^\varepsilon  )    \) under the adversarial contamination scenario (with different constants), since \( S \ge n-\eta -1 - \left\lfloor \varepsilon n \right\rfloor\) almost surely (\Cref{lemma:rho_bounds}), \( \inf_{X_{1:n}^\varepsilon \in \sA(X_{1:n},\varepsilon )}\widehat{\alpha }^\varepsilon (\kappa ) \ge \underline{\alpha }(\kappa )  \) with high probability (\Cref{lemma:alphaorder}), and \( m(\underline{\alpha }(\kappa ))\le c_w(\nu_p + \lvert \kappa -\mu  \rvert ) \left(\frac{\underline{c}}{n}\ln \frac{4}{\delta } + \frac{\xi}{7}\varepsilon \right)^{-\frac{1}{p}}  \) (\Cref{lemma:bias_order}).

\section{Removing \Cref{assum:kappa_independent}}\label[appendix]{sec:kappa_dependent}

To show \Cref{assum:kappa_independent} can be dispensed with in the case \( w(t)=\mathbf{1}_{t<1} \), we first show that, for any given sample \( x_{1:n} \) and initial estimates \( \kappa_1 \) and \( \kappa_2 \), then \( \widehat{\mu }(\kappa_1)  \) is not too far from \( \widehat{\mu }(\kappa_2)  \) (with this distance depending on the distance between the base estimates). Simultaneously, \Cref{thm:shrinkage_concentration_adv} ensures that \( \widehat{\mu }(\mu;X_{1:n}^\varepsilon )  \) attains optimal concentration bounds (\( \mu  \) is deterministic and thus independent of the sample). We conclude that \( \widehat{\mu }(\widehat{\kappa };X_{1:n}^\varepsilon  )  \) also concentrates appropriately, with an additional error term depending on \( \lvert \widehat{\kappa } -\mu   \rvert  \) of lesser order than the one present in \Cref{thm:shrinkage_concentration_adv}.

Let us start by formally stating the result ensuring closeness of the shrinkage estimators based on different initial estimates:

\begin{lemma}\label{lemma:base_estimator_error_unbalanced}
  Let \( \widehat{\mu }_W(\kappa )= \widehat{\mu }_W(\kappa;X_{1:n} )  \) as in \eqref{eq:def_estimator_weighted}. Then, for any \( \kappa _1, \kappa _2 \in \R  \) such that \( S_1=\sum_{i=1}^n w(\widehat{\alpha }(\kappa _1) \lvert X_i -\kappa_1 \rvert  ),\, S_2=\sum_{i=1}^n w(\widehat{\alpha }(\kappa _2) \lvert X_i -\kappa_2 \rvert  ) \), we have
  \[
    \lvert \widehat{\mu }_W(\kappa _1) -\widehat{\mu }_W(\kappa _2)  \rvert  \le \frac{n-S_1 \wedge S_2}{S_1 \wedge S_2}(m(\widehat{\alpha }(\kappa _1) ) + m(\widehat{\alpha }(\kappa _2) )+ \lvert \kappa_1 -\kappa_2 \rvert ),
  \] 
  where \( m(\alpha )=\sup_{t \ge 0} t w(\alpha t) \). 
\end{lemma}
\begin{proof}
Without loss of generality, we assume that \( S_1 \le S_2 \). Note that \[
  \widehat{\mu }_W(\kappa_1)-\widehat{\mu }_W(\kappa _2)=  \sum_{i=1}^n X_i\left(\frac{w(\widehat{\alpha }(\kappa _1) \lvert X_i -\kappa _1 \rvert  )}{S_1} -  \frac{w(\widehat{\alpha }(\kappa _2) \lvert X_i -\kappa _2 \rvert  )}{S_2} \right)
  \] Denote \( \Delta_i = \frac{w(\widehat{\alpha }(\kappa_1) \lvert X_i -\kappa_1  \rvert )}{S_1}-\frac{w(\widehat{\alpha }(\kappa_2)  \lvert  X_i -\kappa _2 \rvert )}{S_2} \). Define \( I_+:=\{i \in [n]: \Delta_i > 0\}, I_-:= \{i \in [n]: \Delta_i < 0\}   \). Note that, since both sets of weights have the same total sum, then \( \sum_{i \in I_+}\Delta_i=-\sum_{i \in I_-}\Delta_i  = W \le \sum_{i\in I_+} \frac{1}{S_1} - \frac{w(\widehat{\alpha }(\kappa_2) \lvert X_i -\kappa_2 \rvert  )}{S_2} \le  \frac{n-S_1}{S_1} \). In this last inequality, we used that \( w(\widehat{\alpha }(\kappa _2) \lvert X_1-\kappa _2 \rvert  ) \le  1 \le S_2 / S_1 \) to extend the sum from \( I_+ \) to \( [n] \). There exists a coupling \( (\gamma_{i,j})_{i \in I_+, j \in I_-}>0 \) such that \[
\Delta_i = \sum_{j \in I_-} \gamma_{i,j}, \text{for \( i\in I_+ \) and } -\Delta_j = \sum_{i \in I_+} \gamma_{i,j}, \text{for \( j\in I_- \)}.
\] 
Then we can rewrite \[
  \sum_{i=1}^n X_i \Delta_i = \left( \sum_{i \in I_+} X_i \sum_{j \in I_-} \gamma_{i,j} - \sum_{j \in I_-} X_j \sum_{i \in I_+} \gamma_{i,j}  \right) = \sum_{i \in I_+}\sum_{j \in I_-} \gamma_{i,j}(X_i -X_j).
\] 
Next, we use the triangle inequality to bound \begin{equation}\label{eq:decomposition_unbalanced}
  \left\lvert \sum_{i=1}^n X_i \Delta_i \right\rvert \le \sum_{i \in I_+}\sum_{j \in I_-} \gamma_{i,j} \lvert X_i -X_j \rvert \le \sum_{i \in I_+}\sum_{j \in I_-} \gamma_{i,j}(\lvert X_i -\kappa_1 \rvert + \lvert \kappa_1 -\kappa_2 \rvert + \lvert X_j -\kappa_2 \rvert ).
\end{equation}  
Now, notice that \( \sum_{i \in  I_+}\sum_{j \in I_-}\gamma_{i,j} \lvert X_i -\kappa _1 \rvert= \sum_{i \in I_+} \Delta_i \lvert X_i -\kappa_1 \rvert   \). Also observe that
\[
  \Delta_i \le  w(\widehat{\alpha }(\kappa_1) \lvert X_i -\kappa_1 \rvert  ) \cdot \left(\frac{1}{S_1}-\frac{w(\widehat{\alpha }(\kappa_2) \lvert X_i - \kappa _2 \rvert )}{S_2}\right).
\]
This implies that
\[
  \sum_{i \in I_+} \Delta_i \lvert X_i -\kappa _1 \rvert \le m(\widehat{\alpha }(\kappa _1)  ) \sum_{i \in I_+} \left(\frac{1}{S_1}-\frac{w(\widehat{\alpha }(\kappa_2) \lvert X_i -\kappa _2 \rvert  )}{S_2}\right) \le \frac{n-S_1}{S_1} \cdot m(\widehat{\alpha }(\kappa_1)   ).
\]
Arguing similarly for the term involving \( \lvert X_j -\kappa _2 \rvert  \) in \eqref{eq:decomposition_unbalanced}, we have that
\[
  \sum_{j \in I_-} -\Delta_j \lvert X_j -\kappa _2 \rvert \le m(\widehat{\alpha }(\kappa _2)  ) \sum_{j \in I_-} \left(\frac{1}{S_2}-\frac{w(\widehat{\alpha }(\kappa _1)\lvert X_j-\kappa _1 \rvert  )}{S_1}\right) \le \frac{n-S_1}{S_1} \cdot m(\widehat{\alpha }(\kappa_2)   ).
\] 
Finally, we also bound \( \sum_{i \in I_+}\sum_{j \in I_-} \gamma_{i,j} \lvert \kappa _1 - \kappa _2 \rvert \le W \lvert \kappa _1 - \kappa _2 \rvert \le \frac{n-S_1}{S_1}\lvert \kappa _1 -\kappa _2 \rvert  \) and put everything together to obtain the desired result.
\end{proof}

We would like to employ this bound setting \( \kappa_1=\mu  \) and \( \kappa_2=\widehat{\kappa }  \). We still have to control the term \( m(\widehat{\alpha }(\widehat{\kappa } ))  \) in order to attain a bound that only depends on \( \widehat{\kappa }  \) through \( \lvert \widehat{\kappa } -\mu  \rvert  \), as desired. In order to do this without assuming independence of \( \widehat{\kappa }  \) from the sample, we take advantage of the particular structure of \( \widehat{\alpha }  \) enforced by the function \( w(t)=\mathbf{1}_{t<1} \). Indeed, the following proposition shows that \( \kappa \mapsto m(\widehat{\alpha }(\kappa ) ) \) is \( 1 \)-Lipschitz:  
\begin{lemma}
  Assume that \( w(t)=\mathbf{1}_{t<1} \). Then, for any \( \kappa_1, \kappa_2 \in \R  \), \( \lvert m(\widehat{\alpha }(\kappa _1) ) - m(\widehat{\alpha }(\kappa _2) )  \rvert  \le \lvert \kappa _1 -\kappa _2 \rvert  \).   
\end{lemma}
\begin{proof}
We start by recalling the definition of \( \widehat{\alpha }(\kappa )  \):
\begin{align}
  (\widehat{\alpha }(\kappa  ))^{-1} &= \left(\inf \left\{\alpha > 0: \sum_{i=1}^n w(\alpha \lvert X_i -\kappa  \rvert  ) \le n - \eta  \right\}\right)^{-1}\nonumber\\
                                     &=\sup \left\{r > 0: \sum_{i=1}^n \mathbf{1}_{\lvert X_i -\kappa  \rvert < r } \le n - \eta  \right\}.\label{eq:radius_alpha}
\end{align} 
Notice that, for any \( \gamma >0 \),  \( \lvert X_i - \kappa _1 \rvert < (\widehat{\alpha }(\kappa_1) )^{-1} + \gamma \implies \lvert X_i -\kappa _2 \rvert < (\widehat{\alpha }(\kappa_1) )^{-1} + \lvert \kappa _1 - \kappa _2 \rvert +\gamma   \). Thus, \[
  \sum_{i=1}^n \mathbf{1}_{\lvert X_i -\kappa _2 \rvert < (\widehat{\alpha }(\kappa_1) )^{-1} + \lvert \kappa _1 -\kappa _2 \rvert +\gamma  } \ge \sum_{i=1}^n \mathbf{1}_{\lvert X_i -\kappa _1 \rvert < (\widehat{\alpha }(\kappa_1) )^{-1} + \gamma } > n - \eta.
\]  
Therefore, all values strictly greater than \( (\widehat{\alpha }(\kappa_1 ))^{-1}+\lvert \kappa_1 - \kappa _2 \rvert   \) are out of the set in \eqref{eq:radius_alpha} and greater than its supremum \( (\widehat{\alpha }(\kappa _1))^{-1}  \). To conclude the lemma we notice that \( m(\alpha)=\sup_{t\ge 0}tw(\alpha t)=\sup_{t\ge 0}t\mathbf{1}_{t<\alpha^{-1} }=\alpha^{-1} \). 
\end{proof}
Using the previous two lemmas, we can now show that the shrinkage estimator based on the data-dependent initial estimate \( \widehat{\kappa }  \) also concentrates at the optimal rate, with an additional error term depending on \( \lvert \widehat{\kappa } -\mu  \rvert  \):
\begin{equation}\label{eq:dependent_bound_1}
  \lvert \widehat{\mu }_W(\widehat{\kappa } ) - \widehat{\mu }_W(\mu )   \rvert \le 2 \cdot\frac{n-S_{\mu}  \wedge S_{\widehat{\kappa } }}{S_{\mu } \wedge S_{\widehat{\kappa } }}(m(\widehat{\alpha }(\mu  ) )+ \lvert \widehat{\kappa } -\mu  \rvert ),
\end{equation} 
where \( S_\mu = \sum_{i=1}^n w(\widehat{\alpha }(\mu )\lvert X_i-\mu  \rvert  ) \) and \( S_{\widehat{\kappa } }=\sum_{i=1}^n w(\widehat{\alpha }(\widehat{\kappa } ) \lvert X_i - \widehat{\kappa }  \rvert  ) \).  
We have already shown that \( \widehat{\mu }_W(\mu )  \) concentrates well (\Cref{thm:shrinkage_concentration_adv} combined with \Cref{prop:normalized_shrinkage}) and that, on the same event, \( m(\widehat{\alpha }(\mu ) ) \le \inf_{q \in (1,p]}c_w \nu _q (\frac{\underline{c}}{n} \ln \frac{4}{\delta } + \frac{\xi}{7}\varepsilon )^{-\frac{1}{q}} \) (\Cref{lemma:alphaorder}). Now, let us show that \( S_\mu \wedge S_{\widehat{\kappa } }  \) is properly lower-bounded. To that end, we need a uniform version of \Cref{lemma:rho_bounds} that holds for all \( \kappa \in I_{\widehat{\kappa } } \) simultaneously. This is given by the following lemma, under the slightly stronger assumption that the distributions \( P \in \sP \) are absolutely continuous with respect to the Lebesgue measure: 
	\begin{lemma}\label[lemma]{lemma:rho_bounds_uniform}
    Assume \( X_{1:n} \sim P^{\otimes n} \), where \( P \) is absolutely continuous with respect to the Lebesgue measure. Then, with probability \( 1 \), for all \( \kappa \in \R \) and \( X_{1:n}^\varepsilon  \in \sA(X_{1:n},\varepsilon ) \)  
    \[ n-\eta-2-\left\lfloor n \varepsilon  \right\rfloor \leq \sum_{i=1}^n w\left(\widehat{\alpha}\left(\kappa; X_{1:n}^\varepsilon \right) |X_i^\varepsilon  - \kappa|\right)\leq n - \eta.  \]
	\end{lemma}
  \begin{proof}
    For the purpose of this proof, denote \( \widehat{\alpha }= \widehat{\alpha }(\kappa;X_{1:n}^\varepsilon )   \) Define
		\[
      S_{X^\varepsilon _{1:n}}(\alpha):=\sum_{i=1}^n w(\alpha |X_i^{\varepsilon} - \kappa|).
    \] As in \Cref{lemma:rho_bounds}, the upper bound is a consequence of the definition of $\widehat{\alpha}$ and right-continuity of $w$:
		\[
      S_{X_{1:n}^{\varepsilon}}\left(\widehat{\alpha}\right)=\lim_{\alpha\downarrow\widehat{\alpha}}S_{X_{1:n}^{\varepsilon}}(\alpha)\leq n-\eta
		\]We proceed by proving the lower bound.
    Assume that there exists \( \kappa \in \R \) such that $S_{X_{1:n}^{\varepsilon}}\left(\widehat{\alpha}\right) < n-\eta-2-\left\lfloor n \varepsilon  \right\rfloor,$ with positive probability. By definition of $\widehat{\alpha}$, $\lim_{\alpha\uparrow\widehat{\alpha}}S_{X_{1:n}^{\varepsilon}}(\alpha)\geq n-\eta.$ Thus,
		\[
      \lim_{\alpha\uparrow\widehat{\alpha}}S_{X_{1:n}^{\varepsilon}}(\alpha)-S_{X_{1:n}^{\varepsilon}}(\widehat{\alpha})>2+ \left\lfloor n \varepsilon  \right\rfloor
		\]
    and $\widehat{\alpha}$ is a discontinuity of $\alpha\mapsto w(\alpha \lvert  X_i^{\varepsilon}-\kappa\rvert )$ for at least \( 3 + \left\lfloor n \varepsilon  \right\rfloor \)  different indices. Since \( X_{1:n}^\varepsilon \neq X_{1:n} \) holds for at most \( \left\lfloor n \varepsilon  \right\rfloor \) indices, then there are at least three different indices \( i,j,k \) such that \( \widehat{\alpha }  \) is a discontinuity of \( \alpha  \mapsto w(\alpha \lvert X_i - \kappa  \rvert )  \). Let \( \sD \) denote the discontinuities of \( w \), which are strictly positive since \( w(0)=1 \) and \( w \) is right-continuous. Define \( \sD / \sD = \left\{\frac{d_1}{d_2}: d_1,d_2 \in \sD \right\}\). Observe that this set is countable (since \( \sD \) is countable). Then the indices \( i,j,k \) are such that \( \lvert X_i - \kappa  \rvert, \lvert X_j - \kappa  \rvert, \lvert X_k - \kappa  \rvert >0 \) and \begin{equation}\label{eq:discontinuities}
      \frac{\lvert X_i - \kappa \rvert}{\lvert X_j - \kappa  \rvert },        \frac{\lvert X_i - \kappa \rvert}{\lvert X_k - \kappa  \rvert } \in \sD / \sD. 
    \end{equation}   
    Let us break the implications of \( \frac{\lvert   X_i-\kappa\rvert }{\lvert X_j -\kappa  \rvert } \in \sD / \sD \) into different cases. The first case concerns the scenario in which \( \lvert X_i- \kappa  \rvert = \lvert X_j -\kappa  \rvert \). In this case, either \( X_i=X_j \) or \( \kappa = \frac{X_i+X_j}{2} \). In the second case, the ratio \( \frac{\lvert X_i - \kappa  \rvert }{\lvert  X_j -\kappa  \rvert } =t \) for \( t \in  \sD / \sD \) with \( t \neq  1 \).
    This implies that either \( \kappa =\frac{X_i-tX_j}{1-t} \) or \(   \kappa =\frac{X_i+tX_j}{1+t} \). 

    Then, \eqref{eq:discontinuities} holding for some \( \kappa  \) implies that there exists a linear functional \( T \) with range \( \R \) and such that \( T(X)=T(X_1,\ldots,X_n)=0 \). For example, when \( X_i = X_j\), \( T \) may be chosen as \( T(x)=x_i-x_j \). The other cases lead to the same conclusion once we equate both identities for \( \kappa  \). Since the range of \( T \) is \( \R \), its kernel has dimension \( n-1 \) and the probability that \( T(X)=0 \) is zero, by absolute continuity with respect to the Lebesgue measure. A union bound argument over these cases, the finitely-many triplets \( (i,j,k) \) and the countable set \( \sD / \sD \) concludes the proof.
	\end{proof}
  In order to use this result to prove the main theorem of this section, we replace \Cref{assum:kappa_independent} by the following one, which is slightly stronger but allows us to circumvent the independence requirement:
\renewcommand\theassumption{3\ensuremath{^\prime}}
\begin{assumption}\label{assum:abs_continuous}
	The distributions \( P\in \sP \) are absolutely continuous with respect to the Lebesgue measure.
\end{assumption}\renewcommand\theassumption{\arabic{assumption}}
The same remark we made after stating \Cref{assum:kappa_independent} holds in this case; if our distribution \( P \) is not absolutely continuous, we can always add a small amount of noise to the data to make it so, without affecting the concentration guarantees of the estimator by more than an arbitrarily small amount.

We can now state and prove a version of \Cref{thm:shrinkage_concentration_adv} that does not require \Cref{assum:kappa_independent} for the case \( w(t)=\mathbf{1}_{t<1} \). The proof of this result is essentially a recollection of the previous lemmas and discussions.

\begin{theorem}[Concentration of trimmed shrinkage estimators under adversarial contamination]\label{thm:trimmed_concentration_adv}
  Let $X_{1:n}\sim P^{\otimes n}$ with $P\in\sP$. Let \( \eps \in [0,\frac{1}{2}) \). Under \Cref{assum:abs_continuous}, set $\eta = \ln\frac{4}{\delta}+(1+\xi )\eps n $ for some \( \xi>0 \) and \( w(t)=\mathbf{1}_{t<1} \). Assume \( 2\cdot(1+\xi)\eps + \frac{\overline{c}}{n}\ln \frac{4}{\delta } \le c_0< 1 \) for some absolute constant \( c_0 \in (0,1) \), where \( \overline{c}=16+4\xi^{-1} \). If $\widehat{\mu}_W^\eps  =  \widehat{\mu}_W(\widehat{\kappa };X_{1:n}^\eps) $ is as in \eqref{eq:def_estimator_weighted} then, with probability at least $1-\frac{3\delta}{4}$, for all \( X_{1:n}^\varepsilon \in \sA(X_{1:n},\varepsilon ) \), \( \lvert \widehat{\mu }_W^\varepsilon -\mu   \rvert  \) is at most  
	\begin{equation}\label{eq:conclusion_concentration_trimmed}
    C^T_{w,\xi }\left\{\inf_{1<q\le 2\wedge p} \nu _q \left(\frac{1}{n }\ln \frac{4}{\delta }\right)^{1-\frac{1}{q}}+\inf_{1<q\le p} \nu _q \varepsilon ^{1-\frac{1}{q}}\right\} +\frac{6 \ln \frac{4}{\delta } + (4+2\xi )\varepsilon n}{n(1-c_0)}\cdot \lvert \widehat{\kappa } -\mu  \rvert 
	\end{equation}
  for $C^T_{w,\xi}>0$ depending only on $w$ and \( \xi  \).

\end{theorem}
\begin{proof}
  We start by breaking down the error \( \lvert \widehat{\mu }_W(\widehat{\kappa };X_{1:n}^\varepsilon  ) - \mu  \rvert  \):   
  \begin{align*}
    &\lvert \widehat{\mu }_W(\widehat{\kappa };X_{1:n}^\varepsilon  )  -\mu  \rvert  \\
    &\hspace{4em}\le \underbrace{\lvert \widehat{\mu }_W(\widehat{\kappa};X_{1:n}^\varepsilon  ) -   \widehat{\mu }_W(\mu ;X_{1:n}^\varepsilon )  \rvert }_{\hypertarget{term:1}{\text{\MakeUppercase{\romannumeral 1}}}}+ \underbrace{\lvert\widehat{\mu }_W(\mu ;X_{1:n}^\varepsilon ) - \widehat{\mu }(\mu ;X_{1:n}^\varepsilon )  \rvert}_{\hypertarget{term:2}{\text{\MakeUppercase{\romannumeral 2}}}} + \underbrace{\lvert \widehat{\mu }(\mu ; X_{1:n}^\varepsilon ) -\mu  \rvert }_{\hypertarget{term:3}{\text{\MakeUppercase{\romannumeral 3}}}}.
  \end{align*} 
  Then, we bound \hyperlink{1}{\text{\MakeUppercase{\romannumeral 1}}} by applying \eqref{eq:dependent_bound_1} and using the bounds on \( m(\widehat{\alpha }(\mu ) ) \) and \( S_\mu \wedge S_{\widehat{\kappa } }  \) provided by \Cref{lemma:bias_order} combined with \Cref{lemma:alphaorder} (holding on an event set \( A \)) and \Cref{lemma:rho_bounds_uniform} (holding on an event set \( B \)), respectively, yielding, on the event \( A \cap  B \):
  \begin{align}
    \lvert \widehat{\mu }_W(\widehat{\kappa } ) - \widehat{\mu }_W(\mu )   \rvert &\le 2 \cdot\frac{n-S_{\mu}  \wedge S_{\widehat{\kappa } }}{S_{\mu } \wedge S_{\widehat{\kappa } }}(m(\widehat{\alpha }(\mu  ) )+ \lvert \widehat{\kappa } -\mu  \rvert )\nonumber \\
 &\le 2\cdot\frac{\eta +2 + \left\lfloor n \varepsilon  \right\rfloor}{n-\eta -2-\left\lfloor \varepsilon n \right\rfloor}\left(\inf_{q \in (1,p]} c_w \nu_q \left(\frac{\underline{c} }{n}\ln\frac{4}{\delta}+\frac{\xi}{7}\varepsilon \right)^{-\frac{1}{q}} + \lvert \widehat{\kappa }-\mu   \rvert \right)\nonumber\\
 & \le 2 \cdot \frac{3 \ln \frac{4}{\delta } + (2+\xi )\varepsilon n}{n(1-c_0)}\left(\inf_{q \in (1,p]} c_w \nu_q \left(\frac{\underline{c} }{n}\ln\frac{4}{\delta}+\frac{\xi}{7}\varepsilon \right)^{-\frac{1}{q}} + \lvert \widehat{\kappa }-\mu   \rvert \right)\label{eq:renormalized_sum}\\ 
 & \le C^0_{w,\xi }\cdot\inf_{q \in (1,p]}\nu_q \left(\frac{1}{n}\ln\frac{4}{\delta}+\varepsilon \right)^{1-\frac{1}{q}}  +  \lvert \widehat{\kappa }-\mu   \rvert \cdot \frac{6 \ln \frac{4}{\delta } + (4+2\xi )\varepsilon n}{n(1-c_0)} \nonumber,
  \end{align} 
  where \( P(B)=1 \) and \( A \) is contained in the event in which \hyperlink{3}{\MakeUppercase{\romannumeral 3}} is bounded. Inequality \eqref{eq:renormalized_sum} follows from the definition of \( \eta  \) and the assumption that \( 2\cdot(1+\xi)\eps + \frac{\overline{c}}{n}\ln \frac{4}{\delta } \le c_0< 1 \).
  The second term \hyperlink{term:2}{\MakeUppercase{\romannumeral 2}} is bounded by applying \Cref{prop:normalized_shrinkage} and using the same bounds as before for \( m(\widehat{\alpha }(\mu ) ) \) and \( S_\mu \wedge S_{\widehat{\kappa } }  \) (on the same event sets \( A \) and \( B \)), yielding \[
    \lvert \widehat{\mu }_W(\mu ;X_{1:n}^\varepsilon ) - \widehat{\mu }(\mu ;X_{1:n}^\varepsilon )   \rvert \le  m(\widehat{\alpha }(\mu ;X_{1:n} ) )\frac{n-S_\mu }{S_\mu } \le  C^1_{w,\xi }\cdot\inf_{q \in (1,p]}\nu_q \left(\frac{1}{n}\ln\frac{4}{\delta}+\varepsilon \right)^{1-\frac{1}{q}}  .
  \] 
 
  Finally, \hyperlink{term:3}{\MakeUppercase{\romannumeral 3}} is bounded by applying \Cref{thm:shrinkage_concentration_adv} taking the base estimator to be the population mean \( \mu  \). The probability of this event is at least \( 1- \frac{3\delta}{4} \) instead of \( 1-\delta  \), since \( \lvert \mu -\mu  \rvert=0   \) with probability \( 1 \). Putting everything together yields the desired result.
\end{proof}
\begin{remark}\label[remark]{remark:worse_kappa}
  The dependence on \( R_{\widehat{\kappa } }\left(\frac{\delta}{4}\right) \) induced by this bound is considerably diminished when compared to the bound in \Cref{thm:shrinkage_concentration_adv}; rate-optimal concentration bounds are attained even if \( R_{\widehat{\kappa } }\left(\frac{\delta}{4}\right) = \Theta \left(\frac{1}{n} \ln \frac{4}{\delta }\right)^{-\frac{1}{2 \wedge p}} \), which grows as \( \frac{1}{n} \ln \frac{4}{\delta } \to 0 \) (that is, \( \delta \gg e^{-n} \)). In particular, the result of \Cref{prop:exact_tm} is enough to attain sub-Gaussian bounds for the exact recovery of the trimmed mean estimator.
\end{remark}
\begin{remark}\label[remark]{remark:iteration}
  One can further suppress the impact of base estimator selection by iterating this procedure: the concentration bounds are attained under an event which depends only on the concentration of \( \widehat{\mu }(\mu )  \) and \( \widehat{\alpha }(\mu )  \), independently of the base estimator. This means that, on this event, if we have that the term multiplying \( \lvert \widehat{\kappa } -\mu  \rvert  \), given by  \[
  \frac{6 \ln \frac{4}{\delta } + (4+2\xi )\varepsilon n}{n(1-c_0)},
  \] 
  is smaller than \( r<1 \), then we can iterate as many times as necessary to make the contribution of \( \lvert \widehat{\kappa } - \mu   \rvert   \) arbitrarily small, worsening the constant that multiplies the other terms by an additional multiplicative factor of up to \( \frac{1}{1-r} \).
\end{remark}
\section{Proof of \Cref{thm:multivariate_concentration_adv}}\label[appendix]{sec:proof_multivariate}
The proof of \Cref{thm:multivariate_concentration_adv} follows essentially the same steps as the proof of \Cref{thm:trimmed_concentration_adv}, with the main difference residing in the proof of an analogue of \Cref{lemma:bias_order}. We state all the multivariate versions of previous results and restrict ourselves to proving this analogue of \Cref{lemma:bias_order}, which is the only part of the proof that requires a non-trivial modification. The other results are stated without proof, since their proofs are mostly identical to the univariate case, with the absolute value replaced by the Euclidean norm.
\begin{theorem}[Bias-variance decomposition]\label{theorem:bias_variance_multivariate} Under \Cref{assum:malpha_finite,assum:kappa_independent}, let \( A_{\widehat{\kappa } } \) be a deterministic set such that \( \widehat{\kappa } \in A_{\widehat{\kappa } }  \) with probability at least \( 1-\frac{\delta}{4} \) and $I_\alpha (\kappa )$ be a family of sets indexed by \( \kappa \in A_{\widehat{\kappa } } \) such that, for every \( \kappa \in A_{\widehat{\kappa } } \), 
	\begin{equation}
		\label{eq:alpha_in_alphak_whp_multivariate}
    \Pr\left[\widehat{\alpha }(\kappa )\in I_\alpha (\kappa ) \right] \geq 1 - \frac{\delta}{2} .
	\end{equation}
  Define $\sF_\kappa  = \left\{ x \mapsto \langle v, \kappa + (x-\kappa)w(\alpha \lVert x-\kappa\rVert)-\mu \rangle : v \in \mathbb{S}^{d-1},\alpha  \in I_\alpha (\kappa )\right\}$.
	Then, with probability at least $1-\delta$, \( \lVert \widehat{\mu } -\mu  \rVert \) is at most
  \begin{equation}\label{eq:bias_var_bound_multivariate}
     \sup_{\kappa \in A_{\widehat{\kappa }}(\delta ) } \left\{ \sup_{\alpha \in I_\alpha (\kappa )} \lVert b(\alpha ,\kappa )\rVert + 2\Emp(\sF_\kappa )+\sqrt{\frac{2\sigma^2}{n} \ln \frac{4}{\delta } }+ \frac{8m(\underline{\alpha }(\kappa ))}{3n}\ln \frac{4}{\delta }\right\},
  \end{equation} 
  where \( \Emp(\sF)= \mathbb{E}\left[\sup_{f \in \sF} \frac{1}{n}\sum_{i=1}^n f(X_i)-\mathbb{E}f(X_i)\right]\), \( \sigma^2:=\sup_{f \in \sF_\kappa}\mathbb{V}f(X) \), \( m(\alpha ):=\sup_{t \ge 0}  t w(\alpha t)   \), and \( b(\alpha,\kappa ):=\kappa -\mu +\mathbb{E}\left[(X-\kappa )w(\alpha \lVert X -\kappa \rVert)\right] \).
\end{theorem}
\begin{remark}
In this multivariate version of the bias-variance decomposition, we keep the expectation of the supremum of the empirical process \( \Emp(\sF_\kappa ) \) obtained from Bousquet's version of Talagrand's inequality instead of using a VC-dimension-based bound, which will help in yielding tighter bounds.
\end{remark}

\begin{lemma}\label[lemma]{lemma:alphaorder_multivariate}
	Under \Cref{assum:abs_continuous}, there exist absolute constants \( \underline{c}, \overline{c}>0 \) such that \( \eta =\ln \frac{4}{\delta }+(1+\xi )\eps n \) and $\underline{\alpha}(\kappa) < \overline{\alpha}(\kappa)$ are implicitly defined by 
	\[ \E w(\underline{\alpha}(\kappa) \lVert X-\kappa\rVert) = 1 - \frac{\underline{c}\ln \frac{4}{\delta }}{n}-\frac{\xi}{7}\varepsilon  \quad \text{ and } \quad  \E w(\overline{\alpha}(\kappa) \lVert X-\kappa\rVert) = 1 - \frac{\overline{c}\ln \frac{4}{\delta }}{n}-(1+2\xi)\eps  \]
  for every \( \kappa \in A_{\widehat{\kappa}} \) and \eqref{eq:alpha_in_alphak_whp_multivariate} holds with $I_\alpha(\kappa) = (\underline{\alpha}(\kappa), \overline{\alpha}(\kappa))$ whenever \( \frac{\overline{c} \ln \frac{4}{\delta }}{n}+(1+2\xi )\eps < 1 \). Moreover, for all \( \kappa \in A_{\widehat{\kappa } } \), \begin{equation}\label{eq:alpha_contamination_multivariate}
    \Pr\left[\underline{\alpha }(\kappa )<\inf_{X_{1:n}^\varepsilon \in \sA(X_{1:n},\varepsilon ) }\widehat{\alpha}^\eps (\kappa )\le \widehat{\alpha}(\kappa )< \overline{\alpha}(\kappa )\right]\geq 1-\frac{\delta}{2}  
  \end{equation}
\end{lemma}
\begin{remark}
  In the proof of the univariate version of this lemma, we used \Cref{lemma:rho_bounds} to control the total sum of the weights. To that end, we used the fact that the level sets of \( \lvert X_i - \kappa  \rvert \) have zero measure. In the multivariate case, we ask for absolute continuity of \( P \) with respect to the Lebesgue measure in order to ensure that the level sets of \( \lVert X_i - \kappa  \rVert  \) have zero measure, which allows us to use the same argument as in the univariate case to achieve the desired control of the total sum of weights.
\end{remark}
\begin{proposition}\label{prop:bias_order_multivariate}
Assume that \( m(\alpha )< \infty \) for all \( \alpha >0 \) and that \( w(t)\ge (1-t^2)_+ \).   
	\begin{gather*}
    \sup_{\alpha \in I_\alpha (\kappa )}\lVert b(\alpha,\kappa )\rVert\le \sqrt{\lVert \Sigma \rVert \left(\frac{\overline{c}}{n} \ln \frac{4}{\delta } + \xi^{-1}\varepsilon \right) } + \lVert \kappa -\mu \rVert \left(\frac{\overline{c}}{n} \ln \frac{4}{\delta } + \xi^{-1}\varepsilon \right) , \\
		m(\underline{\alpha}(\kappa))\leq c_w (\sqrt{ \tr(\Sigma)}+\lVert \kappa-\mu\rVert)\left(\frac{\underline{c} }{n}\ln\frac{4}{\delta } + \frac{\xi}{7 }\varepsilon \right)^{-\frac{1}{2}},\\
    \mathbb{E}\left[\sup_{f \in \sF_\kappa } \frac{1}{n}\sum_{i=1}^n f(X_i)-\mathbb{E}\left[f(X_i)\right]\right] \le 8\sqrt{ \frac{ \tr(\Sigma)}{n}}+4\lVert \kappa -\mu \rVert \cdot \left(\frac{\overline{c}}{n} \ln \frac{4}{\delta } + \xi^{-1}\varepsilon \right), \\
    \sigma^2 \le 2\lVert \Sigma \rVert + 2\lVert \kappa -\mu \rVert^2 \left(\frac{\overline{c}}{n} \ln \frac{4}{\delta } + \xi^{-1}\varepsilon \right),
	\end{gather*}
  where \( c_w=\sup_{t \ge 0} tw(t) \). Here \( \lVert \Sigma\rVert \) represents the spectral norm of \( \Sigma  \) and \( \tr(\Sigma) \) is the trace of \( \Sigma \). 
\end{proposition}
In order to prove this result, we first state and prove the following proposition, which helps us bound \( \Emp(\sF_\kappa ) \) and may be of independent interest: 

\begin{proposition}\label{prop:shrinkage_emp}
  Let \( X,X_1,\ldots,X_n \) be i.i.d. \( E \)-valued, centered random variables, where \( E \) is a Banach space. Let \( \sH \subseteq \{g\circ r\mid g\colon \R \to [0,1] \text{ non-increasing}\}   \), for some function \( r\colon E \to \R \), be a countable class of functions. Then \[
    \mathbb{E}\sup_{w \in \sH}\left\|\sum_{i=1}^n X_i w(X_i) - \mathbb{E}\left[Xw(X)\right]\right\| \le 8\mathbb{E}\left\|\sum_{i=1}^n X_i\right\|.
  \] 
\end{proposition}
\begin{proof}
Let \( B \) be the dual unit ball of \( E \). Then, \[
  \mathbb{E}\sup_{w \in \sH}\left\|\sum_{i=1}^n X_i w(X_i) - \mathbb{E}\left[Xw(X)\right]\right\| =    \mathbb{E}\sup_{\underset{w \in \sH}{v \in B}}\sum_{i=1}^n \langle v, X_i w(X_i) - \mathbb{E}\left[Xw(X)\right]\rangle. 
\]   
By symmetrization, we can bound the last expression by \[
  2\mathbb{E}\sup_{\underset{w \in \sH}{v \in B}}\sum_{i=1}^n \xi_i \langle v, X_i\rangle w(X_i),
\]
where \( \xi_i \) are indepedent Rademacher random variables independent of \( X_i \). Let \( \pi \) be the random permutation depending on \( X_1,\ldots,X_n \) such that \( r(X_{\pi(1)})\le \ldots \le r(X_{\pi(n)}) \). Then, in particular, \( w(X_{\pi(1)}) \ge \ldots \ge  w(X_{\pi(n)}) \) for every \( w \in \sH \). Denote \( S_k = \sum_{i=1}^k \xi_{\pi(i)} \langle v, X_{\pi(i)} \rangle \). We can write \[
  \sum_{i=1}^n \xi_{\pi(i)} \langle v, X_{\pi(i)}\rangle w(X_{\pi(i)}) = w(X_{\pi(n)})S_n +  \sum_{i=1}^{n-1} (w(X_{\pi(i)}) - w(X_{\pi(i+1)}))S_i. 
\] 
Notice that \( w(X_{\pi(n)}),w(X_{\pi(i)})-w(X_{\pi(i+1)}) \in [0,1] \) and that \( w(X_{\pi(n)})+\sum_{i=1}^{n-1} w(X_{\pi(i)})-w(X_{\pi(i+1)})=w(X_{\pi(1)}) \le 1\). Thus, the sum \( \sum_{i=1}^n \xi_i \langle v, X_{\pi(i)} \rangle w(X_{\pi(i)}) \) is a convex combination of the partial sums \( S_k \) (and possibly zero). Therefore, \[
  \sup_{\underset{w \in \sH}{v \in B}}\sum_{i=1}^n \xi_i \langle v, X_i\rangle w(X_i) \le \sup_{v \in B}\max_{1\le k \le n} S_k=\max_{1\le k \le n}\left\lVert \sum_{i=1}^k \xi_{\pi(i)} X_{\pi(i)}\right\rVert. 
\]
We now bound, conditioned on the variables \( X_1,\ldots,X_n \), \[
  \mathbb{E}_{\xi }\max_{1\le k \le n}\left\|\sum_{i=1}^k\xi_{\pi(i)} X_{\pi(i)}\right\|=\mathbb{E}_{\xi }\max_{1\le k \le n}\left\|\sum_{i=1}^k\xi_i X_{\pi(i)}\right\| \le 2 \mathbb{E}_\xi \left\|\sum_{i=1}^n \xi_iX_{\pi(i)}\right\|=2 \mathbb{E}_\xi \left\|\sum_{i=1}^n \xi_iX_i\right\|,
\] 
where both equalities hold by independence of \( \xi_1,\ldots,\xi_n  \) from \( X_1,\ldots,X_n \) and the inequality is a consequence of Levy's theorem for Banach spaces. The proof is concluded by de-symmetrization.
\end{proof}
\begin{proof}
  Similarly to \Cref{lemma:bias_order}, we decompose the bias as
		\begin{align}
			\lVert b(\alpha,\kappa )\rVert & =  \lVert \mathbb{E}[(X-\kappa)w(\alpha \lVert X-\kappa\rVert)]+\kappa-\mu\rVert\nonumber \\
  & \leq \lVert\mathbb{E}[(X-\mu)w(\alpha \lVert X-\kappa\rVert)]\rVert\label{eq:est_bias_multivariate} \\
  &+ \lVert (\kappa-\mu)[1-\mathbb{E} w(\alpha \lVert X-\kappa\rVert)]\rVert\label{eq:kappa_bias_multivariate}.
		\end{align}

		To bound \eqref{eq:est_bias_multivariate}, note that
    $\mathbb{E}[(X-\mu)w(\alpha \lVert X-\kappa\rVert)] = - \mathbb{E}[(X-\mu)(1-w(\alpha \lVert X-\kappa\rVert))]$ and, by the dual definition of the norm and linearity of the inner product, we rewrite \eqref{eq:est_bias_multivariate} as \begin{align*}
      \sup_{v \in \mathbb{S}^{d-1}} \mathbb{E}[\langle v, X-\mu \rangle &(1-w(\alpha \lVert X-\kappa \rVert))] \\&\le \sup_{v \in \mathbb{S}^{d-1}} \sqrt{\mathbb{E}\left[\langle v,X-\mu \rangle^2\right]} \sqrt{\mathbb{E}\left[(1-w(\alpha \lVert X-\kappa \rVert))^2\right]}.\\ 
                                                                                                             & \le \sup_{v \in \mathbb{S}^{d-1}}\sqrt{\langle v, \Sigma v \rangle} \sqrt{\mathbb{E}\left[1-w(\alpha \lVert X-\kappa \rVert)\right]}\\ 
                                                                                                             &= \sqrt{\lVert \Sigma \rVert \mathbb{E}\left[1-w(\alpha \lVert X-\kappa \rVert)\right]}. 
    \end{align*} 
    The bound follows by taking the supremum over \( \alpha  \) and employing the definition of \( \overline{\alpha }(\kappa ) \). 
		Next, we bound $m(\underline{\alpha}(\kappa))$. To that end, we lower bound $\underline{\alpha}(\kappa)$ since
		\begin{equation}\label{eq:m_xp_multivariate}
			m(\alpha)=\sup_{t\ge 0}tw(\alpha t)=\alpha^{-1}\sup_{t\geq0}tw(t)=c_w\alpha^{-1}.
		\end{equation}
    Recall that, by definition, $\mathbb{E}[w(\underline{\alpha}(\kappa)\lVert X-\kappa\rVert)]=1-\frac{\underline{c}}{n}\ln\frac{4}{\delta } - \frac{\xi }{7}\varepsilon $. Then the property $w(t)\geq(1-t^2)_+\geq1-t^2$ implies $1-\frac{\underline{c}}{n}\ln\frac{4}{\delta } -\frac{\xi }{7}\varepsilon \geq 1-\underline{\alpha}(\kappa)^2\mathbb{E}[\lVert X-\kappa\rVert^2]$ and so
		\begin{equation*}
			\underline{\alpha}(\kappa)\geq\left(\frac{\underline{c}}{n}\ln\frac{4}{\delta }+\frac{\xi }{7}\varepsilon  \right)^{\frac{1}{2}}\frac{1}{\mathbb{E}[\lVert X-\kappa\rVert^2]^{\frac{1}{2}}}\geq\left(\frac{\underline{c} }{n}\ln\frac{4}{\delta } + \frac{\xi }{7}\varepsilon  \right)^{\frac{1}{2}}\frac{1}{\sqrt{ \tr(\Sigma)}+\lVert \kappa-\mu\rVert}.
		\end{equation*}
		By \eqref{eq:m_xp_multivariate},
		\begin{equation}\label{eq:malpha_bound_multivariate}
			m(\underline{\alpha}(\kappa))\leq c_w (\sqrt{ \tr(\Sigma)}+\lVert \kappa-\mu\rVert)\left(\frac{\underline{c} }{n}\ln\frac{4}{\delta }+ \frac{\xi }{7}\varepsilon \right)^{-\frac{1}{2}},
		\end{equation}
    where we used that \( \mathbb{E}\left[\lVert X-\mu  \rVert^2\right]=\tr(\Sigma)\). This is proved by taking the orthonormal eigenvectors of \( \Sigma \), denoted as \( v_i \), and writing \begin{align*}
      \mathbb{E}\left[\sum \langle X- \mu , v_i \rangle^2\right]=\sum \langle v_i, \Sigma v_i \rangle = \tr(\Sigma). 
    \end{align*} 

    In order to prove the bound on \[
\Emp(\sF_\kappa ) = \mathbb{E}\left[\sup_{f \in \sF_\kappa } \frac{1}{n}\sum_{i=1}^n f(X_i)-\mathbb{E}\left[f(X_i)\right]\right],
    \] 
firstly notice that \begin{align}
  \Emp(\sF_{\kappa })&\le \mathbb{E}\left[\sup_{\alpha \in I_\alpha(\kappa )} \frac{1}{n} \left\lVert\sum_{i=1}^n (X_i-\mu  )w(\alpha \|X_i - \kappa \|) - \mathbb{E}\left[(X-\mu )w(\alpha \|X-\kappa \|)\right]\right\rVert\right] \label{eq:emp_bound_a2}\\
                     &+ \mathbb{E}\left[\sup_{\alpha \in I_\alpha (\kappa )} \frac{1}{n} \left\lVert\sum_{i=1}^n(\kappa -\mu )(w(\alpha\lVert X_i -\kappa \rVert) - \mathbb{E}\left[w(\alpha \lVert X_i-\kappa \rVert)\right])\right\rVert\right].\label{eq:emp_bound_b2}
\end{align}  
On the one hand, we can bound \eqref{eq:emp_bound_a2} by applying \Cref{prop:shrinkage_emp} with \( X_i \) replaced by \( X_i - \mu  \) and \( \sH = \{ x \mapsto w(\alpha \|x +\mu -\kappa \|)\colon \alpha \in I_\alpha (\kappa )\}  \). This gives us \[
  \eqref{eq:emp_bound_a2} \le \frac{8}{n}\mathbb{E}\left\|\sum_{i=1}^n (X_i -\mu )\right\| \le 8\sqrt{\frac{\tr(\Sigma)}{n}}.
\]
To prove the second bound, we apply Jensen's inequality and write \[
  \mathbb{E}\left\lVert \sum_{i=1}^n (X_i - \mu )\right\rVert^2=\sum_{i=1}^n \sum_{i=1}^n \mathbb{E}\left[\langle X_i -\mu , X_j-\mu \rangle\right]=\sum_{i=1}^n \mathbb{E}\left[\lVert X_i -\mu \rVert^2\right]= n \cdot \tr(\Sigma ).
\] 
On the other hand, in \eqref{eq:emp_bound_b2} we do not want to simply remove the weights, since they reduce the dependence on the error of \( \kappa  \). Instead, we use symmetrization and de-symmetrization to substitute the term \( \mathbb{E}w(\alpha \|X_i -\kappa \|) \) by \( 1 \). This gives us \begin{align*}
  \eqref{eq:emp_bound_b2} &\le 4\lVert \kappa -\mu  \rVert \mathbb{E}\left[ \sup_{\alpha \in I_\alpha(\kappa )}\frac{1}{n} \left\lvert \sum_{i=1}^n 1-w(\alpha \lVert X_i -\kappa \rVert) \right\rvert \right]\\
                          &\le 4\lVert \kappa -\mu  \rVert \mathbb{E}\left[ \frac{1}{n} \left\lvert \sum_{i=1}^n 1-w(\overline{\alpha} \lVert X_i -\kappa \rVert) \right\rvert \right] = 4\lVert \kappa -\mu \rVert \cdot \left(\frac{\overline{c}}{n}\ln \frac{4}{\delta } + \xi^{-1} \varepsilon \right).
\end{align*}

Finally, we bound the variance term \( \sigma^2 \).
  \begin{align*}
    \sigma^2 &\le \sup_{\underset{\alpha \in I_{\alpha }(\kappa )}{v \in \mathbb{S}^{d-1}}}\mathbb{E}\left[\langle v, (X -\kappa )w(\alpha \| X -\kappa \|) + \kappa -\mu  \rangle^2 \right]\\
             &\le\sup_{\underset{\alpha \in I_{\alpha }(\kappa )}{v \in \mathbb{S}^{d-1}}}2\mathbb{E}\left[\langle v, (X -\mu  )w(\alpha \| X -\kappa \|) \rangle^2+ \langle v,\kappa -\mu  \rangle^2(1-w(\alpha \|X-\kappa \|))^2 \right]  \\&\le 2\lVert \Sigma\rVert + 2\|\kappa -\mu \|^2\left(\frac{\overline{c}}{n}\ln\frac{1}{\delta }+ \xi^{-1} \varepsilon \right).
  \end{align*} 
In the first inequality we used that the mean minimizes the square risk.
\end{proof}
As before, these last results help bound the error of the uncontaminated estimator \( \widehat{\mu }(\widehat{\kappa } )  \). In order to translate this into control of the error of \( \widehat{\mu }^\varepsilon (\widehat{\kappa } )  \), we state the following lemma:  

\begin{lemma}[Bounding the error introduced by contamination]\label[lemma]{lemma:contamination_error_multivariate} Let any \( \kappa \in \R , \eta \in (0,n), \varepsilon \in [0,1/2) \). Consider any sample \( X_{1:n} \) and \( X_{1:n}^\varepsilon \in \sA(X_{1:n},\varepsilon ) \). Then, \( \lVert \widehat{\mu}^\varepsilon(\kappa ) - \widehat{\mu}(\kappa )   \rVert \) is at most
\[
   \left(\varepsilon+\frac{2n-\sum_{i=1}^n w(\widehat{\alpha }^\varepsilon (\kappa )\lVert X_i^\varepsilon -\kappa  \rVert  )-\sum_{i=1}^n w(\widehat{\alpha }(\kappa ) \lVert X_i -\kappa  \rVert  )}{n}\right) \left(m(\widehat{\alpha}(\kappa  ) )+m(\widehat{\alpha }^\varepsilon(\kappa  )  )\right) .
\]  
\end{lemma}
The proof of \Cref{thm:multivariate_concentration_adv} is concluded by combining these last results, analogously to the proof of \Cref{thm:shrinkage_concentration_adv}.
\section{Additional experimental results}
\subsection{Shrinkage functions that violate \Cref{assum:malpha_finite,assum:rho_bound}}

This subsection is dedicated to evaluating the performance of shrinkage estimators that employ shrinkage functions \( w \) that violate either \Cref{assum:malpha_finite} or \Cref{assum:rho_bound}, or both. Specifically, we consider \( w(t)=\frac{1}{\ln(e+t^2)} \), which violates only \Cref{assum:malpha_finite}, \( w(t)=1-\sqrt{1-(1-t)_+^2} \), which violates only \Cref{assum:rho_bound}, as well as \( w (t)=\frac{1}{\ln(e+t)} \) and \( w(t)=\frac{1}{1+\sqrt{t} } \), which violate both assumptions. 

The computational experiments are similar to the one whos results are presented in \Cref{tab:is_shrinkage_good}, where we add these violating estimators to the previously considered ones. The results, presented in \Cref{tab:violating_shrinkage}, indicate that, while these violating shrinkage estimators still often improve upon the base estimators, they tend to perform worse than the ones that satisfy both assumptions. In particular, the estimators that have derivative \( -\infty \) at \( 0 \), given by \( w(t)=(1+\sqrt{t})^{-1} \) and \( w(t)=1-\sqrt{1-(1-t)_+^2}  \), often perform the worst suggesting that, at least in the considered scenarios, \Cref{assum:rho_bound}, which controls the decay close to the origin, is more relevant than \Cref{assum:malpha_finite}, which controls the asymptotic decay.
\begin{table}[t]
\tiny
\begin{tabular}{l@{\extracolsep{\fill}}|cccc|cccc|cccc|cccc}
& \multicolumn{4}{c|}{$\overline{X}$}& \multicolumn{4}{c|}{$M$}& \multicolumn{4}{c|}{TM}& \multicolumn{4}{c}{MoM}
\\
& N& SN& T& ST& N& SN& T& ST& N& SN& T& ST& N& SN& T& ST
\\
\hline
$1\wedge t^{-1}$ & \(1\) & \(\mathbf{12}\) & \(-38\) & \(\mathbf{1}\) & \(-20\) & \(\mathbf{-67}\) & \(45\) & \(\mathbf{-62}\) & \(1\) & \(\mathbf{7}\) & \(-3\) & \(\mathbf{0}\) & \(-13\) & \(\mathbf{-5}\) & \(-36\) & \(\mathbf{3}\) \\
$(1+t^p)^{-1}$ & \(1\) & \(\mathbf{3}\) & \(-31\) & \(-10\) & \(-20\) & \(\mathbf{-69}\) & \(60\) & \(-65\) & \(2\) & \(\mathbf{-1}\) & \(8\) & \(-9\) & \(-12\) & \(\mathbf{-12}\) & \(-29\) & \(-8\) \\
$(\ln(e+t))^{-1}$ & \(2\) & \(2\) & \(-11\) & \(-11\) & \(-19\) & \(-70\) & \(107\) & \(-67\) & \(2\) & \(-3\) & \(40\) & \(-13\) & \(-12\) & \(-12\) & \(-7\) & \(-10\) \\
$(\ln(e+t^2))^{-1}$ & \(1\) & \(3\) & \(-28\) & \(-11\) & \(-20\) & \(-69\) & \(66\) & \(-66\) & \(2\) & \(-1\) & \(12\) & \(-11\) & \(-12\) & \(-12\) & \(-25\) & \(-10\) \\
$1-\sqrt{1-(1-t)_+^2}$ & \(\mathbf{2}\) & \(2\) & \(\mathbf{-2}\) & \(-3\) & \(\mathbf{-19}\) & \(-70\) & \(\mathbf{129}\) & \(-64\) & \(\mathbf{2}\) & \(-3\) & \(\mathbf{54}\) & \(\mathbf{-5}\) & \(\mathbf{-12}\) & \(-12\) & \(\mathbf{2}\) & \(-2\) \\
$(1+\sqrt{t})^{-1}$ & \(\mathbf{2}\) & \(2\) & \(\mathbf{-2}\) & \(\mathbf{-3}\) & \(\mathbf{-19}\) & \(-70\) & \(\mathbf{129}\) & \(\mathbf{-64}\) & \(\mathbf{2}\) & \(-3\) & \(\mathbf{54}\) & \(-5\) & \(\mathbf{-12}\) & \(-12\) & \(\mathbf{2}\) & \(\mathbf{-1}\) \\
\hline
\end{tabular}
\caption{Relative error of the shrinkage estimator for a given shrinkage function (rows) with respect the corresponding base estimator and distribution (columns). The shrinkage functions that have derivative \( -\infty \) at \( 0 \) (thus violating \Cref{assum:rho_bound}) display the worst performances across different base estimators and distributions.} 
\label{tab:violating_shrinkage}
  
\end{table}
\end{appendix}

\clearpage
\paragraph{Acknowledgments.}
 The authors thank Roberto I. Oliveira, Guillaume Lecué, and Arnak Dalalyan for their comments on previous versions of the manuscript.

\paragraph{Funding.}
A.C. was supported by a “FAPERJ Nota 10" grant (SEI-260003/004731/2025) from Fundação de Amparo à Pesquisa do Estado do Rio de Janeiro (FAPERJ). L.R. was partially supported by CNPq. P.O. was supported by grant SEI-260003/001545/2022 from FAPERJ.
\bibliographystyle{unsrtnat}
\bibliography{bibliography} %

\end{document}